\documentclass[12pt]{amsart}
\usepackage{graphicx}
\usepackage{amsfonts,amsthm,amsmath,amssymb,latexsym}
\usepackage{graphicx,color,epstopdf,xcolor}
\usepackage[all]{xy}
\usepackage{wrapfig}
\usepackage{amssymb,amsthm,amsmath,latexsym,epsfig, amscd, mathrsfs}
\usepackage{hyperref}
\hypersetup{colorlinks,linkcolor={blue},citecolor={red}} 
\numberwithin{figure}{section}
\numberwithin{equation}{section}
\newcommand{\eps}{\varepsilon}

\newcommand{\calS}{\mathcal{S}}
\newcommand{\calH}{\mathcal{H}}
\newcommand{\calB}{\mathcal{B}}
\newcommand{\calO}{\mathcal{O}}
\newcommand{\calR}{\mathcal{R}}

\newcommand{\la}{\lambda}
\newcommand{\s}{\sigma}
\newcommand{\G}{\Gamma}
\newcommand{\g}{\gamma}
\newcommand{\D}{\Delta}
\renewcommand{\L}{\mathcal{L}}
\renewcommand{\S}{\mathscr{S}}
\newcommand{\M}{\mathscr{M}}

\def\Xint#1{\mathchoice
{\XXint\displaystyle\textstyle{#1}}%
{\XXint\textstyle\scriptstyle{#1}}%
{\XXint\scriptstyle\scriptscriptstyle{#1}}%
{\XXint\scriptscriptstyle\scriptscriptstyle{#1}}%
\!\int}
\def\XXint#1#2#3{{\setbox0=\hbox{$#1{#2#3}{\int}$}
\vcenter{\hbox{$#2#3$}}\kern-.5\wd0}}

\def\dashint{\Xint-}
\newcommand{\diam}{\mathrm{diam}}
\renewcommand{\d}{\delta}

\newcommand{\dist}{\mathrm{dist}}
\newcommand{\m}{\mathrm{mod}}
\begin{document}


\title[QS co-Hopfian Menger curves \& Sierpi\'nski spaces  ]{Quasisymmetrically co-Hopfian \\Menger Curves and Sierpi\'nski spaces}
\address{Department of Mathematics, Kansas State University, Manhattan, KS, 66506-2602}



\author{Hrant Hakobyan}


              \email{hakobyan@math.ksu.edu}           


\maketitle

\setcounter{tocdepth}{1}
\newtheorem{theorem}{Theorem}[section]
\newtheorem{lemma}[theorem]{Lemma}
\newtheorem{corollary}[theorem]{Corollary}
\newtheorem{problem}[theorem]{Problem}
\newtheorem{conjecture}[theorem]{Conjecture}
\newtheorem{question}[theorem]{Question}
\newtheorem*{question*}{Question}
\theoremstyle{definition}
\newtheorem{definition}[theorem]{Definition}
\newtheorem{example}[theorem]{Example}
\newtheorem{xca}[theorem]{Exercise}
\renewcommand{\thefootnote}{\alph{footnote}}
\theoremstyle{remark}
\newtheorem{remark}[theorem]{Remark}


%

\begin{abstract}
A metric space $X$ is quasisymmetrically co-Hopfian if every quasisymmetric embedding of $X$ into itself is onto.
We construct the first examples of metric spaces homeomorphic to the universal Menger curve and higher dimensional Sierpi\'nski spaces, which are quasisymmetrically co-Hopfian. We also show that the collection of quasisymmetric equivalence classes of spaces homeomorphic to the Menger curve is uncountable.
These results answer a problem and generalize results of Merenkov from \cite{Mer:coHopf}.
\end{abstract}

\tableofcontents

\newpage

\section{Introduction}


\subsection{QS co-Hopfian Menger curves}

In recent years quasiconformal geometry of fractal spaces has been investigated extensively, see for instance \cite{Bonk:ICM,Bonk:SCUnif,BonkMerenkov,BKM,BLM,Kleiner:ICM,MTW,Mer:coHopf}. Much of this interest is rooted in questions arising in geometric group theory and Mostow type rigidity results, cf. \cite{Bonk:ICM,Kleiner:ICM}. In particular, motivated by questions in geometry of Gromov hyperbolic groups, Merenkov \cite{Mer:coHopf}
recently studied metric spaces having a co-Hopfian property. A metric space $X$ is said to be \textit{quasisymmetrically (QS) co-Hopfian}
if every quasisymmetric embedding of $X$ into itself is onto. If a metric space $X$ satisfies the stronger property that every continuous one-to-one map of $X$ into itself is onto (e.g. finite sets, $\mathbb{S}^n, n\geq 1$, etc.), $X$ is \emph{topologically co-Hopfian}. Classical fractals, such as the Sierpi\'nski carpet and the Menger curve, cf. Fig. \ref{fig:carpet&sponge}, are self similar spaces and therefore are neither topologically nor QS co-Hopfian.

\begin{figure}[h]
\centerline{
\includegraphics[width=0.6\textwidth]{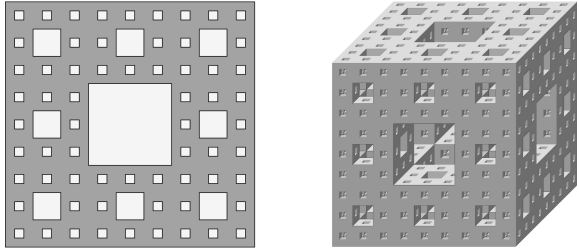}}
\caption{Sierpi\'nski carpet (left) and Menger curve (right).}\label{fig:carpet&sponge}
\end{figure}

Until recently no examples were known of compact metric spaces which were QS co-Hopfian but not topologically co-Hopfian. In \cite{Mer:coHopf} Merenkov constructed the first such example by showing that there is a metric space homeomorphic to the standard Sierpi\'nski carpet $\mathscr{S}_1\subset\mathbb{R}^2$ that is QS co-Hopfian. In the same paper Merenkov asked if there is a QS co-Hopfian metric space that is homeomorphic to the Menger curve. We answer this affirmatively.
\begin{theorem}\label{thm:menger-co-Hopf}
There is a metric space homeomorphic to the Menger curve which is QS co-Hopfian.
\end{theorem}
The construction of the metric space in Theorem \ref{thm:menger-co-Hopf} 
is given in Section \ref{section:slit-menger-definition}. Theorem \ref{thm:menger-co-Hopf} follows from Theorem \ref{thm:slit-menger-co-Hopf}, which is proved in Section \ref{Section:menger-coHopf-proof}.

The metric space in Theorem \ref{thm:menger-co-Hopf}, which will be denoted by $D\mathfrak{M}$, is a ``double" of a metric space $\mathfrak{M}$, which we will call a \emph{slit Menger curve} and is a self-similar fractal space of Hausdorff dimension $3$ and topological dimension $1$. The proof of Theorem \ref{thm:menger-co-Hopf} is quite different from that in \cite{Mer:coHopf} and uses new topological and analytic techniques. The main topological idea is to construct $\mathfrak{M}$ in such a way that $D\mathfrak{M}$ is ``fibered" over a base Sierpi\'nski carpet $\mathcal{B}$ (of Hausdorff dimension $2$) in a way that almost every fiber is a topological circle, cf. Section \ref{Sec:menger-fibered-over-slit-carpet}.  A QS mapping of $D\mathfrak{M}$ into itself then induces a mapping of the carpet $\mathcal{B}$ into itself and we show that this induced mapping is surjective, cf. Section \ref{Section:menger-coHopf-proof}. This requires a careful analysis of the topology of fibers over the peripheral circles of $\mathcal{B}$ and is the core of the argument.

Geometry of metric spaces homeomorphic to the classical Sierpi\'nski carpet has recently been studies in \cite{Bonk:SCUnif,BonkMerenkov,BKM}. In particular, from the rigidity results of Bonk, Kleiner and Merenkov \cite{BKM} it follows that the collection of quasisymmetric equivalence classes of carpets (as well as of higher dimensional Sierpi\'nski spaces) is uncountable, see e.g. the discussion in \cite[Page ~593]{BonkMerenkov}. We show that a similar result also holds for the Menger curve. 

\begin{theorem}\label{thm:uncountable-menger}
The set of quasisymmetric equivalence classes of metric spaces homeomorphic to the Menger curve is uncountable.
\end{theorem}

To obtain Theorem \ref{thm:uncountable-menger} we will show that the construction of the slit Menger curve from Theorem \ref{thm:menger-co-Hopf} is flexible enough to allow for an uncountable class of slit carpets where quasisymetric rigidity holds, i.e. if two members in the class are quasisymmetrically equivalent then they are isometric. Note, that it was already known that there are countably many Menger curves which are not quasisymetrically equivalent. Indeed, it follows from the work of Bourdon and Pajot \cite{BourdonPajot} that there are countably many Menger curves of distinct conformal dimensions. In our examples however, all the inequivalent Menger curves are of Hausdorff and conformal dimension $3$. 

\subsection{QS co-Hopfian Sierpi\'nski spaces}
Both, Sierpi\'nski carpet and Menger curve have topological dimension $1$. Thus, the following question is quite natural.


\begin{question}
Is there a metric space of topological dimension greater than $1$ which is QS co-Hopfian but not topologically co-Hopfian?
\end{question}

To answer this question it seems quite natural to try to generalize the methods in \cite{Mer:coHopf}. A crucial part of these methods are moduli estimates for curve families in multiply connected slit domains, see Section \ref{Section:slit-spaces}. However the technique in \cite{Mer:coHopf} works only for quite special and symmetric planar domains and uses conformal mappings, thus does not generalize to higher dimensions.

In this paper we develop a new method for estimating moduli of families of curves in multiply connected ``slit" domains for quite general configurations of slits and in all dimensions, see Lemma \ref{lemma:main_estimate}, and its consequences, Lemmas \ref{lemma:curves_in_porous_carpets} and \ref{lemma:curves-in-diadic-carpets}. In particular, it allows us to answer the above question affirmatively.

\begin{theorem}\label{thm:Sierpinski-co-hopf}
For every $n\geq1$ there is a metric space homeomorphic to the standard Sierpi\'nski space of topological dimension $n$ which is quasisymmetrically co-Hopfian.
\end{theorem}

\begin{wrapfigure}{tr}{0.4\textwidth}
  \begin{center}
    \includegraphics[width=0.3\textwidth]{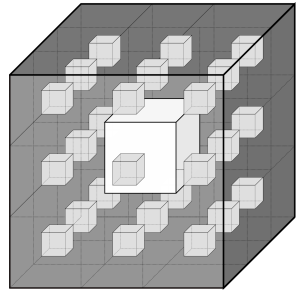}
  \end{center}
  \caption{Sierpi\'nski space \small{$\mathscr{S}_2\subset\mathbb{R}^3$.}}\label{fig:2d-Sierpinski}
\end{wrapfigure}

Here the {standard Sierpi\'nski space} is a compact subset of the Euclidean space constructed as follows. Let $F_0=[0,1]^{n}$. Let $F_1$ be the subset of $F_0$ obtained by dividing it  into $3^{n}$ disjoint, congruent triadic cubes of side-length $1/3$ and removing the middle one. Thus, $F_1$ is a union of $3^{n}-1$ (closed) triadic cubes of generation $1$. Suppose $F_i$ has been defined and is a union of triadic cubes. To define $F_{i+1}$ divide each triadic cube contained in $F_i$ into $3^{n}$ subcubes of generation $(i+1)$ and remove the (open) central subcube. The closed set $ \S_{n-1}=\bigcap_{i=0}^{\infty} F_i \subset \mathbb{R}^n$ is called the \textit{standard $(n-1)$-dimensional Sierpi\'nski set}. Note, that $\S_0$ is the standard middle-thirds Cantor set $C_3\subset\mathbb{R}$, while $\S_1$ is the Sierpi\'nski carpet in the plane, cf. Figures \ref{fig:carpet&sponge} and \ref{fig:2d-Sierpinski} for $\S_1$ and $\S_2$. 
The space $\S_n$ has topological dimension $n$ and in fact every compact subset of $\mathbb{R}^{n+1}$ of topological dimension $n$ can be embedded in $\S_n$, see \cite{Stanko}.

From the definition above we see that we can write $\S_{n-1} = [0,1]^{n} \setminus \bigcup_{i=1}^{\infty} D_i,$ for a sequence $D_1,D_2,\ldots$ of open  triadic cubes in $\mathbb{R}^{n}$. For every $i\geq 1$ we have that $\partial D_i\subset\S_{n}$ is a topological sphere of dimension $n-1$, which we call a \emph{peripheral sphere}.


We say that a metric space $X$ is a \emph{Sierpi\'nski carpet or Sierpi\'nski $n$-space} if it is homeomorphic to $\S_1$ or to $\S_n$ for some $n\geq1$, respectively. A \emph{peripheral sphere} of a Sierpi\'nski $n$-space $X$ is a non-separating subset of $X$ which is homeomorphic to $\mathbb{S}^{n-1}$. Equivalently a peripheral sphere is the image in $X$ of a peripheral sphere in $\S_n$.

To prove Theorem \ref{thm:Sierpinski-co-hopf} we introduce and study a class of spaces which we call \emph{slit Sierpi\'nski spaces}, cf. Section \ref{Section:slit-spaces}. Essentially, a slit Sierpi\'nski $(n-1)$-space is a Gromov-Hausdorff limit of a sequence of multiply connected ``slit domains" in $\mathbb{R}^n$.


Here a \emph{slit domain} is a finitely connected domain in $\mathbb{R}^n$ of the form $I\setminus \cup_{i=1}^k s_i$, where $I$ is a box, i.e. $I=(a_1,b_1)\times\ldots\times(a_n,b_n)$, and the slits $s_i\subset I$ are pairwise disjoint $(n-1)$ dimensional hypercubes which are all contained in planes parallel to a fixed $(n-1)$-dimensional plane, e.g. the coordinate hyperplane $\{x_1=0\}$. See Fig. \ref{fig:diadic-slits} for an example of a slit domain in the plane. The topological dimension of an $(n-1)$-dimensional slit space is $n-1$.

Given a slit Sierpi\'nski $n$-space $X$ we consider the \emph{double of X}, denoted by $DX$, which is obtained by identifying two slit Sierpi\'nski spaces along the boundary of the outer box $\partial{I}$. One important feature of doubles of slit $n$-spaces is that they can be thought of as being ``fibered" over an interval $[a,b]\subset\mathbb{R}$ with almost all fibers being homeomorphic to a sphere $\mathbb{S}^{n-1}$ of codimension $1$. This is in contrast to the slit Menger curve, which is fibered over a Sierpi\'nski carpet of Hausdorff dimension $2$, with almost all fibers being $1$-dimensional circles.

In \cite{Mer:coHopf} it was shown  that if $S$ is a slit carpet corresponding to a  very particular sequence of slits in the unit square $[0,1]^2$, then the double of $S$ is QS co-Hopfian.  One important property of the slit carpet in \cite{Mer:coHopf} is porosity. Here we say that a Sierpi\'nski carpet $X$ is \emph{porous} if peripheral circles appear in all locations and scales. This means that for every $x\in X$ and $0<r<\diam X$ there is a peripheral sphere contained in the ball $B(x,r)$ of diameter comparable to $r$. It turns out that porosity alone implies that doubles of slit spaces are QS co-Hopfian.




\begin{theorem}\label{thm:main-cohopf}
If $X$ is a porous slit Sierpi\'nski space whose peripheral spheres are uniformly relatively separated then the double of $X$ is QS co-Hopfian.
\end{theorem}
We refer to Section \ref{Section:slit-spaces} for the precise definition of uniform relative separation used above, which loosely speaking means that the peripheral spheres are not too close to each other. Theorem \ref{thm:main-cohopf} is sharp in the following sense.

\begin{theorem}\label{thm:non-self-similar-cohopf-1}
For every $n\geq 1$ there is an $n$-dimensional slit Sierpi\'nski space which is not porous, but is QS co-Hopfian.
\end{theorem}

Theorems \ref{thm:main-cohopf} and \ref{thm:non-self-similar-cohopf-1} follow from Theorem \ref{thm:main1}. The examples in Theorem \ref{thm:non-self-similar-cohopf-1} are given by a class of slit spaces which we call \textit{standard (or diadic)} non self-similar slit Sierpi\'nski spaces, see  Section \ref{section:diadic-slit-spaces}. These spaces correspond to slit domains where the slits are placed at the centers of the diadic cubes in $[0,1]^n$, cf. Figure \ref{fig:diadic-slits}. We provide a sufficient condition guaranteeing  QS co-Hopfian property for the doubles of diadic Sierpi\'nski spaces, which includes many non-porous examples. In fact we show that there are examples of QS co-Hopfian spaces such that the diameter of the largest peripheral sphere in any ball $B(x,r)$ is of the order $o(r)$ as $r\to0$. This means that a metric space can ``look like" $\mathbb{R}^n$ on small scales, i.e. have Gromov-Hausdorff tangent spaces isometric (or quasisymmetric) to $\mathbb{R}^n$, but still be QS co-Hopfian. This is very different from the case of the slit carpet considered in \cite{Mer:coHopf} since its tangents cannot be quasisymmetrically embedded in $\mathbb{R}^2$, cf. \cite{MerWildrick}.


\subsection{Gromov hyperbolic spaces, groups and their boundaries}
The property of being quasi-symmetrically co-Hopfian is important for boundaries of Gromov hyperbolic spaces and in particular for Gromov hyperbolic groups, see \cite{BSchramm,delaHarpe,Mer:coHopf} and references therein for the background on these topics. In particular, QS co-Hopficity is related to the quasi-isometric co-Hopfian property of unbounded metric spaces.

A map $f:(X,d_X)\to (Y,d_Y)$ is a \textit{quasi-isometric embedding} if there are constants $L\geq 1$ and $C>0$ such that
\begin{align*}
  L^{-1}d_X(x,y) - C \leq d_Y(f(x),f(y)) \leq L d_X(x,y) + C
\end{align*}
for all $x,y\in X$. The spaces $X$ and $Y$ are called \emph{quasi-isometric} if there is a quasi-isometric embedding $f:X\to Y$, which  is a \emph{quasi-isometry}, i.e. if there is a constant $C_1<\infty$ such that for every point $z\in Y$ there is a point $x\in X$ such that $d_Y(f(x),z)< C_1$. A metric space $X$ is \emph{quasi-isometrically co-Hopfian} if every quasi-isometric embedding of $X$ into itself is in fact a quasi-isometry.

It turns out that if $X$ is a roughly geodesic Gromov hyperbolic space then it is quasi-isometrically co-Hopfian if its boundary at infinity $\partial_{\infty} X$ is quasisymmetrically co-Hopfian, cf. \cite{Mer:coHopf}. Moreover, if $(Z,d_Z)$ is a compact metric space then there is a visual roughly geodesic Gromov hyperbolic space $X$ such that $\partial_{\infty} X$ is bi-Lipschitz (and therefore also quasisymmetric) to $(Z,d_Z)$.  Combining this with Theorems \ref{thm:menger-co-Hopf} and \ref{thm:Sierpinski-co-hopf} we obtain the following results.

\begin{theorem}
There is a quasi-isometrically co-Hopfian visual roughly geodesic Gromov hyperbolic space $X$ whose boundary at infinity is homeomorphic to the Menger curve.
\end{theorem}

\begin{theorem}
For every $n\geq 1$ there is a quasi-isometrically co-Hopfian visual roughly geodesic Gromov hyperbolic space $X$ whose boundary at infinity is homeomorphic to the $n$-dimensional Sierpi\'nski space $\S_n$.
\end{theorem}


An important class of Gromov hyperbolic spaces arises in the theory of Gromov hyperbolic groups, cf. \cite{delaHarpe}. For every finitely generated group $G$ and a symmetric generating subset $S\subset G$ one may consider the Cayley graph $\G(G,S)$. The latter is the graph whose vertex set is $G$ and two vertices $a,b\in G$ are connected by an edge if and only if $a^{-1} b\in S\cup S^{-1}$. A natural metric on the Cayley graph is then obtained by defining the length of each edge to be equal to $1$. A finitely generated group $G$ is \emph{Gromov hyperbolic} if its Cayley graph $\G(G,S)$ is a Gromov hyperbolic metric space for some choice of the generating set $S$. It turns out that if $\G(G,S)$ is hyperbolic for one choice of $S$ then it is hyperbolic for any other choice of a generating set in $G$.

A Gromov hyperbolic group $G$ is said to be \emph{quasi-isometrically co-Hopfian} if $\G(G,S)$ is quasi-isometrically co-Hopfian. From the discussion above it follows that a Gromov hyperbolic group is quasi-isometrically co-Hopfian if the boundary at infinity of its Cayley graph, denoted simply by $\partial_{\infty}G$, is QS co-Hopfian when equipped with a visual metric.


If $G$ is a Gromov hyperbolic group then $\partial_{\infty}G$ is either homeomorphic to a sphere $\mathbb{S}^n$  for some $n\geq 1$ (hence is topologically co-Hopfian) or is a bounded complete metric space with no manifold points, cf. \cite[Theorem 4.4]{KapBen}. Besides the spheres, the only spaces known to occur as boundaries of Gromov hyperbolic groups include Sierpi\'nski spaces $\S_n, n\geq1$, universal Menger compacta of topological dimension $n=1,2,3$, Pontriagin surfaces and trees of manifolds, see \cite{BourdonPajot,Dranishnikov,DymaraOsajda,KapKl,Lafont,PrzytyckiSwiatkowski}. 

Menger curve and Sierpi\'nski carpets often occur as boundaries of groups. For instance, if $G$ is indecomposable and its boundary  is connected, has no local cut points and has topological dimension one then $\partial_{\infty}G$ is homeomorphic to either a circle, the Sierpi\'nski carpet, or the Menger curve, cf. \cite{KapKl}. In fact, the boundary of a ``generic" Gromov hyperbolic group is homeomorphic to the Menger curve, cf. \cite{Dahmani:random-groups}. Higher dimensional Sirpi\'nski spaces also appear as boundaries of groups. If $G$ is the fundamental group of a compact negatively curved $(n+2)$ -
dimensional Riemannian manifold $M$, $n\geq 1$, with nonempty and totally geodesic boundary, then $\partial_{\infty}G\subset\mathbb{S}^{n+1}$ is homeomorphic to $\S_n$, cf. \cite{Lafont}. 

It is not known if there is a Gromov hyperbolic group $G$ which is quasi-isometrically co-Hopfian, or equivalently $\partial_{\infty}G$ is QS co-Hopfian, unless $\partial_{\infty}G$ is a sphere. In particular it is not known if there are group boundaries homeomorphic to the Sierpi\'nski carpet or the Menger curve which are quasisymmetrically co-Hopfian, cf. Problem 1.11 in \cite{KapLuk} and also \cite{Mer:coHopf}. In the positive direction, Kapovich and Lukyanenko \cite{KapLuk} showed that if $M$ is a complete  non-compact hyperbolic manifold of dimension $n \geq 3$ of finite volume then $\pi_1(M)$ is quasi-isometrically co-Hopfian.




This paper is organized as follows. In Section \ref{section:prelims} we provide some of the background material. In Section \ref{Section:slit-spaces} we define the slit Sierpi\'nski spaces and formulate some of their properties. In Section \ref{section:mod0-co-hopf} we formulate and prove Theorem \ref{thm:mod0-co-hopf}, which is a general result linking modulus and QS co-Hopfian properties of slit Sierpi\'nski spaces. In Sections \ref{section:main-estimate} and \ref{Section:main-estimate-proof} we formulate and prove our main modulus estimates. Sections \ref{section:slit-menger-definition} through \ref{Section:menger-coHopf-proof} are devoted to the proof of Theorem \ref{thm:menger-co-Hopf}. In Section \ref{section:slit-menger-definition} we give the construction of the slit Menger curve, its double and prove some of their properties. In Sections \ref{Sec:menger-fibered-over-slit-carpet} and \ref{Section:QS-maps-fiber-preserving} we show that the double of the slit Menger curve $D\mathfrak{M}$ is ``fibered" over a slit carpet and QS maps of $D\mathfrak{M}$ are ``fiber preserving". Theorem \ref{thm:menger-co-Hopf} is finally proved in Section \ref{Section:menger-coHopf-proof} by combining the results of the previous sections. A reader interested only in the proof of Theorem \ref{thm:menger-co-Hopf}, can skip most of the material from Sections \ref{Section:slit-spaces} through \ref{Section:main-estimate-proof}. The main ingredients from these sections used in the proof of Theorem \ref{thm:menger-co-Hopf} are the definition of slit carpets, Lemma \ref{lemma:image-Ahlfors-regular} and Lemma \ref{lemma:curves_in_porous_carpets}. In Section \ref{Section:uncountable-menger} we prove Theorem \ref{thm:uncountable-menger}.
In Section \ref{section:remarks-problems} we state several corollaries of our results and formulate some open problems.

\section{Background and Preliminaries}\label{section:prelims}

Given a metric space $(X,d_X)$, a point $x\in X$ and $0<r<\diam X$ we will denote by $B=B(x,r)$ the open ball of radius $r$ and center at $x$. For a constant $C>0$ and a ball $B=B(x,r)$ we let $CB=B(x,Cr)$.

If $E$ and $F$ are subsets of $X$, we define the distance between $E$ and $F$ as follows:
\begin{align*}
  \dist(E,F) = \inf\{d_X(x,y) : x\in E, y\in F\}.
\end{align*}

For $t>0$ we will denote by $\mathcal{H}^t$ the $t$-dimensional Hausdorff measure on $X$. Thus for every $E\subset X$ we have
\begin{align*}
  \mathcal{H}^t(E) = \lim_{\delta\to0}\inf \left\{ \sum_{i=1}^{\infty} r_i^t \, :  \, E\subset\bigcup_{i=1}^{\infty} B(x_i, r_i), \, r_i<\delta\right\}
\end{align*}

Recall, that a metric measure space $(X,\mu)$ is \textit{{Ahlfors $Q$-regular}} for some $Q>0$ if there is a
constant $C\geq 1$ such that for every $x\in X$ and $0<r<\diam X$ the
following inequalities hold
\begin{equation}\label{ineq:Ahlfors}
  \frac{1}{C} r^Q \leq \mu(B(x,r))\leq C r^Q.
\end{equation}

It is well known and easy to see that in (\ref{ineq:Ahlfors}) the measure $\mu$ can be replaced by the Hausdorff measure $\calH^Q$. See \cite{Hein:book} for the background on Haudorff measures, dimension and Ahlfors regularity.

\subsection{Modulus}\label{section:modulus}

Given a metric measure space $(X,\mu)$ and a family of curves $\G$ in $X$
we say that a Borel measurable function $\rho:X\to[0,\infty)$ is
\emph{$\G$-admissible} if $\int_{\g}\rho d s\geq 1$ for every locally
rectifiable curve $\gamma\in\G$, where $ds$ denotes the arclength element.
The $p$-\emph{modulus} of $\G$ for $p\geq1$  is defined as
$$\m_p \G = \inf_{\rho} \int_X \rho^p d\mu,$$
where the infimum is taken over all $\G$-admissible functions $\rho$.

From the definition it follows that every admissible function for $\G$ gives an upper estimate for modulus.

\begin{lemma}[See Lemma 5.3.1 in \cite{HKST}]
Let $\G$ be a family of curves in a Borel subset $A$ of $(X,\mu)$ such that $\mathrm{length}(\g)\geq L>0$ for every $\g\in\G$. Then 
\begin{align}\label{ineq:modulus-basic-estimate}
\m_p(\G)\leq \mu(A)L^{-p}.
\end{align}
\end{lemma}

Some of the most important properties of modulus are listed in the following lemma and we will often use these just by referring to the name of the appropriate property.

If $\G_1$ and $\G_2$ are curve families in $X$, we will say that $\G_2$ overflows $\G_1$ and will write $\G_1 < \G_2$ if every curve $\g_2\in \G_2$ contains some curve $\g_1 \in \G_1$.

\begin{lemma}[See \cite{Hein:book}]\label{lemma:modulus-properties}
For every $p\geq 1$ the following properties hold.
  \begin{enumerate}
    \item $\mathrm{\textsc{(Monotonicity)}}$ $\m_p\G \leq \m_p\G'$,
        if $\G\subset\G'$
    \item $\mathrm{\textsc{(Subadditivity)}}$ $\m_p\G \leq
        \sum_i\m_p\G_i$, if $\G=\bigcup_{i=1}^{\infty}\G_i$.
    \item $\mathrm{\textsc{(Overflowing)}}$ If $\G_1 < \G_2$ then $\m_p(G_1)\geq \m_p(G_2)$.
  \end{enumerate}
\end{lemma}

Thus modulus can be thought of as an outer measure on families of curves in $X$. For this reason one often says that a property holds for $\m_p$-{\emph{almost every
$\g\in\G$}} if it fails only for a family $\G_0\subset\G$ such
that $\m_{p}(\G_0)=0$. We refer to ~\cite{Hein:book,HKActa,Vaisala:lectures} for further details on modulus of curve families
including the definitions of rectifiability and arclength in $\mathbb{R}^n$ as well as in general metric
spaces.

On several occasions we will also need the following result.

\begin{lemma}\label{lemma:modulus-under-projections}
Suppose $\tau:(X,d_X,\mu)\to (Y,d_Y,\nu)$ is an $L$-Lipschitz map of metric measure spaces, and there is a constant $C\geq 1$ such that for every Borel set $E\subset X$ we have 
\begin{align}\label{ineq:measure-non-decreasing}
\mu(E)\leq C\nu(\tau(E)).
\end{align} 
Then for every family of curves $\G$ in $X$ we have 
\begin{align}\label{ineq:modulus-under-lipschitz}
  \m_p \G \leq C L^p \m_p \tau(\G).
\end{align}
\end{lemma}

\begin{proof}
Let $\rho$ be an admissible metric for $\tau(\G)$. Define a metric $\tilde{\rho}=(\rho\circ\tau) \cdot L$ on $X$. Since, $\tau$ is $L$-Lipschitz, we have that for every locally rectifiable $\g\in\G$ the image $f\circ \gamma$ is also locally rectifiable and moreover
$$\int_{\g} \tilde{\rho} \, ds = \int_{\g} (\rho\circ \tau) \cdot L \,ds \geq  \int_{\tau(\g)} \rho ds \geq 1,$$
cf. \cite[Page 12]{Vaisala:lectures} or \cite{HKST}. Since $\m_p$-almost every $\g\in \G$  (and $\g'\in\tau(G)$) is locally rectifiable, it follows that $\tilde{\rho}$ satisfies the admissibility condition for $\m_p$-almost every $\g\in\G$. By (\ref{ineq:measure-non-decreasing})  we have
$\int_{X} (\tilde{\rho})^p d\mu \leq C L^p \int_{Y} \rho^p d\nu$. Taking an infimum over all admissible $\rho$'s completes the proof.
\end{proof} 

\subsection{QS mappings and Tyson's Theorem}
A homeomorphism $f$  between metric spaces $(X,d_X)$ and $(Y,d_Y)$ is called
\textit{quasisymmetric} (or QS) if for all distinct triples $x,y,z\in{X}$ we have
\begin{equation}\label{QS}
\frac{d_Y(f(x),f(y))}{d_Y(f(y),f(z))}\leq{\eta\left(\frac{d_X(x,y)}{d_X(y,z)}\right)},
\end{equation}
for some fixed increasing function $\eta:[0,\infty)\rightarrow[0,\infty)$.

Below, we will need the following result of Tyson, who showed that in quite general spaces quasisymmetry implies quasi-invariance of the moduli of families of curves.

\begin{theorem}[Tyson, \cite{Tyson}]\label{thm:Tyson}
  If $f:X\to Y$ is a quasisymmetric mapping between locally compact, connected Ahlfors $Q$-regular spaces, with $Q>1$, then there is a constant $K\geq 1$ such that
\begin{eqnarray}\label{ineq:geom-QC}
  {K^{-1}}\m_Q f(\G) \leq \m_Q \G \leq K \m_Q f(\G),
\end{eqnarray}
for every curve family $\G\subset X$.
\end{theorem}

The constant $K$ in (\ref{ineq:geom-QC}) depends only on the distortion function $\eta$ and the Ahlfors regularity constants of $X$ and $Y$.

\subsection{Sierpi\'nski spaces and Cannon's Theorem}\label{subsection:menger_definitions} 
A classical theorem of Whyburn states that every compact set obtained by removing a sequence of Jordan domains from the sphere $\mathbb{S}^2$ satisfying  certain properties is homeomorphic to the standard Sierpi\'nski carpet $\S_1$. We will need the following theorem of Cannon \cite{Cannon}, which generalizes Whyburn's characterization to higher dimensions. Note, that Cannon stated his theorem for all $n\geq 2$ except for $n=4$. However, it is known now that the theorem holds for $n=4$ as well, see for instance the discussion in Section 2 of \cite{LafontTshishiku}.

\begin{theorem}[Cannon \cite{Cannon}]\label{thm:Cannon}
  Let $n\geq 2$. Suppose $D_i\subset\mathbb{S}^{n}, i\geq 0,$ is a sequence of topological $n$-balls satisfying the conditions
  \begin{itemize}
  \item[($\S \it{1}$)] $\overline{D}_i \cap\overline{D}_j =\emptyset, \mbox{ for } i\neq j$,
  \item[($\S \it{2}$)] $\diam(D_i)\to 0$ as $i\to\infty$,
  \item[($\S \it{3}$)] $\overline{\left(\bigcup_i D_i \right)}=\mathbb{S}^{n}$.
\end{itemize} Then the compact set $\S=\mathbb{S}^{n} \setminus \bigcup_{i}D_i$ is homeomorphic to $\S_{n-1}$.
\end{theorem}

We will call a set $\S\subset\mathbb{S}^n$ as in Theorem \ref{thm:Cannon} an \emph{$(n-1)$-dimensional Sierpi\'nski space (Sierpi\'nski carpet for $n=2$)} or just a \emph{Sierpi\'nski space} if the dimension is clear from the context. The spheres $\partial(D_i)\cong \mathbb{S}^{n-1}$ will be called the \emph{peripheral spheres} (or circles if they are of dimension $1$) of $\S$.

More generally, a metric space $X$ is a \emph{metric Sierpi\'nski $n$-space} if it is homeomorphic to the standard Sierpinski set $\S_n\subset\mathbb{R}^{n+1}$. An $(n-1)$ dimensional sphere $S$ embedded in a metric Sierpi\'nski $n$-space $X$ is called a \emph{peripheral sphere} if $X\setminus S$ is connected. This is equivalent to the fact that $S=f(\partial {D_i})$ for some homeomorphism $f:\S_n\to X$ and some peripheral sphere $\partial D_i$ of $\S_n$.

\section{Slit Sierpi\'nski spaces}\label{Section:slit-spaces}
In this section we generalize the construction of the slit Sierpi\'nski carpet from \cite{Mer:coHopf} and define slit Sierpi\'nski spaces. These spaces are constructed using sequences of slit domains in $\mathbb{R}^n, n\geq 2$. Unlike \cite{Mer:coHopf} though we do not impose conditions on the geometry or the location of the slits. Our main condition is uniform relative separation described below.

\subsection{Slit domains in $\mathbb{R}^n$}\label{Sec:slit-domains}
For $n\geq 1$ and $1\leq i\leq n$ we denote by $\pi_i:\mathbb{R}^n\to\mathbb{R}$ the projection map onto the $i$-th coordinate.
Let
$I=(a_1,b_1)\times\ldots \times (a_n,b_n)$ be a bounded open box in $\mathbb{R}^n, n\geq 2$. The \emph{center of $I$} is the point
$$c(I)=\left(\frac{b_1-a_1}{2},\ldots,\frac{b_n-a_n}{2}\right).$$

We say that a subset $s\subset I$ is a \emph{slit or a vertical hypercube in $I$} if
\begin{align*}
  s=\{x\}\times[c_2,d_2]\times \ldots \times[c_n,b_n],
\end{align*}
such that $ l(s):= d_2-c_2 = \ldots = d_n-c_n.$
%
Thus, a slit $s\subset I$ is an $(n-1)$-dimensional closed box contained in the hyperplane $\pi_1^{-1}(x)$ for some $x\in\pi_1(I)$ all the sides of which have equal lengths, see Fig. \ref{fig:slit-in-3d}. We will call $l(s)$ the \emph{sidelength of $s$}.

Given a sequence of disjoint vertical
slits $\mathcal{S}=\{s_i\}_{i=1}^{\infty}$ compactly contained in $I$, for
every $i\geq 1$ we define the \textit{slit domain} $S_i\subset I$ and the
\textit{infinite slit set} $S\subset I$ by letting
\begin{align*}
  S_i = I\setminus \bigcup_{j=1}^i s_j
\quad
\mbox{ and }
\quad
  S = I \setminus \bigcup_{j=1}^{\infty} s_j.
\end{align*}

Throughout the paper we will impose some conditions on the sequence of slits $\calS=\{s_i\}\subset I$ analogous to ($\S \it{1}$),($\S \it{2}$),($\S \it{3}$) of Cannon's theorem.

First, we need to quantify the notion of disjointness. Recall, that \emph{relative distance} between two subsets $E$ and $F$ of a metric space $X$ is given by
\begin{align*}
  \Delta(E,F) = \frac{\dist(E,F)}{\min (\diam E, \diam F)}.
\end{align*}
Now, a sequence of subsets $\{X_i\}_{i\in\mathbb{N}}$ of $X$ is \textit{uniformly relatively
separated} if there is a constant $\sigma>0$ such that $\D(X_i,X_j) \geq \sigma$ if $i\neq j$.
%
The notion of uniform relative separation of peripheral spheres of metric carpets is crucial
in the study of their quasiconformal geometry, cf. \cite{Bonk:SCUnif,BonkMerenkov}.

We will say that the slits $\calS=\{s_i\}\subset I$ \emph{are uniformly relatively separated in $I$} if the collection $\{\partial I, \calS\}$ is uniformly separated, i.e. there is a constant $\sigma>0$ such that for all distinct $i,j\in\mathbb{N}$ the following inequalities hold,
\begin{align} \label{uniform-rel-separation}
  \D(s_i,\partial I) \geq \sigma \quad \mbox{ and } \quad \D(s_i,s_j) \geq \sigma.
\end{align}

Everywhere below we will assume that the slits $s_i\subset I$ satisfy the following three conditions:

\begin{itemize}
  \item[($\calS_1$)] $s_i$'s are uniformly relatively separated in $I$,
  \item[($\calS_2$)] $\diam (s_i) \to 0$ as $i\to\infty$,
  \item[($\calS_3$)] $\calS$ is dense in $I$.
  \end{itemize}
Property $(\calS_1)$ above may be thought of as a quantitative version of $(\S \it 1)$.

Often we will assume another property, which is related to the notion of porosity. We say that the slits $s_i\in I$ \emph{occur on all locations and scales in} $I$ if the following condition is satisfied.
\begin{itemize}
  \item[($\calS_4$)] There is a constant $C\geq 1$ such that for every ball $B=B(x,r)\subset I$ there is a slit $s_i\subset B$ such that $\diam s_i \geq r/C$.
\end{itemize}
Note that $(\calS_4)$ implies $(\calS_3)$ but the reverse implication is not true in general.

\subsection{Slit Sierpi\'nski spaces and their doubles}\label{subsection:slit-spaces}

%
Given a domain $D\subset\mathbb{R}^n, n\geq 1$ the \emph{inner or path metric $d_D$} is defined by
\begin{align*}
  d_{D}(x,y) = \inf \{l(\g)\, | \, \g \mbox{ connects $x$ and $y$ } \}\,
\end{align*}
where $x,y\in D$, $l(\g)$ is the Euclidean length (or $\calH^1$ measure) of $\g$ and the infimum is over all the curves $\g\subset D$ connecting $x$ and $y$.

For a sequence of slits $\calS\subset I \subset \mathbb{R}^n$ satisfying the properties $(\calS_1)-(\calS_3)$, we will construct a metric space $M(\calS)$ corresponding to $\calS$ such that $M(\calS)$ will have topological dimension $n-1$ and which may be (homeomorphically) embedded in $\mathbb{R}^n$. The presentation here follows \cite{Mer:coHopf}.

 Let $\overline{M}_0=\bar I$. For $i\geq 1$ let $\overline{M}_i$ denote the completion of the domain $S_i$ in the path metric $d_{{M}_i}$. The new metric on the completion will be denoted by $d_{\overline{M}_i}$. Note that the boundary components of $\overline{M}_i$ corresponding to the slits of $S_i$ are homeomorphic to $(n-1)$ dimensional sphere $\mathbb{S}^{n-1}$, and so we call them \emph{peripheral spheres} of $\overline{M}_i$ and the remaining boundary component - the \emph{outer peripheral sphere} or \emph{outer boundary}. For every $i,j\in\mathbb{N}\cup\{0\}$ such that $i\leq j$ there is a $1$-Lipschitz map $  \tau_{i,j}: \overline{M}_j\to\overline{M}_i$, 
which identifies the points on the slits of ${\overline{M}_j}$ which correspond to the same point in ${\overline{M}_i}$. Equivalently,
\begin{align*}
  \tau_{i,j}(p)=\tau_{i,j}(q) \quad \mbox{ whenever } \quad d_{\overline{M}_i}(p,q)=0.
\end{align*}
Thus, we obtain an inverse system of topological spaces $(\overline{M}_i,\tau_{i,j})$. We denote
 \begin{align*}
   M(\calS) = \lim_{\longleftarrow}(\overline{M}_i,\tau_{i,j})
 \end{align*}
 and call $M(\calS)$ the \emph{slit Sierpi\'nski space (or carpet) corresponding to $\calS$}. More explicitly, the points in the slit space $M(\calS)$ are sequences $(p_1,p_2,\ldots)$, such that for every $i\geq0$ we have $p_i\in M_i$ and $p_i=\tau_{i,i+1}(p_{i+1}).$

 Note, that $M(\calS)$ is a compact Hausdorff topological space. For $i\geq 0$ we will denote by $  \tau_i:M(\calS)\to \overline{M}_i$
 the natural projections. The \emph{slits} and the \emph{outer boundary} of $M(\calS)$ are defined as the inverse limits of the slits and the outer boundary of $\overline{M}_i$ and as such are topological spheres of dimension $n-1$. From the fact that the slits $\calS$ are dense in $I$ it follows that the slits (or peripheral spheres) are dense in $M(\calS)$.

Given $p=(p_0,p_1,\ldots),q=(q_0,q_1,\ldots)\in M(\calS)$ the sequence $\{ d_{\overline{M}_i}(p_i,q_i)\}$ is non-decreasing, bounded and therefore convergent. Hence the limit of the sequence is independent of the enumeration of the sequence of the slits $\{s_i\}$ and we may unambiguously define a distance function on $M(\calS)$ as follows
\begin{align*}
  d_{\calS}(p,q)=\lim_{n\to\infty} d_{\overline{M}_i}(p_i,q_i).
\end{align*}

Since the metric on $M(\calS)$ is independent of the enumeration of the sequence of slits $s_i$, from now on without loss of generality we will assume that for every $i\geq 0$, we have
\begin{align*}
l(s_i)\geq l(s_{i+1}).
\end{align*}

Recall that a curve $\g$ in a metric space $X$ is a \emph{geodesic} if for every pair of points $p$ and $q$ on $\g$ the distance between them is equal to the length of $\g$ between $p$ and $q$. A metric space is said to be geodesic if every pair of points $p$ and $q$ in $X$ can be connected by a geodesic.

It was shown in \cite{Mer:coHopf} that the slit carpet defined in that paper was a geodesic metric space. The same proof works for $M(\calS)$.

%

The following result is a generalization of Lemma 2.1 in \cite{Mer:coHopf}, where it is proved for $n=2$ and for a very symmetric and self similar situation.

\begin{lemma}\label{lemma:homeo}
Suppose $\calS$ is a sequence of slits in $I\subset\mathbb{R}^n$ satisfying $(\calS_1)-(\calS_3)$. Then the metric space $(M(\calS),d_{\calS})$ is homeomorphic to the $(n-1)$-dimensional Sierpi\'nski space $\S_{n-1}$ whose peripheral spheres are the slits together with the outer boundary of $M(\calS)$.
\end{lemma}
\begin{proof}
The idea is to construct a Lipschitz embedding of $L: M(\calS) \to \mathbb{R}^n$ so that the conditions of Theorem \ref{thm:Cannon} are satisfied. Let $\s>0$ be the constant of uniform relative separation of $\calS$ in $I$ and for $i\geq 1$ denote   $\eps_i:={\s l(s_i)}/{2}.$
Furthermore, let $U_i$ be the $\eps_i$ neighborhood of the boundary component corresponding to $s_i$ in $\overline{M}_i$,
\begin{align*}
  U_i = \{ p\in\overline{M}_i : \dist(p,\tau_{0,i}^{-1}(s_i))< \eps_i\}.
\end{align*}
By the definition of $\sigma$ we have that $U_i\cap s_j =\emptyset$ for $j<i$, since $l(s_i)\leq l(s_j)$.

We construct the map $L$ by induction. Let $\la_1:\overline{M}_1\to \overline{M}_0 = I_0 \subset\mathbb{R}^n$, be defined so that it ``opens up" the vertical slit $s_1$ to a topological $(n-1)$-sphere which bounds a topological $n$-ball $D_i\subset I$, and is equal to the identity outside of the $\eps_1$ neighborhood of $s_1$. Moreover, $\la_1$ can be chosen to be $C_1$ Lipschitz for any constant $C_1>1.$ Indeed, one way of defining $\la_1$  is as follows. For $p=(x_1,\ldots,x_n)\in s_1$ let $p^{+},p^{-}\in\overline{M}_1$ be the two (``right and left") preimages of $p$ under $\tau_{0,1}$, and define
\begin{align}
  \la_1(p^{\pm}) = (x_1 \pm \eps \dist(p,\partial{s_1}), x_2,\ldots,x_n),
\end{align}
where $\partial s_1$ denotes the ``$(n-2)$-dimensional boundary of the slit", i.e. the boundary of $s_2$ in the hyperplane $\pi_1^{-1}(x_1)$. It is easy to see that $\la_1\mid_{s_1}$ is $(1+\eps)$-Lipschitz and can be extended to a Lipschitz map which agrees with $\tau_{0,1}$ (or identity) outside of $U_1$.

For $i\geq 1$ let $C_i:=(1+2^{-i})$. Then, because $\tau_{i-1,i}(U_i)$ does not intersect any of the boundary components of $\overline{M}_{i}$, there is a one-to-one $C_i$-Lipschitz map
$\la_i:\overline{M}_{i} \to \overline{M}_{i-1},$ such that
\begin{itemize}
 \item[\emph{(a).}] $\la_i$ agrees with $\tau_{i-,i+}$ on $\partial U_i$, and
 \item[\emph{(b).}] $\la_i(s_i)$ is a topological sphere such that
$\dist(\la_i(s_i), \la_i(\partial U_i))\geq \eps_i/2.$
\end{itemize}
In other words $\la_i$ ``opens up" the slit $s_i$ into a topological sphere which is bounded away from $\partial U_i$. The map $\la_i$ can be constructed like $\la_1$ above, but $\eps$ has to be small enough so that condition $(b)$ above is satisfied.

Next, for $n\geq 1$ let $L_n= \la_1\circ\ldots\circ \la_n: \overline{M}_n \to I \subset\mathbb{R}^n.$ Then $L_n$ is Lipschitz with the Lipschitz constant $C_1\ldots C_n\leq\prod_{i=1}^{\infty}(1+2^{-i})\leq e^2.$ By the Arzel\`a - Ascoli theorem the sequence of maps $L_n\circ{\tau_n}:M(\calS)\to I$  has a subsequence converging to a $C$-Lipschitz map $L_{\infty}: M(\calS) \to M_0 \subset\mathbb{R}^3$.

To see that $L_{\infty}$ is injective, note that by construction it is injective on the set of points not belonging to the slits of $M(\calS)$. Moreover, for a slit $J\subset \overline{M}_m$ and every $n>m$ we have that the maps $L_n$ and  $L_m\circ \tau^{-1}_{m,n}$ coincide on $\tau^{-1}_{m,n}(J)$. Therefore $L_{\infty}$ is an embedding of $M(\calS)$ into $I$, which maps every pericheral sphere $\tau_0^{-1}(J)$ onto a topological sphere in $I$ which bounds a topological ball $D_i\subset I$. Thus $M(\calS)$ is homeomorphic to the set $M=I\setminus \bigcup_{i=1}^{\infty} D_i$ under $L_{\infty}$. The conditions $\calS_1,\calS_2$ and $\calS_3$ imply that $M\subset I$ satisfies conditions $(\mathscr{S}1),(\mathscr{S}2)$ and $(\mathscr{S}1)$ of Cannon's theorem, and applying the latter shows that $L_{\infty}(M(\calS))$ is homeomorphic to the Sierpi\'nski space $\mathscr{S}_{n-1}$.
\end{proof}
%
%

If $M(\calS)$ is an $n$-dimensional slit Sierpi\'nski space then the metric space obtained by gluing two copies of $M(\calS)$ along the outer peripheral spheres by the identity map is called the \emph{double of} $M(\calS)$ and is denoted by $DM(\calS)$.
%
From Cannon's theorem it follows that $DM(\calS)$ is homeomorphic to $\S_n$ as well.

We say that a metric Sierpi\'nski space $X$ is \emph{porous} if for every ball $B=B(x,r)
\subset X$ there is a peripheral sphere $S$ contained in $B$ such that the
diameter of $S$ is comparable to the radius of the ball, i.e. $\diam(S)\geq
{r}/{A}$ where $A<\infty$ is independent of $x$ and $r$. Note that $M(\calS)$ is porous if and only if $(\calS_4)$ holds, i.e. if the slits $s_i\in \calS$ occur on all locations and scales.

The following result is analogous to a similar result for the slit Sierpi\'nski carpet, in \cite{Mer:coHopf}. The proof in \cite{Mer:coHopf} immediately generalizes to our case, so we omit it. The main difference is that even though the slits in our case are not placed in a self-similar fashion they are nevertheless uniformly relatively separated.

%

\begin{lemma}\label{lemma:Merenkov-ahlfors-regular}
There is a constant $C\geq 1$, independent of $i\geq 1$ such that for every Borel set $E\subset M(\calS)$ we have
  \begin{align}\label{ineq:Merenkov_measures_comparable}
    \frac{1}{C}\calH^n(\tau_i(E)) \leq \calH^n(E) \leq C \calH^{n}(\tau_i(E)).
  \end{align}
Furthermore, the metric spaces $M(\calS)$ and $DM(\calS)$ equipped with the Hausdorff $n$-measure $\calH^n$ are $(n-1)$-dimensional Sierpi\'nski spaces which are compact, path connected, and Ahlfors $n$-regular. Moreover, if the slits $s_i\in\calS$ appear on all locations and scales then $M(\calS)$ are $DM(\calS)$ are porous.
\end{lemma}


\subsection{Standard (or diadic) non-self-similar slit Sierpi\'nski spaces} \label{section:diadic-slit-spaces} A particular class of slit Sierpi\'nski spaces which we will consider below may be defined as follows.

\begin{figure}[h]
\centerline{
\includegraphics[width=0.25\textwidth]{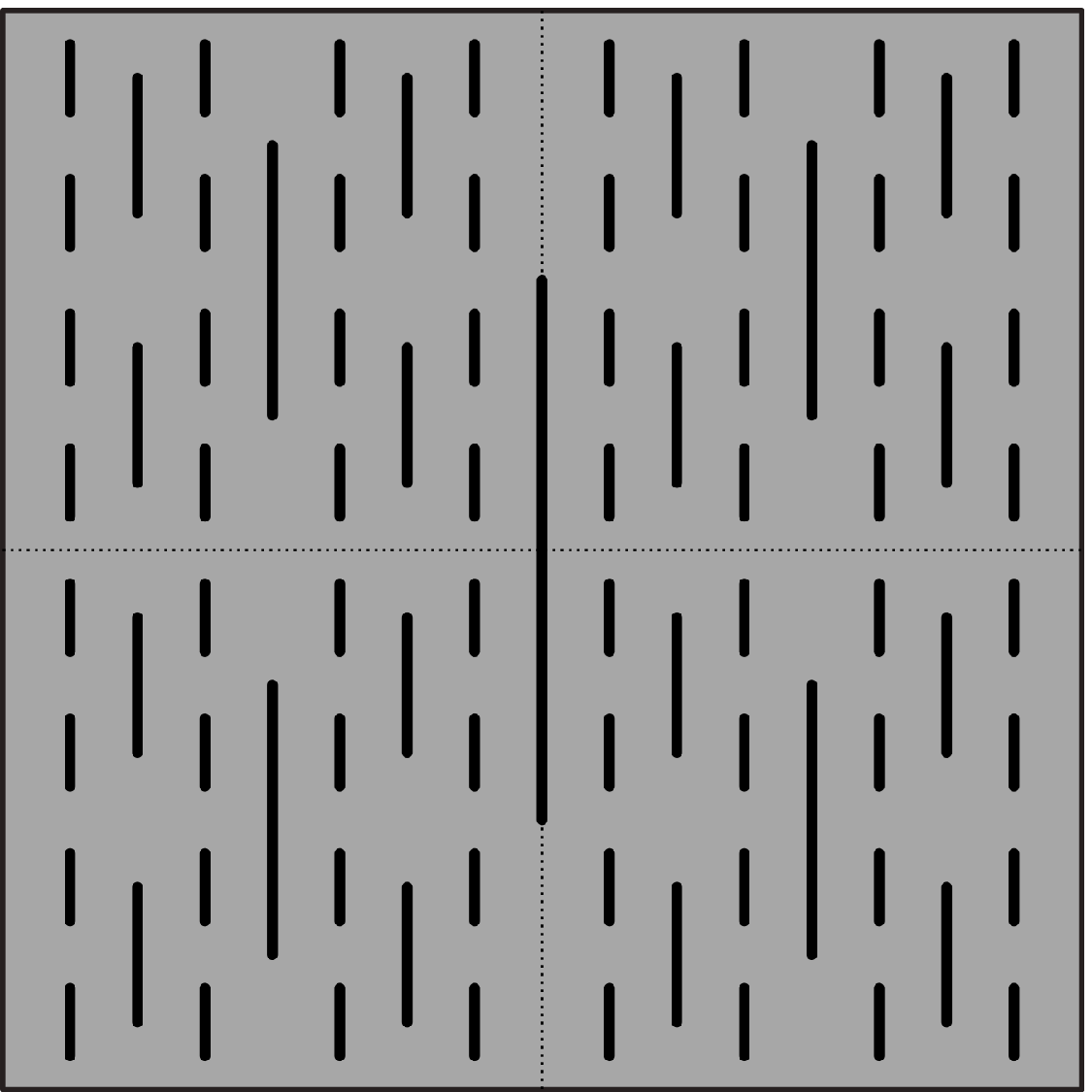} \qquad \qquad
\includegraphics[width=0.25\textwidth]{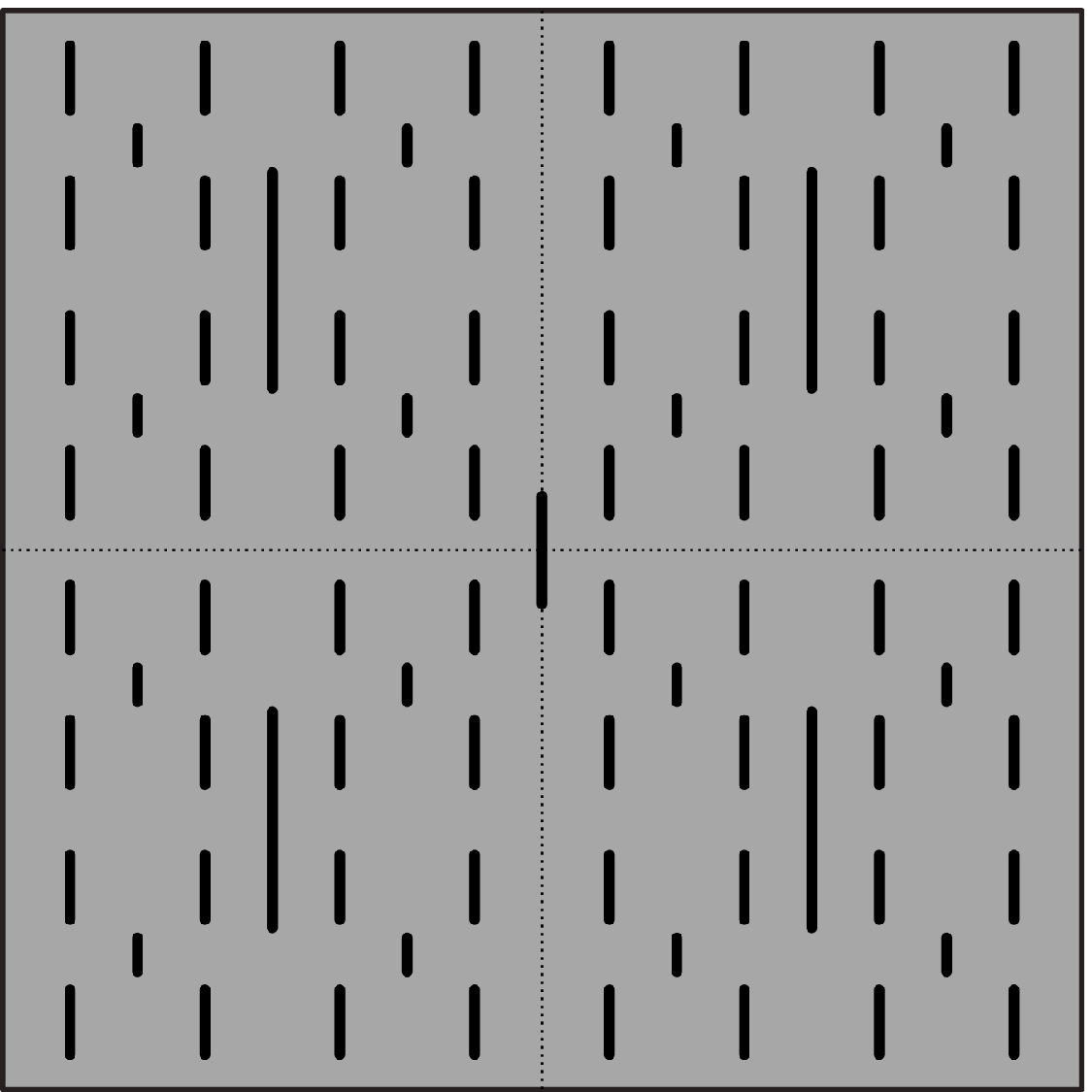}}
\caption{{Standard slit domains in $\mathbb{R}^2$. In the picture on the left relative lengths of the slits are constant multiples of the sidelength of the corresponding diadic square and the domain corresponds to $(1/2,1/2,1/2,1/2)$. In the right picture the domain corresponds to  $(1/10,2/5,1/8,1/2)$.}}\label{fig:diadic-slits}
\end{figure}


Let $\Delta_i$ be the collection of dyadic cubes of generation $i$ and $\Delta=\cup_{i=1}^{\infty}\Delta_i$ be the collection of all dyadic cubes in $[0,1]^n, n\geq 2$. For every $Q\in\Delta$ we will denote by $l(Q)$ the sidelength of $Q$. Given a sequence $\textbf{r}=\{r_i\}_{i=0}^{\infty}$, such that $0\leq r_i<1, \forall i\geq0$ we define a sequence of slits $\calS_{\mathbf{r}}=\{s(Q)\}_{Q\in\D}$ in $[0,1]^n$ corresponding to $\textbf{r}$ as follows. If $Q=Q_0=[0,1]^n$ we define $s(Q)$ as the vertical hypercube of sidelength
$l(s(Q_0))=r_0 = {r_0}/{2^0}$
with the same center as the center of $Q_0$. If $r_0=0$ we define $s(Q)$ to be the empty set.

In general, if $Q\in\D_i$ is an $i$-th generation diadic cube then $s(Q)$ is the slit of sidelength $l(s(Q))=r_i l(s(Q))={r_i}/{2^i}$,
%
such that the center of $s(Q)$ is the same as the center of the cube $Q$. Again, $s(Q)=\emptyset$ if $r_i=0.$

If $\textbf{r}$ is a sequence as above, we will call the space $M(\calS_r)$ a \emph{standard (or diadic) non-self-similar slit Sierpi\'nski space}. 
Note that if $r_i\to0$ then $M(\calS_{\textbf{r}})$ is not porous.


\section{Modulus and QS co-Hopfian spaces}\label{section:mod0-co-hopf}
Suppose $M(\calS)$ is a Sierpi\'nski $(n-1)$-space corresponding to a sequence of slits $\calS\subset I\subset \mathbb{R}^n$. We will denote by $\vartheta$ the ``projection" map from $M(\calS)$ to the first coordinate axis in $\mathbb{R}^n$, i.e.
\begin{align*}
  \vartheta : = \pi_1 \circ \tau_0 : M(\calS) \to \mathbb{R}.
\end{align*}

We say that a subset (e.g. a curve or a sphere) $E$ of $M(\calS)$ of $DM(\calS) $ is \emph{vertical} if $\vartheta(E)$ is a point in $\mathbb{R}$.
Now, if $X$ is a Sierpi\'nski $(n-1)$-space $M(\calS)$ of $DM(\calS) $ we denote by $\G_{v}(X)$ and $\G_{nv}(X)$ the families of vertical and non-vertical curves in $X$, respectively, i.e.
\begin{align*}
\G_{v}(X)&:=\{ \g \subset X : \vartheta (\g)= \{p\}\in\mathbb{R}\},\\
\G_{nv}(X)&:= \{ \g \subset X : \g\notin \G_{nv}(X) \}\\
&\,\,=\{ \g \subset X : \overline{\vartheta(\g)} = [a,b] \mbox{ for some } a<b \mbox{ in } \mathbb{R}\}.
\end{align*}
If the underlying space $X$ is clear from the context we will suppress it from the notation and will simply denote the families by $\G_{v}$ and $\G_{nv}$.

It turns out that the study of the families of vertical and non-vertical curves is essential in determining if a slit Sirpi\'nski space is quasisymmetrically co-Hopfian or not. This is manifested in the following result. 

\begin{theorem}\label{thm:mod0-co-hopf}
If $n>1$  and $M(\calS)$ is a slit Sierpi\'nski space of topological dimension $n-1$ such that
$\m_n(\G_{nv}(M(\calS))=0$ then the double of $M(\calS)$ is quasisymmetrically co-Hopfian.
\end{theorem}
We will see below, that even though non-vertical families can have a vanishing modulus, this is not the case for vertical families, i.e. for \emph{every} slit Sierpi\'nski space $M(\calS)$ we have $\m_n\G_{v}(M(\calS))>0.$ Thus Theorem \ref{thm:mod0-co-hopf} essentially says that if the ``vertical and nonvertical directions" in $M(\calS)$ are very different then $DM(\calS)$ is quasisymmetrically co-Hopfian. 


\subsection{Modulus and quasisymmetric co-Hopficity}\label{subsection:mod0-co-hopf}

In order to prove Theorem \ref{thm:mod0-co-hopf} we first show that if $f$ is a quasisymmetric map of a slit Sierpi\'nski space $M(\calS)$ (or its double) of topological dimension $n-1$ into itself then $f(M(\calS))$ is Ahlfors $n$-regular. This will allow us to use Theorem \ref{thm:Tyson}. In particular, similarly to \cite{Mer:coHopf}, if $\m_n(\G_{\calS})=0$ then $f$ maps almost every vertical curve in $M(\calS)$ to a vertical curve.  We will then show that every vertical $(n-1)$-sphere in $DM(\calS)$ is mapped to a vertical sphere, which will then imply that $f$ induces a mapping of the interval $\vartheta(DM(\calS))=[a_1,b_1]\subset\mathbb{R}$ into itself. Some more work then will show that this induced map is in fact onto $[a_1,b_1]$, implying that $f(DM(\calS))=DM(\calS)$. 


\subsection{Ahlfors regularity of the image}


In order to show that $f(M(\calS))$ is Ahlfors regular we will use the following general result.

\begin{lemma}\label{lemma:lower-Ahlfors-regular}
Let $Q>1$ and $(X,\mu)$ be a bounded Ahlfors $Q$-regular space. Suppose that there is a
constant $\sigma>0$ such that for every ball $B(x,r)\subset X$ there is a
family of curves $\G_{x,r}$ in $B(x,r)$ of diameter at least $r/C$ and such
that $\m_Q(\G_{x,r})\geq \sigma$. If $f$ is a quasisymmetric mapping of $X$
then we have
\begin{align*}
  \calH^Q(B(y,\d)) \geq A \d^Q,
\end{align*}
for every ball $B(y,\d)\subset f(X)$, where $A$ is independent of $y$ and $\d$.
\end{lemma}

\begin{proof}
  Let $B'=B(y,\d)$ be a ball in $f(X)$. Then by quasisymmetry there is a ball
  $B=B(x,r)\subset X$ such that
$f(B(x,r))\subseteq B' \subseteq f(B(x,\eta(1)r)).$  In particular,
$\diam(f^{-1}(B))\leq \diam B(x,\eta(1)r) = 2\eta(1)r.$
Hence for every  curve $\g\in\G_{x,r}$ in $B$ we have
$$\diam \g\geq C^{-1}{r}\geq ({2C\eta(1)})^{-1}{\diam f^{-1}(B')}.$$
By Proposition $10.8$ in \cite{Hein:book} we have that that for every
$\g\in\G_{x,r}$ the following inequality holds
$$\frac{\diam f(\g)}{\diam{B'}} \geq \frac{1}{2\eta\left(\frac{\diam{\g}}{\diam(f^{-1}(B'))} \right)}\geq \frac{1}{2\eta((2C\eta(1))^{-1})}.$$ 
Thus, for each $\g\in\G_{x,r}$ we have 
$\mathrm{length}(\g)\geq\diam f(\g)\geq C_1 \d,$
where $C_1$ depends only on $\eta$ and $C$. Therefore, applying inequality (\ref{ineq:modulus-basic-estimate}) to $f(\G_{x,r})$ with $\mu=\mathcal{H}^Q \lfloor_{A}$ we have
\begin{align*}
\m_Q (f(\G_{x,r})) \leq \frac{\calH^Q(f(B(x,r)))}{(C_1\d)^Q} \leq \frac{1}{C_1^Q} \frac{\calH^Q(B(y,\d))}{\d^Q}.
\end{align*}
Finally, using inequality (\ref{ineq:geom-QC}) of Tyson's theorem we obtain
\begin{align*}
\calH^Q(B(y,\d))\geq C_1^Q K\m_Q(\G_{x,r}) \d^Q \geq A \d^q,
\end{align*}
where $A=C_1^Q K \sigma$.
\end{proof}

\begin{corollary}\label{lemma:image-Ahlfors-regular}
Suppose $\calS=\{s_i\}\subset I$ is a sequence of slits in $I\subset\mathbb{R}^n$  such that the conditions $(\calS_1)-(\calS_3)$ are satisfied. If
$f$ is a quasisymmetric embedding of $M(\calS)$ or $DM(\calS)$ into itself then the image $f(M(\calS))$ or $f(DM(\calS))$ is Ahlfors $n$-regular.
\end{corollary}

\begin{proof}
The fact that $f(M(\calS))$ is upper $n$-regular follows from the Ahlfors regularity of $M(\calS)$, cf. Lemma \ref{lemma:Merenkov-ahlfors-regular}. To show that $f(M(\calS))$  is lower Ahlfors $n$-regular we need to check that the condition of Lemma \ref{lemma:lower-Ahlfors-regular} is satisfied. For this choose a ball $B=B(p,r)\subset M(\calS)$ and let $B'=B(q,cr)\subset I$ be the ball contained in $\tau_0(B)$ given by Lemma 3.2. Next, let  $\G_{p,r}=\tau_0^{-1}(G_{p,r})$ where $G_{p,r}$ is the family of vertical curves in $B'\subset I$ of diameter at least $cr/2$. Standard modulus estimates imply that $\m_n(G_{p,r})\geq 1$ for every $p\in M(\calS)$ and $r\in(0,\diam(M(\calS)))$. Therefore since $\tau_0$ is Lipschitz, using inequalities (\ref{ineq:Merenkov_measures_comparable}) we obtain that $\m_n(\G_{p,r})\geq 1/C$ for some $C<\infty$.
\end{proof}

\subsection{Most vertical curves are mapped to vertical curves in $DM(\calS)$}

\begin{lemma}\label{lemma:vert-to-vert}
Suppose $\calS=\{s_i\}\subset I$ is a sequence of slits in $I\subset\mathbb{R}^n$  such that
$\m_n(\G_{nv})=0$. If $f$ is a quasisymmetric embedding of $DM(\calS)$ into itself, then
$$\m_n(\{\g\in \G_v : f(\g) \in \G_{nv} \}) = 0.$$
In other words, $f$ maps $\m_n$-almost every closed vertical curve to a closed vertical curve.
\end{lemma}

\begin{proof}
Let $\G_{v\to nv}$ be the family of closed vertical curves (circles) in $DM(\calS)$ which is mapped by $f$ to non-vertical ones. Then $f(\G_{v\to nv})\subset \G$, where $\G$ is the family of all closed non-vertical curves in $DM(\calS)$. By monotonicity of the modulus we have $\m_n(f(\G_{v\to nv}))\leq \m_n(\G)=0.$
Now, by Corollary  \ref{lemma:lower-Ahlfors-regular} we have that  $f(DS(\calS))$ is Ahlfors $n$-regular and therefore by Tyson's theorem, $f$ quasipreserves $n$-modulus and therefore $\m_n(\G_{v\to nv})=0.$
\end{proof}

\subsection{Vertical spheres in $DM(\calS)$}

A subset $\s_x$ of $DM(\calS)$ which is homeomorphic to $\mathbb{S}^{n-1}$ and is such that $(\tau_0 \circ \pi_1) (\s_x) = x$ for some $x\in(a_1,b_1)$ will be called a \emph{vertical sphere in $DM(\calS)$}.

Note, that for almost every $x\in(a_1,b_1)$ the set
\begin{align*}
  \s_x = ( \tau_0 \circ \pi_1)^{-1}(x)
\end{align*}
is a well defined subset of $DM(\calS)$ which does not intersect any of the slits and which is obtained by gluing two copies of $n-1$ dimensional cubes along their boundaries, and as such it is homeomorphic to $\mathbb{S}^{n-1}$.

\begin{lemma}\label{lemma:vertical-spheres}
Suppose $\calS=\{s_i\}\subset I\subset\mathbb{R}^n$ is a sequence of slits in $I$ such that $\m_n(\G_{nv})=0$. If $f$ is a quasisymmetric embedding of $DM(\calS)$ into itself then it takes every vertical sphere to a vertical sphere.
\end{lemma}

\begin{proof}
Let $2\leq k\leq n$ and $X_k$ be the $k$-th coordinate axis in $\mathbb{R}^n$. Denote by $\G^k_{v}$ the family of curves $\g$ in $M(\calS)$ such that $\tau_0(\g)\subset I$ is parallel to $X_k$ and connects the two faces of $I$ which are perpendicular to $X_k$. Furthermore, let
\begin{align*}
  \G^k_{v\to v}& =\{ \g\in \G^k_v : f(\g)\in\G_v \}, \\
  \G^k_{v \to nv} & = \G^k_v \setminus \G^k_{v\to v}.
\end{align*}
and 
$$G^k_{v\to v}=\tau_0(\G^k_{v\to v}) \mbox{  and  } G^k_{v\to nv}=\tau_0(\G^k_{v\to nv}).$$
Therefore, using monotonicity and Tyson's theorem for every $2\leq k \leq n$ we have
\begin{align*}
  \m_n \G^k_{v\to nv}
  &\leq
  \m_n\G_{v \to nv}
   \leq K \m_n f(\G_{v\to nv})
   \leq  K \m_n \G_{nv} = 0.
\end{align*}
and hence $\m_n \G_{v\to nv}^{k} = 0$ for all $2\leq k\leq n.$

Now, since $G^k_{v\to nv}$ is a family of disjoint parallel intervals each of which is isometric to its preimage in $M(\calS)$ inequalities ( \ref{ineq:Merenkov_measures_comparable}) imply that $\m_n(G^k_{v\to nv}) \asymp \m_n \G^k_{v\to nv}$ and thus
$\m_n(G^k_{v\to nv})=0$.

But since $G^k_{v\to nv}$ is a product family we have
\begin{align*}
  \m_n(G^k_{v\to nv}) = \frac{\calH^{n-1}(\pi_k^{\perp}(G^k_{v\to nv}))}{b_k-a_k},
\end{align*}
where $\pi_k^{\perp}$ is the projection onto the hyperplane $X_k^{\perp}$.
Since $\m_n(G^k_{v\to nv})=0$ it follows that $\calH^{n-1}(\pi_k^{\perp}(G^k_{v\to nv}))=0$ or that $\pi_k^{\perp}(G^k_{v\to v})$ is a full measure set (in particular is dense) in $\pi_k^{\perp}(I)$.
By continuity of $\vartheta\circ f$ it follows that for \emph{every} vertical curve $\g\in \G^k_{v}$ we have that $f(\g)\in\G_v$, i.e. $\vartheta\circ f(\g)$ is a point in $(a_1, b_1)$.

Now, let $x\in(a_1,b_1)$ and $\s_x=\vartheta^{-1}(x)\subset M(\calS)$ be a vertical square. We want to show that $\vartheta (f(\s_x))$ is a point.  Note, that for every two points $p,q \in \s_x$ there is a curve $\d_{p,q}=\d_2\cup\ldots\cup \d_n\subset\s_x$ connecting $p$ and $q$ such that  $\tau_0(\d_k)$ is a closed interval in $I$ parallel to the axis $X_k$, $k=2,\ldots n$. Thus $f(\d_k)$ is a vertical curve for every $k=2,\ldots,n$ and therefore $\vartheta(f(\d_{p,q}))$ is a point in $(a_1,b_1)$. Since this is true for every pair of points in $\s_x$ it follows that for every $x\in(a_1,b_1)$ we have that $\vartheta(f(\s_x))$ is a point in $(a_1,b_1)$.

Therefore if $f$ is now a quasisymmetric mapping of $DM(\calS)$ into itself and $\Sigma$ is a vertical sphere in $DM(\calS)$ then  $f(\Sigma)$ is a point and therefore $f(\Sigma)$ is a vertical sphere.
\end{proof}

Theorem \ref{thm:mod0-co-hopf} now follows from Lemma \ref{lemma:vertical-spheres} and the following result.


\begin{lemma}\label{lemma:vert-spheres-cohopf}
  If $f$ is a quasisymmetric embedding of $DM(\calS)$ into itself which takes vertical spheres to vertical spheres then $f$ is onto.
\end{lemma}

\begin{proof}
%
%
By our assumption there is a sequence of closed vertical
spheres $\s_{i}\subset DM(\mathcal{S})$ with $\pi_1(\tau_0(\s_i))\to a_1$, such that $f(\s_i)$ is a vertical
sphere for each $i\geq 1$. Let
$$\L = \tau^{-1}_0(L) \mbox{ and } \mathcal{R} = \tau^{-1}_0(R).$$
Next we show that either $f(\L)= \L$ or $f(\L)=\mathcal{R}$.

 Since each $\s_i$ separates $DS$, i.e. $DM(\calS)\setminus
\s_i$ is disconnected, we may denote by $L_i$ the connected component of
$DM(\calS)\setminus \s_i$ containing $\L$. Furthermore, we enumerate $\s_i$'s so that
$L_{i+1}\subset L_{i}$. Then
$\L=\overline {\L} = \bigcap_{i=1}^\infty
\overline{L_i}$
and
\begin{align}\label{preserving-vertical-sides}
  f(\L) = \bigcap_{i=1}^{\infty} \overline{f(L_i)}.
\end{align}
Since $f(L_i)$ is a connected component of $DM(\calS)\setminus f(\s_1)$ containing
$f(\L)$ and $f(L_{i+1})\subset f(L_i)$ it follows that either
$\bigcap_{i=1}^{\infty} \overline{f(L_i)} \subset \L\cup \mathcal{R}$ or
$\bigcap_{i=1}^{\infty} \overline{f(L_i)}$ is the closure of a connected
component of $DM(\calS)\setminus\sigma$ for some vertical sphere $\sigma\subset DM(\calS)$.
The latter cannot happen since if $\s$ is a separating sphere in the
Sierpi\'nski space $DM(\calS)$ then the closure of each component of the complement
of $\s$ is homeomorphic to the Sierpi\'nski space $\S_{n-1}$, which would contradict
(\ref{preserving-vertical-sides}) since $f(\L)$ is homeomorphic to an $n-1$ ball.
Thus, since $\L$ is connected, $f(\L)=\L$ or $f(\L)=\mathcal{R}$. The same
argument works for $\mathcal{R}$ and therefore we have that either
\begin{align*}
f(\L) &= \L \quad \mbox{and} \quad f(\mathcal{R})=\mathcal{R}, \mbox{ or}\\
f(\L) &= \mathcal{R} \quad \mbox{and} \quad f(\mathcal{R})=\L.
\end{align*}

In either case we have $\{\pi_1(\tau_0(f(\s_x))): x\in [a,b]\}=(a_1,b_1)$. In particular
for almost every $x\in[0,1]$ the vertical sphere $\s_x$ is contained in
$f(DS)$. In particular $f(DS)$ is dense and since it is closed we obtain that
$f(DS)=\overline{f(DS)}=DS$.
%
%
%
\end{proof}

\section{Modulus estimates in slit domains}\label{section:main-estimate}

In this section we formulate a general condition for the collection of slits  $\calS=\{s_i\}\subset I$ in a box $I\subset\mathbb{R}^n$,  which implies that the the family of non-vertical curves in $M(\calS)$ has a vanishing modulus. Combining with Theorem \ref{thm:mod0-co-hopf} we are able to show QS co-Hopficity for large classes of slit Seirpi\'nski spaces. In particular, we will be interested in collections of  slits $\calS$  satisfying one of the following properties:
\begin{itemize}
 \item[$(i)$] $\calS$ is uniformly relatively separated and occurs on all locations and scales, or 
\item[$(ii)$] $\calS=\calS_{\textbf{r}}$ for some $\textbf{r}\notin\ell^n$.
\end{itemize}
Note that $(i)$ is equivalent to having conditions $(\calS_1)$ and $(\calS_4)$ of Section \ref{Sec:slit-domains}, while  $\mathcal{S}_{\mathbf{r}}$ is defined in Section \ref{section:diadic-slit-spaces}. 

In this section, assuming the modulus estimates proved in Section \ref{Section:main-estimate-proof}, we prove the following.

\begin{theorem}\label{thm:main1}
If $\calS\subset I \subset\mathbb{R}^n$ is a family of slits satisfying either $(i)$ or $(ii)$ then the double $(DM(\calS),d_{\calS},\calH^n)$ of the slit Sirpi\'nski space corresponding to $\calS$  is quasisymmetrically co-Hopfian.
\end{theorem}

As mentioned in the Introduction, Theorems \ref{thm:Sierpinski-co-hopf}, \ref{thm:main-cohopf} and \ref{thm:non-self-similar-cohopf-1} follow from Theorem \ref{thm:main1}. Indeed, to obtain Theorem \ref{thm:main-cohopf} note that if $M(\calS)$ is porous then the slits $\calS$ satisfy condition $(i)$ and by Theorem \ref{thm:main1} the double of $M(\calS)$ is co-Hopfian. 
Similarly, to obtain Theorem \ref{thm:non-self-similar-cohopf-1} from Theorem \ref{thm:main1}, suppose $\calS=\calS_{\mathbf{r}}$, with $r_i\notin\ell^n$ such that $r_i\to0$. Since $(ii)$ is satisfied $DM(S)$ is co-Hopfian by Theorem \ref{thm:main1}. However, $DM(S)$ is clearly not porous since $r_i\to0$. 

Theorem \ref{thm:main1} is proved at the end of this section by combining  Theorem \ref{thm:mod0-co-hopf} with the modulus estimates obtained below, Lemmas \ref{lemma:curves_in_porous_carpets} and \ref{lemma:curves-in-diadic-carpets}.



\subsection{Modulus estimates in slit domains.}
Let $I$ be a box in $\mathbb{R}^n$, cf. Section \ref{Sec:slit-domains}, and let $L,R$ be the left and right faces of $I$, respectively, i.e.
\begin{align*}
  L&=\{a_1\}\times (a_2,b_2)\times \ldots \times(a_n,b_n),\\
  R&=\{b_1\} \times (a_2,b_2)\times\ldots \times (a_n,b_n).
\end{align*}

We say that a curve $\g:(0,1)\to X$  connects subsets $E$ and $F$ of $X$ if $E\cup F \cup \overline{\g}$ is connected, where $\overline{\g}$ is the closure of the image of $\g$ in $X$.

Given a collection of slits $\calS\subset I\subset\mathbb{R}^n$ let $\G_i=\G_i(\calS)$ and $\G_{\calS}$ be the family of curves connecting  the left face of
$I$ to the right face in the slit domains $S_i$ and the slit set $S$, respectively. More precisely, we let
\begin{align*}
\G_i = \{\g\subset S_i \,|\, \mbox{$\g$ connects $L$ to $R$ in $\bar{I}$} \}, \quad i\geq 1,
\end{align*}
and
\begin{align*}
  \G_{\calS}= \bigcap_{i=1}^{\infty} \G_i.
\end{align*}
Thus, $\G_{\calS}$ is the collection of curves $\g$ in $I\subset\mathbb{R}^n$ connecting $L$ and $R$, which avoid all the slits $s_i\in\calS$.

 The main result in this section is an estimate  on
the modulus of $\G_{\calS}$. As one may expect the modulus of $\G_{\calS}$ depends on the geometry (i.e. sizes and location) of the slits $s_i\in\calS$. To formulate our main result we need the following notation.
 Given $\eps>0$ and a slit $s\subset I$ of sidelength $l(s)>0$ such that $\pi_1(s)=x$ we let $s^{\eps}$ be the box $(x,x+\eps l(s))\times s\subset\mathbb{R}^n$. 

\begin{figure}[t]\label{fig:slit-in-3d}
\centerline{
\psfig{figure=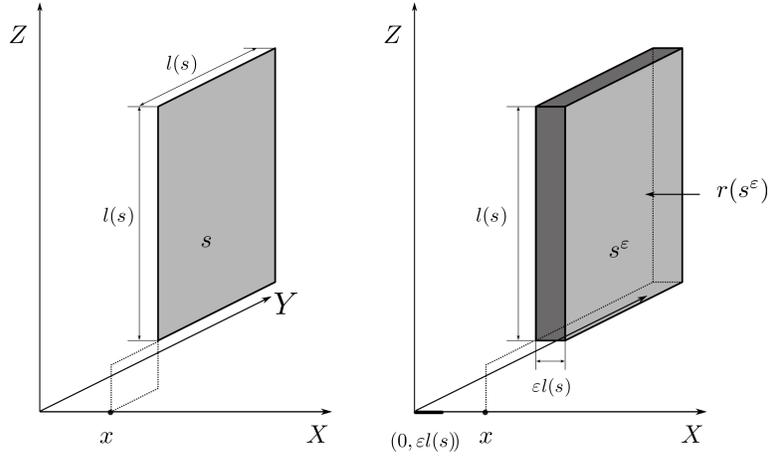,height=6cm}
}
\caption{A slit $s$ in $\mathbb{R}^3$ and its right $\eps$-collar $s^{\eps}$.}
\end{figure}

Equivalently,
\begin{align*}
  s^{\eps}= s + (0,\eps l(s))
  = \{z+t\in\mathbb{R}^n \,\,|\,\, z\in s, \,\, t\in(0,\eps l(s))\},
\end{align*}
where $(0,\eps l(s))$ denotes the corresponding interval in $\mathbb{R}^1 \cong \mathbb{R}^1\times (0,\ldots,0)$. We will call the set $s^{\eps}$ the  \emph{(right) ``$\eps$-collar" of the slit $s$}. Thus $s^{\eps}$ is an $n$-dimensional box with dimensions $\eps l(s)\times l(s) \times \ldots \times l(s)$ the left face of which coincides with $s$, see Figure \ref{fig:slit-in-3d}.

Note also,  that if $\{s_i\}$ is a uniformly relatively separated sequence in $I$ then $s_i^{\eps}\subset I$  for every $i\in\mathbb{N}$ whenever $0<\eps<\sigma$, where $\sigma$ is the separation constant in (\ref{uniform-rel-separation}).

\begin{lemma}[\textbf{Main Estimate}]\label{lemma:main_estimate}
Suppose $\calS\subset (a_1,b_1)\times \ldots (a_n,b_n)$ is a uniformly relatively separated sequence of slits in $\mathbb{R}^n$ for
which there exists $\eps>0$ such that there is a
subsequence $\calS(\eps) = I=\{s_{i_k(\eps)}\}_{k=1}^{\infty}$ such that
\begin{align}\label{condition:disjoint_collars}
s_{i_k(\eps)}^{\eps} \cap s_{i_l(\eps)}^{\eps} =\emptyset,
      \quad k\neq l.
\end{align}
Then there is a constant $C_n>0$, such that for every $p\geq 1$
\begin{align}\label{main_estimate}
\m_p \G_{\calS}\leq {(b_1-a_1)^{-p}}\left[ \calH^n\left(\calR^{\eps}\right) + C_n\calH^n(I)\eps\right],
\end{align}
where $\calR^{\eps}=I \setminus \bigcup_{k=1}^{\infty}s_{i_k(\eps)}^{\eps}.$
\end{lemma}

The proof of Lemma \ref{lemma:main_estimate} in Section \ref{Section:main-estimate-proof} will give more than stated above.
Namely, we will be able to estimate the modulus of a family of curves that is a priori larger than $\G_{\calS}$.

\begin{lemma}\label{lemma:main-estimate-lift}
Suppose the assumptions of Lemma \ref{lemma:main_estimate} are satisfied. Let $\hat{\G}_{\calS}$ be the image under $\tau_0$ of the family of all curves $\g$ in $M(\calS)$ connecting $\mathcal{L}$ and $\mathcal{R}$,
\begin{align*}
\hat{\G}_{\calS} = \{\tau_0(\g) | \, \g \subset M(\calS)  \mbox{ and } \g \mbox{ connects  $\mathcal{L}$ and $\mathcal{R}$} \}.
\end{align*}
Then
\begin{align}\label{main_estimate2}
\m_p (\hat{\G}_{\calS},\calH^n)\lesssim  \calH^n\left(\calR^{\eps}\right) + O(\eps),
\end{align} for every $p\geq 1$.
\end{lemma}


Combining Lemmas \ref{lemma:main_estimate} and \ref{lemma:main-estimate-lift} we have the following.

\begin{corollary}\label{corollary:main_estimate}
Suppose  $\calS\subset I$ is a uniformly relatively separated sequence of slits for
which there exists a sequence $\eps_j\to0$ such that (\ref{condition:disjoint_collars}) holds for $\eps=\eps_j, \forall j\geq 1$  and
$\calH^n\left(\calR^{\eps_j}\right)\to0$. Then
$\m_p \G_{\calS} = \m_p (\hat{\G}_{\calS})=0$ for every $p\geq 1$.
\end{corollary}

Thus, if for every small $\eps>0$ there is a subsequence $\{s_{i_k}\}$ of slits whose $\eps$-collars are disjoint and the union of these $\eps$-collars has full measure in $I$ then $\m_p(\G_{\calS})=0$. The proof of Lemma \ref{lemma:main_estimate} will show that one can have bounded admissible metrics for $\G_{\calS}$ supported essentially on the complement of the disjoint $\eps$-collars of slits $s_{i_k(\eps)}$.

%
%

\subsection{Non-vertical families in slit spaces}
Here we show that under the general condition of previous subsection the collection of all non vertical curves in the slit space  $M(\calS)$ also has vanishing modulus. Recall that we say that a curve $\g$ in $M(\calS)$ or $DM(\calS)$ is \emph{vertical} if $\pi_1(\tau_0(\g))$ is a point in $\mathbb{R}$, otherwise $\g$ is \emph{non-vertical}.

\begin{lemma}\label{lemma:non-vertical0}
Suppose $\calS=\{s_i\}\subset I$ is a uniformly relatively separated sequence of slits for
which there exists a sequence $\eps_j\to0$ such that (\ref{condition:disjoint_collars}) holds for $\eps=\eps_j, \forall j\geq 1$  and
$\calH^n\left(\calR^{\eps_j}\right)\to0$.
Let $\G_{nv}$ be the family of all non-vertical curves in $(M(\calS),d_{\calS},\calH^n)$ or  $(DM(\calS),d_{\calS},\calH^n)$. Then $\m_p(\G_{nv})=0,$
for all $p\geq 1$.
\end{lemma}

\begin{proof}


Let $k\geq0$ and let $\G_k$ be the family of curves $\g$ in $\G_{nv}$ such that $\calH^{1}(\pi_1(\tau_0(\g)))\geq 2^{-k}(b_1-a_1)$, i.e. those whose image $\tau_0(\g)$ oscillates in the first coordinate by at least $2^{-k}(b_1-a_1)$. Then $\G_{nv}=\cup_{k=1}^{\infty} \G_k$. Furthermore, for every $\g\in\G_k$ the projection $\pi_1(\tau_0(\g))$ contains an interval
$$J_{k,j}:=\left(a_1+\frac{b_1-a_1}{2^{k+1}} \cdot j,\, a_1+\frac{b_1-a_1}{2^{k+1}} \cdot (j+1) \right)$$
for some $j\in\{0,\ldots,2^{k+1}-1\}$. Denoting
\begin{align*}
  \G_{k,j} = \{ \g\in\G_{nv} : \pi_1(\tau_0(\g)) \supset J_{k,j} \}
\end{align*}
we can write $\G_k = \bigcup_{j=1}^{2^{k+1}} \G_{k,j}$ and thus
$\G_{nv}=\bigcup_{k=1}^{\infty}\bigcup_{j=1}^{2^{k+1}}\G_{k,j}.$
Therefore, by subadditivity of modulus it is enough to show that $\m_p(\G_{k,j})=0$ for all $k\geq 1$ and $0\leq j <2^k$.
We will show a more general fact. Namely, for every interval $J:=(\alpha,\beta)\subset(a_1,b_1)$ denoting $\G(J)=\{ \g\in \G_{nv} \, | \, \pi_1(\tau_0(\g))\supset J\}$
we will show that $\m_p(\G(J))=0$.
For this we would like to use Lemma \ref{lemma:main-estimate-lift}. However some care has to be taken since we do not know that the relative distance between the slits $s_i\in\calS$ which are contained in $I\cap\pi_1^{-1}(J)$ and the boundary of this box is bounded from below.

Let $G(J)=\tau_0(\G(J))$. We first observe that $\m_p G(J) =0$.

Suppose  $\d\geq0$ is small enough so that $J_{\d}:=(\alpha+\d,\beta-\d)\subset J$. Let
$I_{\d}= I\cap\pi_1^{-1}(\alpha+\d,\beta-\d),$
and
$\calS'=\{s\in \calS : s\in I_\delta\}.$
%
%
Let $G_{\d}(J)$ be the family of curves $\g'$ in $I\cap\pi_1^{-1}(J)$ connecting the vertical sides of that box, such that $\g'\cap I_{\d}$ is connected and $\g'\cap I_{\d}=\tau_0(\g)$ for some $\g\in\G_{nv}$. Thus $\g'$ essentially avoids the slits in $\calS'$. In other words, we disregard the slits contained in the $\d$ neighborhoods of  the left and right faces of $I_0=I\cap\pi_1^{-1}(J)$.

By overflowing property of modulus we have $\m_p(G(J))\leq \m_p(G_{\d}(J)).$
But now we may apply Lemma \ref{lemma:main_estimate} to $G_{\d}(J)$, since the collection of slits $S'$ is uniformly relatively separated in $I\cap\pi^{-1}(J)$. Therefore, if $\eps_j\to0$ is such that $\calH^n(\calR^{\eps_j})\to0$ then
\begin{align*}
\m_p(G_{\d}(J))
&\lesssim 
\calH^n\left(\calR^{\eps_j}\cap I_{\d}\right) + \calH^n(I_0\setminus I_{\d}) + C \eps_j \underset{j\to\infty} \longrightarrow 2\d \prod_{i=2}^n(b_i-a_i).
\end{align*}
Taking $\d\to0$ it follows that for every interval $J\subset\pi_1(I)$ we have $\m_p(G(J))=0$.  From Lemma \ref{lemma:Merenkov-ahlfors-regular} and the fact that  $\vartheta:DM(\mathcal{S})\to[0,1]^n$ is $1$-Lipschitz it follows that  we may apply Lemma \ref{lemma:modulus-under-projections}. Since $G(J)= \vartheta(\G(J))$, it follows that $\m_p(\G(J))=0$, for every $J\subset(a_1,b_1)$. As explained before, subadditivity implies that $\m_p\G_{nv}=0$.
\end{proof}


\subsection{Slits occuring on all locations and scales}

Recall, that we say that the slits $s_i\in I$ \emph{occur on all locations and scales in} $I\subset\mathbb{R}^n$ if there is a constant $C\geq 1$ such that for every ball $B=B(x,r)\subset I$ there is a slit $s_i\subset B$ such that $\diam s_i \geq r/C$.

\begin{lemma}\label{lemma:curves_in_porous_carpets}
  Suppose $\calS=\{s_i\}\subset I$ is a sequence of slits in $I$ which is uniformly relatively separated and occurs on all locations and scales. Then, for every $p\geq 1$, the following holds
  $$\m_p \G_{\calS} = \m_p (\hat{\G}_{\calS})=\m_p(\G_{nv})=0,$$ 
where $\G_{nv}$ is the family of non-vertical curves in $M(\calS)$ of $DM(\calS)$.
\end{lemma}

\begin{proof}
First, we enumerate the sequence $s_i$ so that $l(s_i)\geq l(s_{i+1})$ for $i\geq 1$. Next, fix $0<\eps<\sigma(\calS)$ and define the sequence $i_k=i_k(\eps)$ inductively as follows. Let $i_1=1$. For $k\geq1$ assume $i_1,\ldots,i_k$ have been defined so that the collars  $s_{i_1},\ldots s_{i_k}$ are pairwise disjoint. Note, that there is a slit $s_j$ which does not intersect the $\eps$-collars $s_{i_1}^{\eps},\ldots s_{i_k}^{\eps}$. Indeed, since the slits occur at all locations and scales we may pick a ball $B\in I\setminus(s_{i_1}^{\eps}\cup\ldots\cup s_{i_k}^{\eps})$ and a slit $s_j\subset 2^{-1}B$ such that $s_j^{\eps}\subset B$ and thus has a collar that is disjoint from the previously chosen ones. We let $i_{k+1}$ to be the smallest index, satisfying this property. More precisely we define
\begin{align*}
 i_{k+1}= i_{k+1}(\eps) = \min\{ j \, | \,\, s_j^{\eps} \cap s_{i_l}^{\eps} = \emptyset, \forall l\leq k \}.
\end{align*}
Thus, by definition condition (\ref{condition:disjoint_collars}) of Lemma \ref{lemma:main_estimate} is satisfied.

Next, we wish to estimate the $\calH^n$-measure of $\calR^{\eps}$. Fix $p\in I$ and $r>0$. Since the slits $s_i$ appear on all scales and locations there is a slit $s_i$ such that $s_i\subset B(p,r/2)$ and $l(s_i)\geq r/C$. Now, if $i={i_k(\eps)}$ for some $k$ then $s_i^{\eps}\subset (\calR^{\eps})^c.$ On the other hand if $i\neq i_{k}(\eps)$ for any $k\geq 1$ then $s_i^{\eps}$ intersects one of the collars $s_{i_k}^{\eps}$ for some $i_k<i$ and therefore $l(s_{i_k})\geq l(s_k)$. But this means that $s_{i_k}\subset (\calR^{\eps})^c$. Since $s_i^{\eps}$ has a nontrivial intersection with $s_{i_k}^{\eps}$, it follows that there is a ball $B'\subset B(p,r)$ of radius $(\eps l(s_i)/2)$ which is  contained in the collar $s_{i_k}^{\eps}$ and as such is in the complement of $\calR^{\eps}$. Thus for every ball $B(x,r)\subset I$ there is a ball $B'\subset (\calR^{\eps})^c$ of radius
$$\frac{\eps l(s_i)}{2}\geq \frac{\eps } {2C} \cdot r. $$
Since $\eps$ and $C$ are fixed constants it follows that $\calR^{\eps}$ has  no density points and therefore has zero $\calH^n$-measure (in fact $\calR^{\eps}$ is porous, but we do not need this fact). Applying Corollary \ref{corollary:main_estimate} and Lemma \ref{lemma:non-vertical0} completes the proof.
\end{proof}

\subsection{Standard non-self-similar slits}\label{Section:non-self-similar}

%

Let $\textbf{r}=\{r_i\}$ be a sequence of real numbers such that $0\leq r_i<1$. In Section \ref{Section:non-self-similar} we defined a collection $\calS_{\textbf{r}}$ of slits in $[0,1]^n$ which was used in the construction of standard slit Sierpi\'nski spaces. We will denote by $\G_{\textbf{r}}=\G_{\calS_{\textbf{r}}}$, i.e. the family of curves connecting the vertical sides of the unit cube in $\mathbb{R}^n$ which avoid $\calS_{\textbf{r}}$. We will also let $\hat{\G}_{\mathbf{r}}=\hat{\G}_{\calS_{\mathbf{r}}}$. Here we will apply Lemma \ref{lemma:main_estimate} to obtain the following result.


\begin{lemma}\label{lemma:curves-in-diadic-carpets}

If $\calS_{\textbf{r}}\subset\mathbb{R}^n$ is the standard non-self-similar collection of slits  such that $\textbf{r}\notin\ell^n$ then for every $p\geq 1$, the following holds
  $$\m_p \G_{\calS} = \m_p (\hat{\G}_{\calS})=\m_p(\G_{nv})=0,$$ 
where $\G_{nv}$ is the family of non-vertical curves in $M(\calS)$ of $DM(\calS)$.
\end{lemma}
\begin{proof}
First, note that we may assume that $r_i$ is a power or $1/2$ for every $i\geq 0$. Indeed, if $t_i$ is the largest number of the form $1/2^{m}$ which is less than $r_i$, $i\geq 1$ then $\G_{\textbf{r}} \subseteq \G_{\textbf{t}}$ and therefore
$\m_p \G_{\textbf{r}} \leq \m_p \G_{\textbf{t}}$. Thus, if  $\textbf{t}\in \ell^{2}$ then $\textbf{r}\in \ell^{2}$, since $t_i\asymp r_i$ and if we show that $\m_p \G_{\textbf{t}}=0$ then also $\m_p \G_{\textbf{r}}=0$.

Next, we choose $s_{i_j(\eps)}$ in the same way as in Lemma \ref{lemma:curves_in_porous_carpets}, thus guaranteeing that condition (\ref{condition:disjoint_collars}) is satisfied.

To estimate the measure of the residual set $\calR^{\eps}=[0,1]^n\setminus \bigcup_{k=1}^{\infty}s^{\eps}_{i_j(\eps)}$, let
\begin{align*}
  \calR_k^{\eps} = [0,1]^n \setminus \bigcup_{i=0}^{k} \bigcup_{s^{\eps}_{i_j(\eps)}\in\D_i} s^{\eps}_{i_j(\eps)}.
\end{align*}
By the disjointness property of the $\eps$ collars we have that
\begin{align*}
   \calR_k^{\eps} = [0,1]^n \setminus \bigcup_{i=0}^{k} \bigcup_{Q\in\D_i} s^{\eps}(Q).
\end{align*}
Next, we estimate the measure of $\calR^{\eps}_{k+1}$. Note that for $Q_0\in\D_0$ we have
\begin{align*}
  \calH^n(\calR^{\eps}_0) = 1-\calH^n(s^{\eps}(Q_0))=1-\eps l(s(Q_0))^n =1-\eps r_0^n.
\end{align*}
Now, if $Q\in\D_k$ for some $k>1$ then
\begin{align}\label{equality:measure-of-nsss}
  \calR^{\eps}_k \cap Q  =(\calR^{\eps}_{k-1} \cap Q) \setminus s(Q),
\end{align}
where either $s(Q)$ is contained in a previously removed collar, or it does not intersect any such collar.
Now, if $s(Q)$ is contained in a removed collar then, since $\eps$ is a power of $1/2$, $Q$ is also in the complement of $\calR_k^{\eps}$ and both sides of (\ref{equality:measure-of-nsss}) are empty. On the other hand if $s(Q)\cap\calR_{k-1}^{\eps}\neq \emptyset$ then $s(Q)\subset\calR_{k-1}^{\eps}$ and
%
%
we have
\begin{align*}
  \calH^n(\calR^{\eps}_k \cap Q)=\calH^n(\calR^{\eps}_{k-1} \cap Q)-\calH^n(s^{\eps}(Q)).
\end{align*}
But
\begin{align*}
  \calH^n(s^{\eps}(Q))=\eps l(s(Q))^n = \eps \left(\frac{r_k}{2^k}\right)^n = \eps r_k^n \calH^n(Q) \geq \eps r_k^n \calH^n(\calR^{\eps}_{k-1} \cap Q)
\end{align*}
and therefore if $s(Q)\cap \calR_{k-1}^{\eps}\neq \emptyset$ we have
\begin{align*}
  \calH^n(\calR^{\eps}_k \cap Q)\leq(1-\eps r_k^n) \calH^n(\calR^{\eps}_{k-1}\cap Q).
\end{align*}
Moreover, as explained before if $s(Q)\cap \calR_{k-1}^{\eps} = \emptyset$ then both sides of the inequality are $0$. Therefore
summing over all diadic cubed of generation $k$ we obtain $\calH^n(\calR^{\eps}_k)\leq(1-\eps r_k^n) \calH^n(\calR^{\eps}_{k-1}).$
By induction we have
$$\calH^n(\calR^{\eps}_k)\leq\prod_{i=0}^k(1-\eps r_i^n).$$
So if  $\sum_i r_i^n=\infty$ then
\begin{align*}
  \calH^{n}(\calR^{\eps}) \leq  \calH^{n}\left(\bigcap_{k=1}^{\infty}\calR_k^{\eps}\right)\leq \lim_{k\to\infty} \prod_{i=0}^k(1-\eps r_i^n)=0.
\end{align*}
Taking $\eps_j=1/2^j$ and applying Corollary \ref{corollary:main_estimate} and Lemma \ref{lemma:non-vertical0} we obtain the needed equalities.
\end{proof}

\subsection{Proof of Theorem \ref{thm:main1}}

\begin{proof}
Suppose $\calS$ satisfies $(i)$ or $(ii)$. By Lemmas \ref{lemma:curves_in_porous_carpets} and \ref{lemma:curves-in-diadic-carpets} we have that $\m_3(\G_{nv})=0$.  By Theorem \ref{thm:mod0-co-hopf} the space $DM(\mathcal{S})$ is QS co-Hopfian.
\end{proof}

\section{Proof of the main modulus estimate: Lemma \ref{lemma:main_estimate}}\label{Section:main-estimate-proof}

The idea is to construct a one parameter family of Borel subsets of $I$ such that the characteristic functions of these subsets will be admissible for $\G_{\calS}$. In Subsection \ref{section:construction} we construct a one parameter family of metrics
$\rho^{\eps}_i$ and $\rho^{\eps}$ and prove Lemma \ref{lemma:main_estimate} assuming that these metrics are admissible for $\G_i$ and $\G$. In Subsection \ref{section:admissibility} we show the admissibility of the metrics.

\subsection{Construction of the metric $\rho^{\eps}$}\label{section:construction}
Fix $0<\eps<\sigma(\calS)$. Define a subsequence $s_{i_k(\eps)}$ of
$\{s_i\}_{i=1}^{\infty}$ inductively as follows. Let $s_{i_1(\eps)}$ be a slit of the
largest length in $\calS$. Suppose $s_{i_1(\eps)},\ldots, s_{i_{k-1}(\eps)}$ have been
chosen. Let $s_{i_{k}(\eps)}$ be a slit of the largest length among all those
slits $s_i$ whose $\eps$-collars $s_i^{\eps}$ are essentially disjoint from
the $\eps$-collars of the chosen slits, i.e.
\begin{align}\label{def:allowable-sequence}
  \calH^n\left(s_{i_k(\eps)}^{\eps} \cap \bigcup_{j=1}^{k-1} s_{i_j(\eps)}^{\eps}\right) = 0.
\end{align}
Note, that in some situations the sequence $i_k(\eps)$ may be finite.
However, if the slits are dense in $I$ and $\diam(s_i)\to0$ then for small $\eps$'s  the
sequence $i_k(\eps)$ will be infinite. Next we fix $\eps$ and assume that the
sequence $s_{i}$ was chosen so that condition (\ref{def:allowable-sequence})
was satisfied to start with, i.e. $i_k(\eps)=k, \forall k\geq 1$. Thus, below
we assume that $s_{i+1}^{\eps}$ is essentially disjoint from
$\bigcup_{j=1}^{i}s_i^{\eps}$ and $l(s_{i+1})\leq l(s_{i})$ for all $i\geq
1$.

\begin{figure}[htb]
\centerline{
\psfig{figure=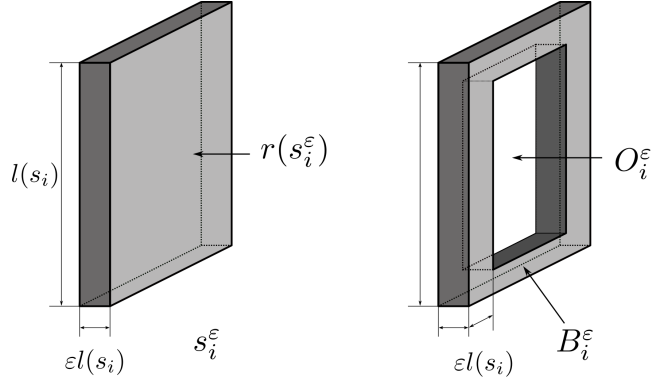,height=5cm}
}
\caption{A slit with its $\eps$-collar in $\mathbb{R}^3$ and its $\eps$-buffer and omitted regions $B_i^{\eps}$ and $O_i^{\eps}$. The dark grey part of the boundary of $s_i^{\eps}$ is the non vertical boundary $\partial_{nv}s_i^{\eps}$ while the set $O_i^{\eps}\subset s_i^{\eps}$ is the collection of points in the collar which are more than $\eps l(s_i)$ away from $\partial_{nv}s_i^{\eps}$.}
\end{figure}

Next, define two disjoint subsets of $s_i^{\eps}$. First, let $r(s_i^{\eps})$ be the right face of the $\eps$-collar $s_i^{\eps}$, or the translate of the slit $s_i$ by $\eps l(s_i)$,
$$ r(s_i^{\eps}) = s_i + \eps l(s_i).$$

We denote by $\partial_{nv} s_i^{\eps}$ the collection of all the nonvertical faces of $s_i^{\eps}$,
\begin{align*}
\partial_{nv} s_i^{\eps} = \partial s_i^{\eps}\setminus(s_i\cup r(s_i^{\eps})).
\end{align*}

Finally we let $B_i^{\eps}$ be the set of points $x$ in the $\eps$-collar of $s_i$, the distance of which from the nonvertical boundary of the collar is less than or equal to the width of the $\eps$-collar, i.e.
\begin{align*}
B_i^{\eps} & = \{x\in s_i^{\eps}\,|\, \dist(x,\partial_{nv} s_i^{\eps})\leq\eps l(s_i)\}. 
\end{align*}

Thus, $B_i^{\eps}$ is a ``rectangular annulus around the thin edge" of the collar $s_i^{\eps}$. We will call $B^{\eps}$ the \emph{$\eps$-buffer of $s_i$}. Note that if $n=2$ then $B_i^{\eps}$ is disconnected and we denote by $B_i^{\eps,+}$ and
$B_i^{\eps,-}$ the uppermost and lowermost largest ``buffer" squares
contained in $s_i^{\eps}$. More precisely, these are the squares in $I$ of
side-length $\eps l(s_i)$, whose left faces are contained in the vertical
slit $s_i$ and which contain the top and bottom endpoints of $s_i$,
respectively.

Next,  we denote
\begin{align*}
  O_i^{\eps}=s_i^{\eps}\setminus B_i^{\eps}=\{x\in s_i^{\eps}\,|\, \dist(x,\partial_{nv} s_i^{\eps})>\eps l(s_i)\}.
\end{align*} and call it the
\textit{``$\eps$-omitted region"} of $s_i$. Thus, $O_i^{\eps}$ is the open
box in $\mathbb{R}^n$ with dimensions
 $$\eps l(s_i) \times (1-2\eps)l(s_i) \times \ldots \times (1-2\eps)l(s_i),$$
whose left face is contained in $s_i$ and which is disjoint from $B_i^{\eps}$ and in particular, $O_i^{\eps}$ has the same center as $s_i^{\eps}$.

Furthermore, we let $R_i^{\eps}=I\setminus s_i^{\eps}$ and
\begin{align*}
\calB_i^{\eps} &= \bigcup_{j=1}^i B_j^{\eps},
  \quad
  \calO_i^{\eps} = \bigcup_{j=1}^i O_j^{\eps},
  \quad
  \calR_i^{\eps} = I \setminus (\calB_i\cup\calO_i).\\
\calB^{\eps}
  &= \bigcup_{j=1}^{\infty} B_j^{\eps},
  \quad
  \calO^{\eps} = \bigcup_{j=1}^{\infty} O_j^{\eps},
  \quad
  \calR^{\eps} = I \setminus (\calB\cup\calO) = \bigcap_{i=1}^{\infty}\calR_i^{\eps}.
\end{align*}

\begin{figure}[t]
\centerline{
{\psfig{figure=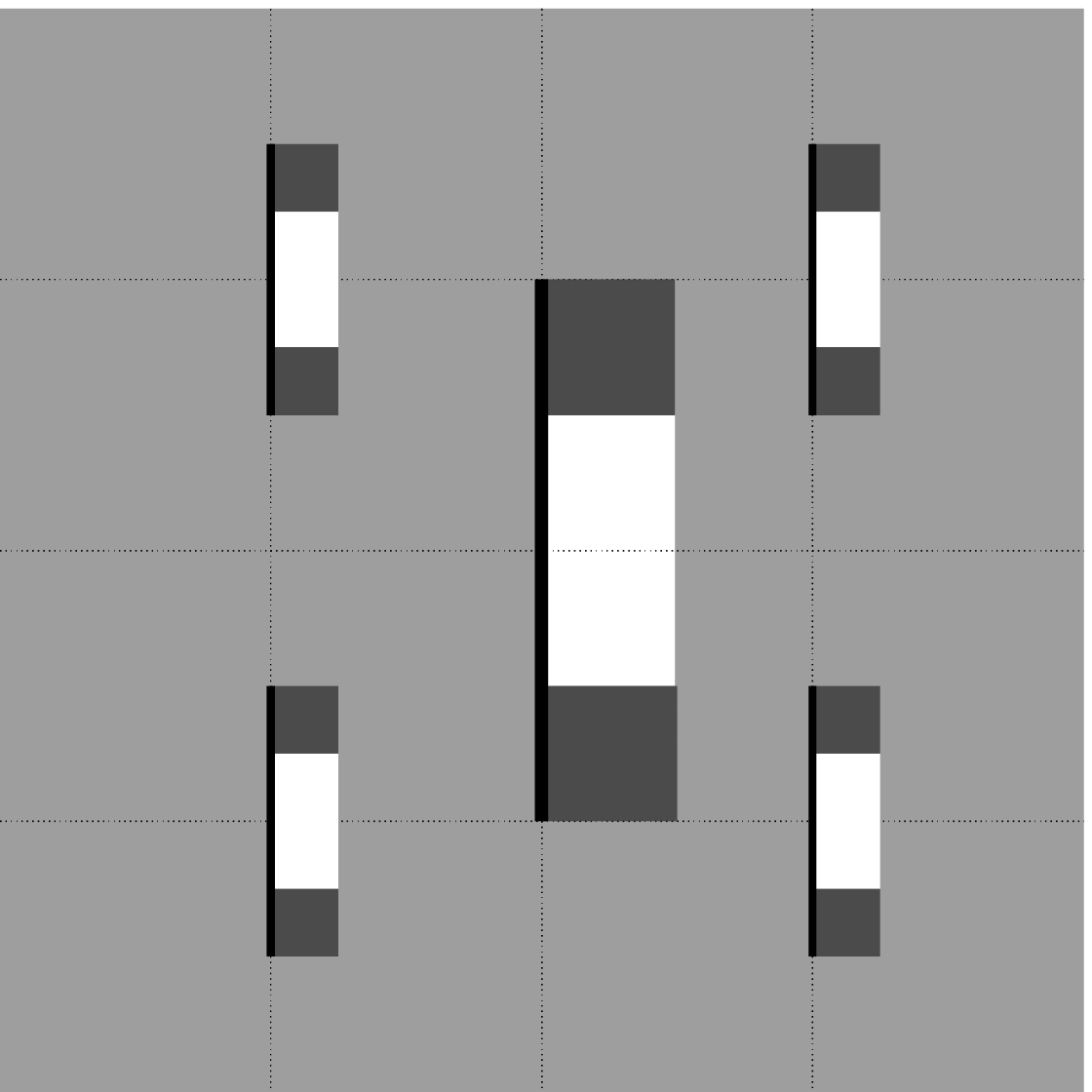,height=4cm} \qquad\qquad \psfig{figure=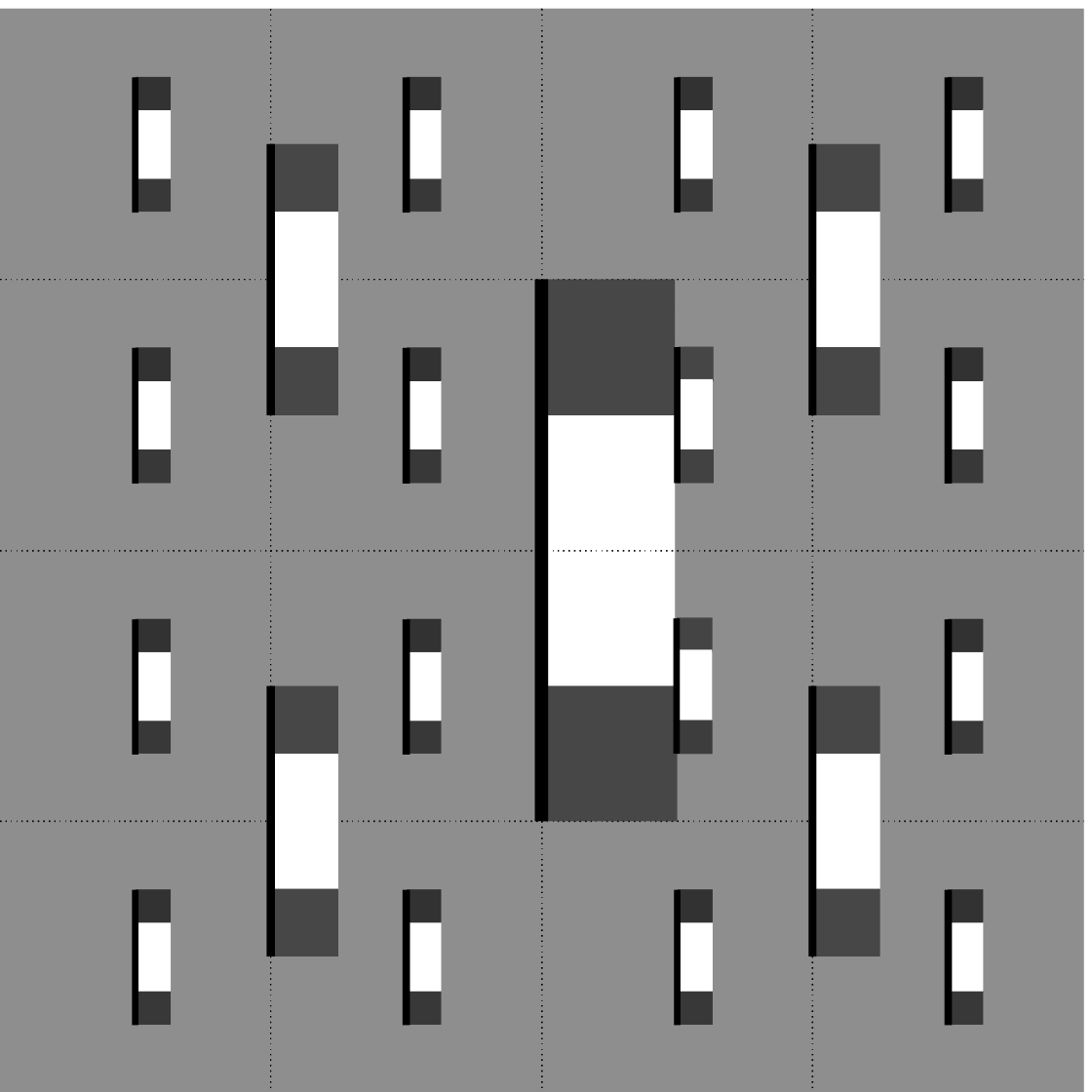,height=4cm}}
}
\caption{{Planar slit domains $S_5$ and $S_{21}$  corresponding to the standard (diadic) placement of slits with $r_i=1/2$ for all $i\geq 0$. Dark grey, light grey and white regions represent the corresponding $\eps$-buffer, omitted and residual sets, respectively for $\eps=1/4$. The sets $\calB^{\eps},\calO^{\eps}\subset [0,1]^2$ are the unions of all light grey and white regions, while the residual set $\calR^{\eps}$ is the intersection of all the light grey regions, respectively. The metric $\rho^{\eps}$ is supported on the complement of all the white rectangles, appearing in all generations.}}
\end{figure}

We will call $\calR^{\eps}$, $\calB^{\eps}$ and $\calO^{\eps}$ the
$\eps$-\textit{residual, buffer and omitted sets} corresponding to the
sequence $\{s_i\}$, respectively. Note that for every $\eps>0$ we have that
the $\eps$-residual set is the complement of all the $\eps$-collars
\begin{align*}
      \calR^{\eps} = I \setminus \bigcup_{k=1}^{\infty}  s_{i_k(\eps)}^{\eps}.
\end{align*}
Moreover, the sets $\calR^{\eps}$, $\calB^{\eps}$ and $\calO^{\eps}$  partition $I$, i.e. they are
pairwise disjoint and
\begin{align}\label{union-is-all}
\calR^{\eps}\cup \calB^{\eps} \cup \calO^{\eps} = I.
\end{align}
Finally, we define the Borel functions $\rho^{\eps}$ and $\rho_{i}^{\eps}$
for $x\in I$ by
\begin{align*}
\begin{split}
    \rho_i^{\eps}(x) &= \frac{1}{b_1-a_1}\chi_{\tiny{\calR_i^{\eps}\cup \calB_i^{\eps}}}(x) = \frac{1}{b_1-a_1}\chi_{I\setminus \calO_i^{\eps}}(x)\\
  \rho^{\eps}(x) &=\frac{1}{b_1-a_1} \chi_{\tiny{\calR^{\eps}\cup \calB^{\eps}}}(x) = \frac{1}{b_1-a_1}\chi_{I\setminus\calO^{\eps}}(x),
\end{split}
\end{align*}
where $\chi_E$ denotes the characteristic function of a set $E\subset\mathbb{R}^n$.

Below we will show that $\rho^{\eps}$ is admissible for $\G_{\calS}$ for every $\eps>0$. Next, we
assume this is true and complete the proof of Lemma
\ref{lemma:main_estimate}.

\begin{proof}[Proof of Lemma \ref{lemma:main_estimate}]
By Lemma \ref{lemma:admissible} it is enough to show that $\int_{I}
(\rho^{\eps})^p d\calH^n \to 0$ as $\eps\to0$. From the definition of
$\rho^{\eps}$ we have
\begin{align}\label{mass_estimate1}
  \int_{I}
(\rho^{\eps})^p d\calH^n =  \int_{I}
(\chi_{\calR^{\eps} \cup \calB^{\eps}})^p d\calH^n
&= (b_1-a_1)^{-p}(\calH^n (\calR^{\eps}) + \calH^{n} (\calB^{\eps})).
\end{align}
%
%
On the other hand, since
$$\calH^{n}(O_i^{\eps}) = \eps l(s_i) \times (1-2\eps)^{n-1} l(s_i)^{n-1} = (1-2\eps)^{n-1} \calH^n(s_i^{\eps}),$$
we have
\begin{align*}
   \calH^{n}(B_i^{\eps})
   &= \calH^n(s_i^{\eps}) - \calH^{n}(O_i^{\eps})
   = (1-(1-2\eps)^{n-1})\calH^n(s_i^{\eps})
\end{align*}
and therefore, since $s_i^{\eps}$'s are pairwise essentially disjoint, we obtain
\begin{align}\label{ineq:mass-of-buffers}
  \begin{split}\calH^{n}(\calB^{\eps}) &=\sum_{i=1}^{\infty}\calH^n(B_i^{\eps})\\
  &=
  (1-(1-2\eps)^{n-1}) \calH^n \left(\bigcup_{i=1}^{\infty} s_i^{\eps}\right) \\
  &\leq (2(n-1)\eps+o(\eps)) \calH^n(I)\\
  &\leq C(n)\calH^n(I) \eps
  \end{split}
\end{align}
as $\eps\to0$. Combining (\ref{mass_estimate1}) and (\ref{ineq:mass-of-buffers}) we obtain the main estimate (\ref{main_estimate}).
\end{proof}

\subsection{Admissibility of $\rho^{\eps}$.}\label{section:admissibility}

In this section we prove that $\rho^{\eps}$ is an admissible metric for the family of curves connecting the left and right faces of $I$ and which avoid the vertical slits $\{s_i\}\subset I$.

\begin{lemma}
  For every $x\in[0,1]^d$ we have
  \begin{align}\label{convergence-metrics}
    \lim_{i\to\infty}\rho_{i}^{\eps}(x) = \rho^{\eps}(x).
  \end{align}
\end{lemma}
\begin{proof}
If $x\in\calR^{\eps}$ then $x\in\calR^{\eps}_i$ and $\rho_{i}^{\eps}(x) = (b_1-a_1)^{-1}$ for every $i\geq
0$. If $x\in\calB^{\eps}$ then there is an $i_0$ such that $x\in\calB_i^{\eps}$ for all
$i\geq i_0$, since $\calB^{\eps}_i$ is an increasing sequence of open sets. Therefore
$\rho_{i}^{\eps}(x)=(b_1-a_1)^{-1}$ for $i\geq i_0$. Thus if $x\in \calR^{\eps}\cup\calB^{\eps}$  then
$\rho_{i}^{\eps}(x)\to (b_1-a_1)^{-1}=\rho^{\eps}(x)$ as $i$ approaches $\infty$. The case, when
$x\in\calO^{\eps}$ is done the same way as $x\in\calB^{\eps}$ and
(\ref{convergence-metrics}) follows.
\end{proof}
By dominated convergence theorem we immediately obtain the following approximation result.
\begin{corollary} With the notation as above we have the following.
\begin{itemize}
\item[\textit{i}.] For every $p>0$ we have
\begin{align*}
\lim_{i\to\infty}\int_{I} (\rho^{\eps}_{i})^p d\calH^n = \int_{I} (\rho^{\eps})^p d\calH^n.
\end{align*}

\item[\textit{ii}.] For every locally rectifiable curve $\g\subset I$ we
    have
\begin{align*}
\lim_{i\to\infty}\int_0^1 \rho_{i}(\g(t))|\g'(t)| dt = \int_0^1 \rho(\g(t))|\g'(t)| dt.
\end{align*}
\end{itemize}
\end{corollary}

Next lemma in the main result of this section.

\begin{lemma}\label{lemma:admissible}
Suppose $\calS=\{s_i\}\subset I$ is a collection of slits, which are uniformly relatively separated.
Then for every $0<\eps<\sigma(\calS)$, where $\s(\calS)$ is the separation constant in \ref{uniform-rel-separation}, the metric $\rho^{\eps}$ is admissible for $\G_{\calS}$.
\end{lemma}

\begin{proof}
To simplify the notation we let $\rho$ and $\rho_{i}$ denote the metrics $\rho^{\eps}$ and
$\rho^{\eps}_i$, respectively. Furthermore, for a (locally-rectifiable) curve
$\g\in\G$ we denote by $l(\g)$ and $l_i(\g)$ the $\rho$ and $\rho_i$-length of $\g$, i.e.
\begin{align*}
l_i(\g) : = \int_0^1 \rho_{i}(\g(t))|\g'(t)| dt, \quad i=1,2,\ldots.
\end{align*}
From the construction of $\rho_i^{\eps}$'s it follows that $l_i(\g)$ is a
decreasing (non-increasing) sequence. Given a curve $\g\in\G_{\calS}$ we want to show
that $l(\g)\geq1$. By (\ref{convergence-metrics}) it is enough to
show that $l_{i}(\g)\geq 1$ for all $i\geq 1$. We prove this by induction.

We
will assume that $\g:[0,1]\to I$ is oriented so that $\g(0)\in L$ and
$\g(1)\in R$, i.e. $\g$ ``starts" on the left face of $I$ and ``ends" on
the right face. Thus, given disjoint subsets $E$ and $F$ of $I$ we will say
that \emph{$\g$ meets $E$ before $F$} if there exists $t\in(0,1)$ such that
$\g(t)\in E$ and $\g(s)\cap F =\emptyset$ for any $0<s<t$. In particular, if
$\g$ intersects $E$ but not $F$ we will still say that $\g$ meets $E$ before
$F$.

Recall, that $r(s_i^{\eps})$ is the right face of $s_i^{\eps}$. Now, for every $\g\in\G_{\calS}$ let
\begin{align*}
\begin{split}
N_1^{\g}& = \{ i\in\mathbb{N} : \g\cap O_i^{\eps} \neq \emptyset\},\\
N_2^{\g}& = \{ i\in\mathbb{N} : \g \mbox{ meets $O_i^{\eps}$ before $r(s_i^{\eps})$} \},\\
N_3^{\g}& = \{ i\in\mathbb{N} : \g \mbox{ meets  $r(s_i^{\eps})$ before $O_i^{\eps}$}\}.
\end{split}
\end{align*}
Note, that if $i\in N_2^{\g}$ then $\g\cap B_i^{\eps}$ has a connected
component connecting $\partial_{nv}s_i^{\eps}$ and $\partial(O_i^{\eps})$ in the buffer region $B_i^{\eps}$ (for $n=2$ there is a component connecting the top of a buffer square $B_j^{\eps,+}$ or
$B_j^{\eps,-}$ to its bottom). Therefore,
\begin{align}
  l_i(\g\cap s_j^{\eps}) = \calH^1(\g\cap B_i^{\eps}) \geq \dist(\partial_{nv}s_i^{\eps},O_i^{\eps}) = \eps l(s_i).
\end{align}

Next, we denote by $t(s_j^{\eps})$ any ``horizontal" interval, i.e. one which is parallel to the first coordinate axis, which is contained in the top face of the $\eps$ collar $s_j^{\eps}$ and connects the vertical faces of $s_j^{\eps}$. For every $\g\in\G_{\calS}$ we inductively define a sequence of not necessarily
connected subsets $\g_i\subset I$ as follows. Since, $t(s_j^{\eps})\subset \calB^{\eps}$ we have
\begin{align}
  l_j (t(s_j^{\eps})) = \eps l(s_i), i=1,2,\ldots.
\end{align}

\begin{figure}[htb]
\centerline{
{\psfig{figure=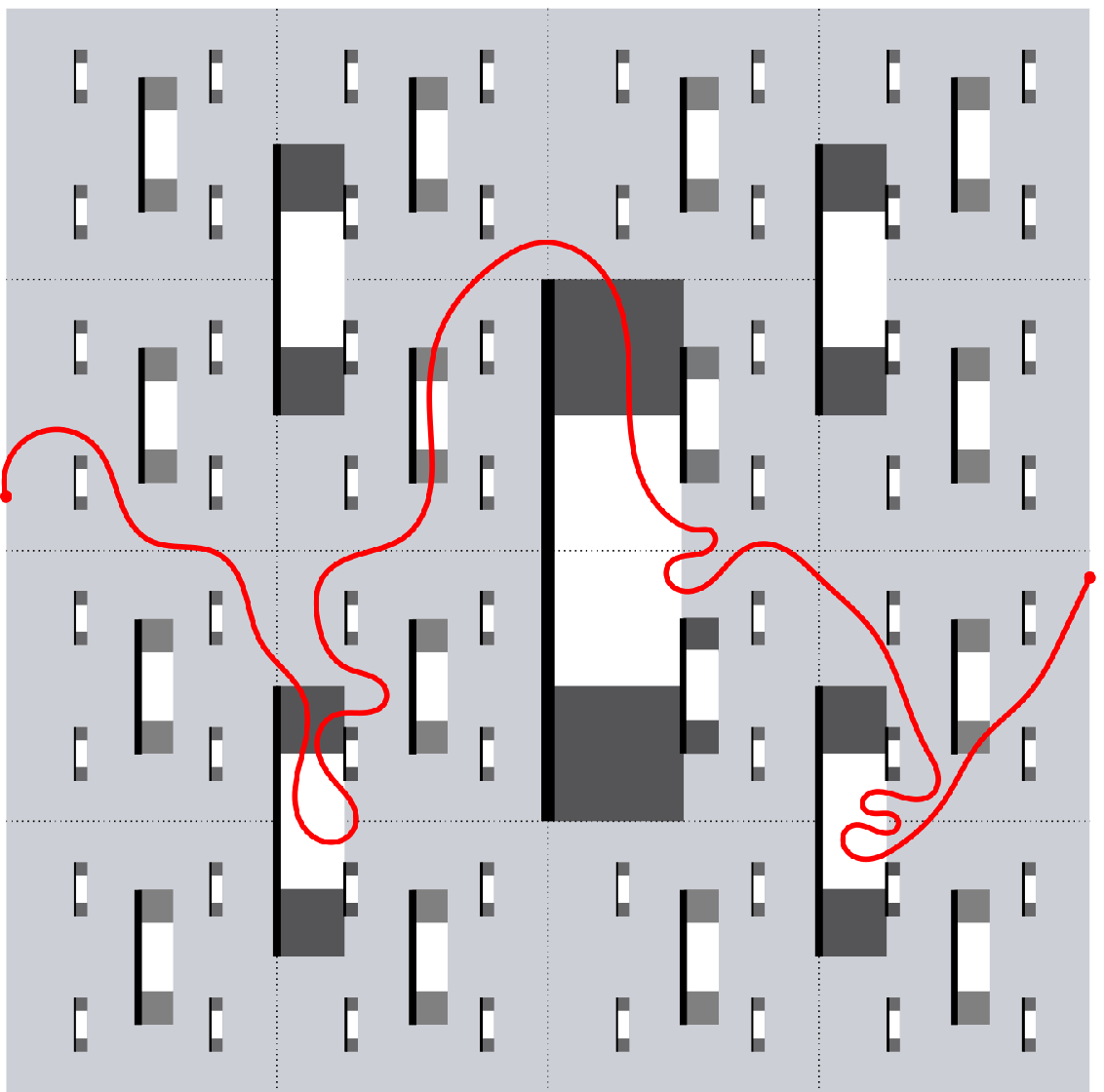,height=3cm} \qquad  \psfig{figure=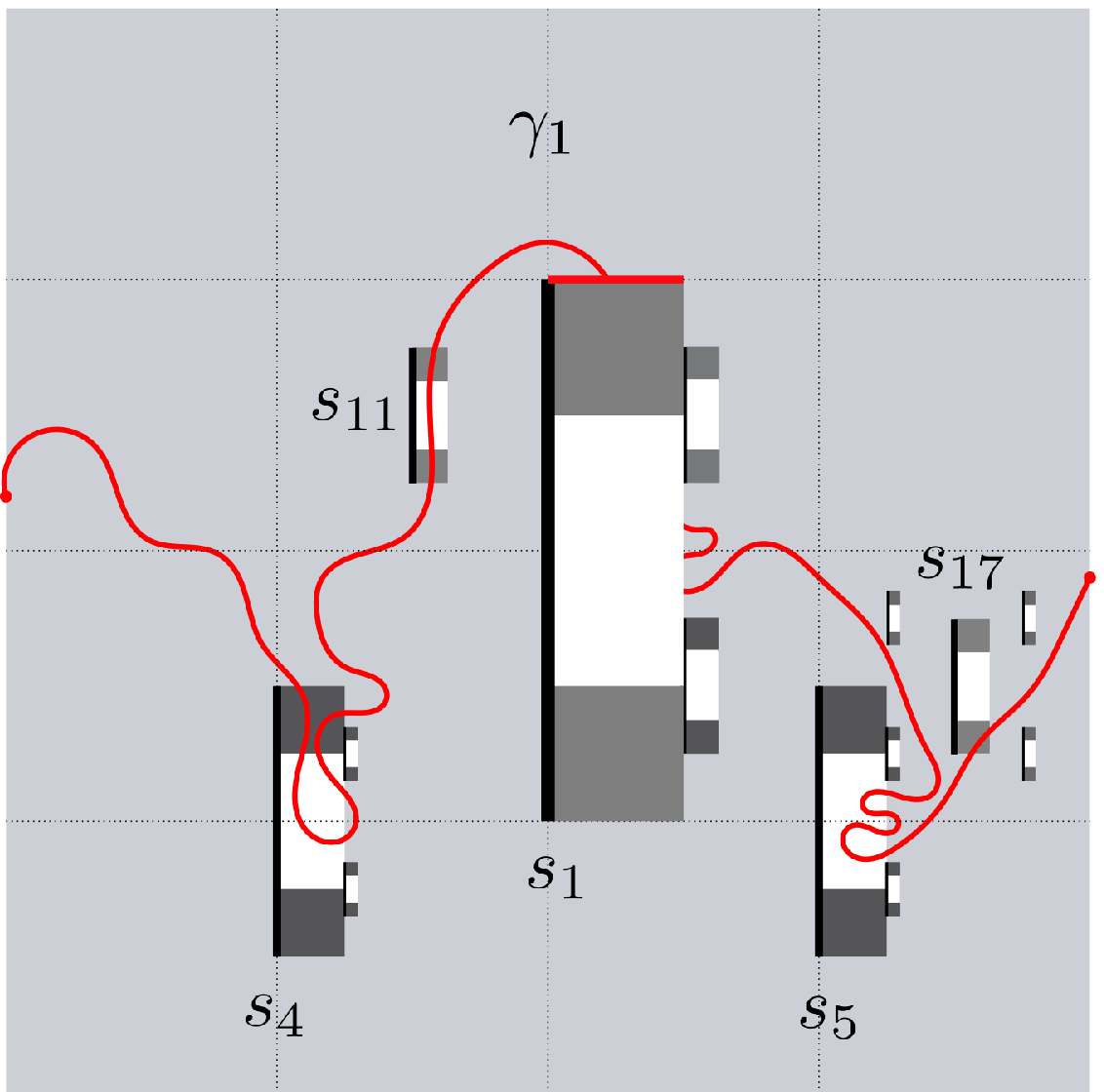,height=3cm} \quad
\psfig{figure=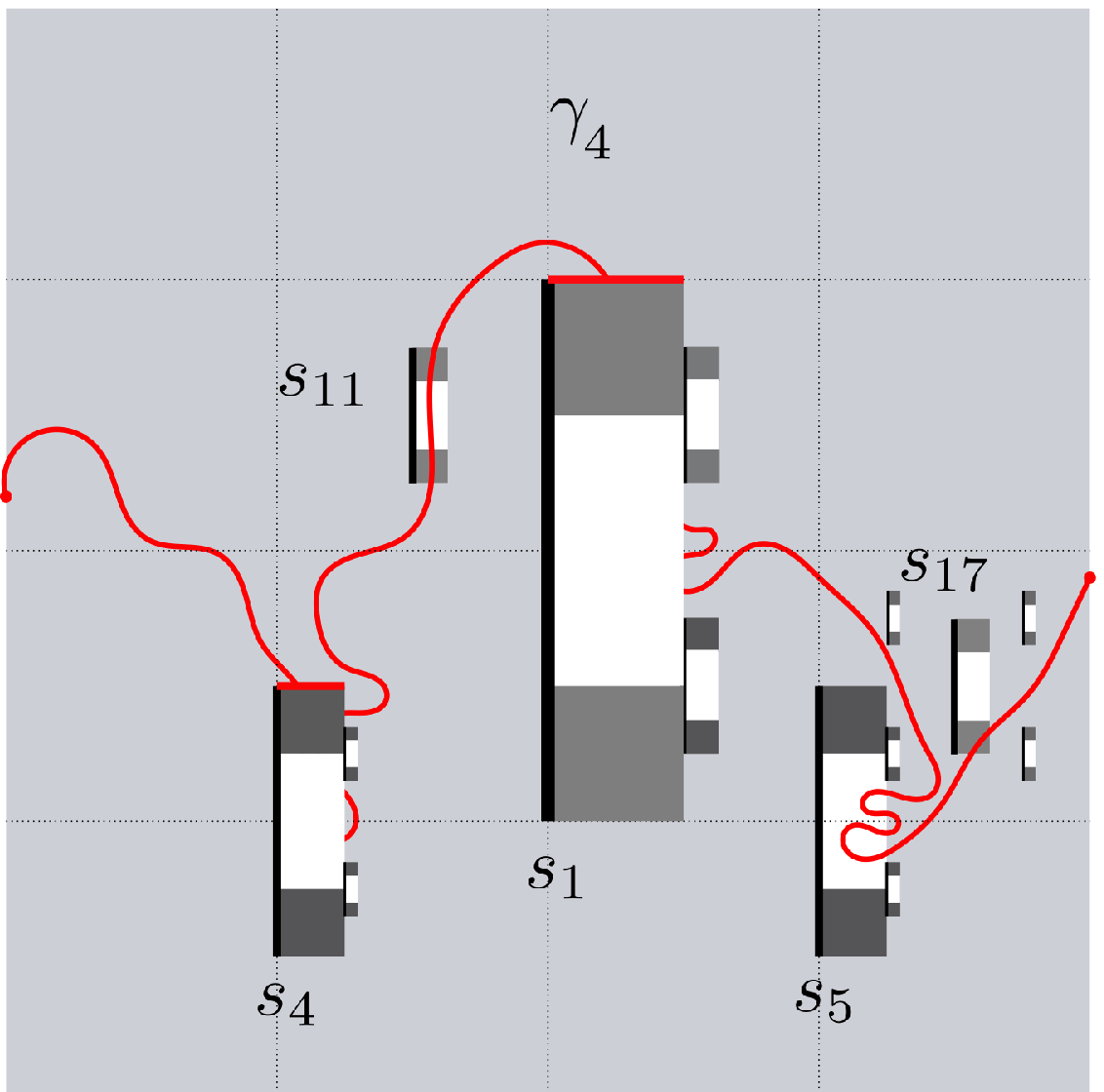,height=3cm} \quad \psfig{figure=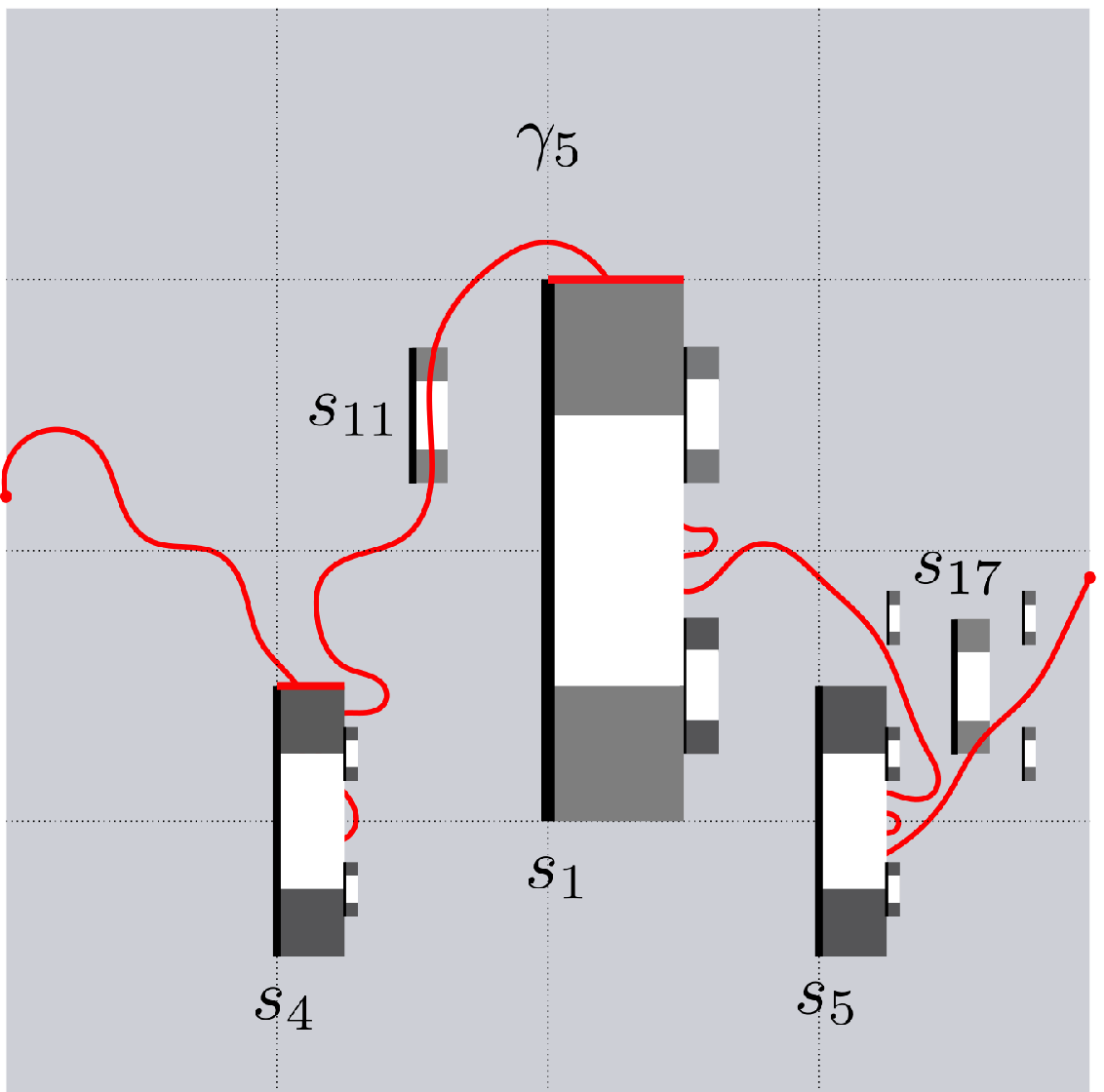,height=3cm} \quad 
}
}
\caption{On the left $\g\in\G_{\calS}$ is a curve in a slit domain. and on the right only the collars of the slits which $\g$ intersects are drawn (the numbers correspond to the number of the diadic square in which the slit is located). Since $\g$ meets the omitted region in the first collar $s_1^{\eps}$ before meeting its right edge, $\g_1$ is obtained by removing from $\g$ the collar $s_1^{\eps}$ and attaching the top $t(s_1^{\eps})$. Further modifications of the curve $\g$ show that if $\g_k$ does not intersect any of the omitted regions it follows that $\rho_k$-length of $\g_k$ is a multiple of its Euclidean length.}
\end{figure}
Let $\g_0 = \g$. Suppose, for
$i\geq 1$ the set $\g_{i-1}$ has been defined. Then, let
\begin{align*}
\g_i =
\begin{cases}
\g_{i-1} \quad & \mbox{ if } i\in N_1^{\g},  \\
\g_{i-1} \setminus s_i^{\eps} \cup t(s_i^{\eps}) & \mbox{ if } i\in N_2^{\g},\\
\g_{i-1} \setminus O_i^{\eps} & \mbox{ if } i\in N_3^{\g}.
\end{cases}
\end{align*}
Equivalently, $\g_i$ may be defines as follows,
\begin{align}
\g_i =
\begin{cases}
\g_{i-1}, \quad & \mbox{ if } \g\cap O_i^{\eps} = \emptyset,  \\
\g_{i-1} \setminus s_i^{\eps} \cup t(s_i^{\eps}), & \mbox{ if $\g$ meets $O_i^{\eps}$ before $r(s_i^{\eps})$},\\
\g_{i-1} \setminus O_i^{\eps}, & \mbox{ if $\g$ meets $O_i^{\eps}$ after $r(s_i^{\eps})$}.
\end{cases}
\end{align}


%
%
Thus, $\g_i$ is obtained from $\g$ by removing all the omitted rectangles
$O_j^{\eps}$ which $\g$ meets after $r(s_j)$ and if $\g\cap B_i^{\eps}$ has a
component connecting $\partial_{nv}s_i^{\eps}$ and $\partial(O_i^{\eps})$  we replace
$\g\setminus s_j^{\eps}$ with a horizontal interval in the buffer of length $\eps l(s_j)$ . From
this description it follows that $\g_i \subset \calR_i^{\eps}\cup\calB_i^{\eps}, \forall
i\geq1$ and in particular
\begin{align}\label{lengths-of-broken-curves}
\begin{split}
l_i(\g_i) &= (b_1-a_1)^{-1}\calH^1(\g_i\cap(\calR_i^{\eps}\cup\calB_i^{\eps})) \\
&= (b_1-a_1)^{-1}\calH^1(\g_i), \forall i\geq 1.
\end{split}
\end{align}

Given all the definitions above, Lemma \ref{lemma:admissible} follows easily from the two lemmas below.
\begin{lemma}\label{lemma:breaking-up}
If $\g\in \G$ is a locally-rectifiable curve then
\begin{align}\label{admissibility:breaking-up}
l_{i}(\g)
\geq l_{i}(\g_{i}), \quad \forall i\geq 1.
\end{align}
\end{lemma}
\begin{lemma}\label{lemma:projection}
  If $\g\in\G$ then
  \begin{align}\label{admissibility:projection}
\pi_1(\g_i)=[a_1,b_1], \quad \forall i\geq 1.
  \end{align}
\end{lemma}
Before proving these results, we first complete the proof of Lemma \ref{lemma:admissible}.
%
For this we estimate the $\rho_i$ length of $\g$ from below
as follows,
\begin{align}
l_i(\g)
& \geq l_i(\g_i) \tag{by (\ref{admissibility:breaking-up})}\\
&= (b_1-a_1)^{-1}\calH^1(\g_i) \tag{$\g_i\subset(\calO_i^{\eps})^c$}\\
&\geq (b_1-a_1)^{-1}\calH^1(\pi_1(\g_i)) \tag{$\pi_1$ is $1$-Lipschiz} \\
&= (b_1-a_1)^{-1}\calH^1([a_1,b_1]) = 1. \tag{by (\ref{admissibility:projection})}
\end{align}
As was noted in the beginning of the proof, the last estimate implies that $l(\g)\geq 1$ whenever $\g\in\G_{\calS}$.
\end{proof}

Next, we prove Lemmas \ref{lemma:breaking-up} and \ref{lemma:projection}.

\begin{proof}[Proof of Lemma \ref{lemma:breaking-up}]
By the definition of $\rho_i$ we have
  \begin{align*}
    l_i(\g) \geq \calH^1 (\g\cap \calO_i^{\eps}) = \calH^1(\g\cap \calR_i^{\eps}) + \calH^1(\g\cap \calB_i^{\eps}).
  \end{align*}
Since $\g\cap\calR_i^{\eps}=\g_i\cap\calR_i^{\eps}$, we have
\begin{align*}
  l_i(\g) &\geq \calH^1(\g\cap \calR_i^{\eps}) + \calH^1(\g\cap \calB_i^{\eps})
  \geq \calH^1(\g_i\cap \calR_i^{\eps}) + \sum_{j=1}^i \calH^1(\g\cap B_j^{\eps}).
\end{align*}
Next, we note that
\begin{align*}
\calH^1(\g\cap B_j^{\eps}) \geq \calH^1(\g_i \cap B_j^{\eps}), \quad \forall j\leq i.
\end{align*}
Indeed, we have the following three cases:

\begin{itemize}
  \item[-] If $\g$ does not meet $O_j$ then $\g$ remains unchanged in $B_j^{\eps}$
      and $\g\cap B_j^{\eps} = \g_i\cap B_j^{\eps}$. In particular $\calH^1(\g\cap B_j^{\eps}) =
      \calH^1(\g_i\cap B_j^{\eps})$.
  \item[-] If $\g$ meets $O_j^{\eps}$ before $r(s_j^{\eps})$ then $\g$ connects $\partial_{nv}s_i^{\eps}$ and $\partial(O_i^{\eps})$
   and therefore $\calH^1(\g\cap B_j^{\eps}) \geq \eps l(s_j^{\eps}) = \calH^1(\g_i\cap B_j^{\eps}).$
  \item[-] If $\g$ meets $O_j^{\eps}$ after $r(s_j^{\eps})$ then $\calH^1(\g\cap B_j^{\eps})\geq
      0 = \calH^1(\g_i\cap B_j^{\eps})$.
\end{itemize}
Therefore,
\begin{align*}
 l_i(\g)
 & \geq \calH^1(\g_i\cap \calR_i^{\eps}) + \sum_{j=1}^i \calH^1(\g\cap B_j^{\eps})\\
 & \geq \calH^1(\g_i\cap \calR_i^{\eps}) + \sum_{j=1}^i \calH^1(\g_i\cap B_j^{\eps})\\
 & = \calH^1(\g_i\cap\calR_i^{\eps}) + \calH^1(\g_i\cap\calB_i^{\eps})
  = l_i(\g_i). \qedhere
\end{align*}
\end{proof}
%

\begin{proof}[Proof of Lemma \ref{lemma:projection}]
Since all the omitted regions $O_i^{\eps}$ are compactly contained in
$I$ and since $\g(0)\in L$ and $\g(1)\in R$, it follows that
$\{a_1\},\{b_1\}\in\pi_1(\g_i)$ for every $i\geq 1$.

For the rest of the proof fix $x\in(a_1,b_1)$. We need to show that
$x\in\pi_1(\g_i)$ for all $i\geq 1$. For this, define $t_x = \min\{t \,|\, \pi_1(\g(t)) = x \}$
and note that if $y > x$ then $\g([0,t_x])\cap \pi_1^{-1}(y) =\emptyset.$

Now, let $i\geq 1$. Then, either $\g(t_x)\in I$ does not belong to an
omitted region $O_j^{\eps}$ for any $j\leq i$ or it belongs to exactly one
such region, since the omitted regions are pairwise disjoint. If
$\g(t_x)$ does not belong to an omitted region then $\g(t_x)\in\g_j$ for
all $j\leq i$ and in particular $x\in\pi_1(\g_i)$. On the other hand, if
$\g(t_x)\in O_{j_0}^{\eps}$ for some $j_0\leq i$ then $\g([0,t_x])$ does not
intersect the vertical hyperplane containing $r(s_{j_0}^{\eps})$, since it is located ``to
the right of $x$". In particular $\g$ meets $O_{j_0}^{\eps}$ before
$r(s_{j_0}^{\eps})$. It follows then from the definition of $\g_j$'s that we
have $t(s_{j_0}^{\eps})\subset \g_{j_0}$. Moreover, since
$t(s_{j_0}^{\eps})$ belongs to a buffer region it remains in the curves $\g_j$ for all $j\geq j_0$ once it
is added. Therefore,
\begin{align*}
x\in\pi (O_{j_0}^{\eps}) = \pi_1(t(s_{j_0}^{\eps})) \subset \pi_1(\g_j)
\end{align*}
for every $j\geq j_0$. Thus $x\in\pi_1(\g_i)$ which completes the proof.
%
%
\end{proof}

\begin{proof}[Proof of Lemma \ref{lemma:main-estimate-lift}] Let $\eps$ and $\rho^{\eps}$ be defined as in Subsection \ref{section:construction}. Define the Borel function $\tilde{\rho}^{\eps}: M(\calS)\to[0,\infty)$ as follows:
\begin{align*}
\tilde{\rho}^{\eps}(x) = 
\begin{cases}
\rho\circ \pi_1(x), & \mbox{ if } x \mbox{ belongs to a peripheral sphere of } M(\calS),\\
1, & \mbox{ otherwise.}
\end{cases}
\end{align*}
From the proof of admissibility of $\rho^{\eps}$ above, and the fact that $\pi_1$ is $1$-Lipschitz it follows that $\tilde{\rho}^{\eps}$ is admissible for $\hat{\Gamma}_{\calS}$. Moreover, since the $\calH^{n}$-measure of all the peripheral spheres is $0$, it follows from Lemma \ref{lemma:Merenkov-ahlfors-regular} and  inequality (\ref{main_estimate}) that 
\begin{align*}
\int_{M(\calS)} (\tilde{\rho}^{\eps})^p d\calH^{n} \lesssim \int_{[0,1]^n} (\rho^{\eps})^p d \calH^n \lesssim \calH^n(\calR^{\eps}) + C_n \calH^n(I) \eps, 
\end{align*} 
where the constants depend only on the constant in (\ref{ineq:Merenkov_measures_comparable}) and $p\geq 1$. 
\end{proof}

\section{Slit Menger curve: Definition and Properties}\label{section:slit-menger-definition}

In this section we construct a metric space $\mathfrak{M}$, which we call a \textit{slit Menger curve}, and establish some of its properties. In particular we use a classical theorem of Anderson to show that $\mathfrak{M}$ is homeomorphic to the Menger curve.

\subsection{Standard Menger curve and Anderson's theorem}




Recall that the classical Menger curve is the compact subset $\mathscr{M}\subset \mathbb{R}^3$ which is constructed as follows. Let $E_0$ be the unit cube $[0,1]^3\subset\mathbb{R}^3$. To define $E_1$ divide $E_0$ into $3^3$ disjoint cubes of sidelength $1/3$ and remove those which do not intersect the one dimensional edges of the boundary of $E_0$. Thus, we remove the interiors of the central cube and the $6$ cubes which intersect the middle squares of the $6$ faces of $\partial{E_0}$. In particular, $E_1$ is the union of $20=3^3-7$ triadic cubes of generation $1$. Continuing by induction, suppose that $E_n$ has been defined for some $n\geq 1$ and is a union of triadic cubes of $[0,1]^3$ of generation $n$. To obtain $E_{n+1}$, from every triadic cube $T\subset E_n$ we remove the central subcube of $T$ and the $6$ subcubes that intersect the middle squares of the faces of $\partial{T}$ of generation $n+1$. The Menger curve is defined as $\mathscr{M}:=\bigcap_{n=0}^{\infty} E_n$.

The following theorem of Anderson provides a characterization of the Menger curve and will be used below. Before formulating it we recall some topological definitions.

A topological space $X$ is \emph{locally connected} if for every $p\in X$ and every open set $U\supset p$ there is an open connected set $N\subset U$ such that $p\in N$. The point $p\in X$ is a \emph{local cut point} of $X$ if there is an open subset $U$ containing $p$ such that $U\setminus p$ is not connected.

A covering $\mathcal{B}$ of a space $X$ is said to be of order at most $m+1$ if every point of $X$ belongs to at most $m+1$ elements of $\mathcal{B}$. The \emph{topological dimension} of $X$, denoted by $\dim_{top}(X)$, is the smallest number $m$ such that every open cover $\mathcal{B}$ of $X$ has a refinement $\mathcal{B}'$ of order at most $m+1$. Recall, that a cover $\mathcal{B}'$ is a \emph{refinement} of $\mathcal{B}$ if for every $B'\in\mathcal{B}'$ there is a $B\in\mathcal{B}$ such that $B'\subset B$.

\begin{theorem}[Anderson \cite{Anderson-homogeneity}]\label{thm:Anderson}
A compact connected space $X$ of topological dimension $1$ is homeomorphic to the Menger curve $\mathscr{M}$ if and only if
\begin{itemize}
  \item[1.] $X$ is locally connected,
  \item[2.] $X$ has no local cut points,
  \item[3.] no open subset of $X$ is planar.
\end{itemize}
\end{theorem}

Note that, to show that a compact space $X$ is homemorphic to the Menger curve one has to show that $\dim_{top}X=1$. For that we will need the following fact.

\begin{lemma}\label{lemma:top-dimension-estimate}
Suppose $X$ is a compact topological space. If for every $\eps>0$ there is a covering $\mathcal{B}$ of $X$ such that every point $p\in X$ belongs to at most $m+1$ elements of $\mathcal{B}$ then $\dim_{top}X\leq m$.
\end{lemma}

The statement above is well known and easily follows from the definition of topological dimension given above and the Lebesgue number lemma, cf. \cite{HurWal,Munkres}.

\subsection{Constructing the Slit Menger curve and its doubles} As a first step we construct a sequence of domains $W_i, i=0,1,\ldots,$ in $W_0=(0,1)^3$, such that $W_{i+1}\subset W_i$. Each $W_{i+1}$ is obtained from $W_{i}$ by removing a compact subset $E_i\subset W_{i}$, where $E_i$ is a union of scaled copies of a fixed subset $E_0$ of $W_0$.

Just like before, let $\pi_j:\mathbb{R}^3\to \mathbb{R}$ be the projection onto the $j$-th coordinate axis, $j\in\{1,2,3\}$.
Recall that $\D_i^{(n)}$ and $\D^{(n)}$ (or simply $\D_i$ and $\D$ if the dimension $n$ is clear from the context) denoted the collections of all dyadic cubes in $(0,1)^n$ of generation $i\geq 0$ and of all generations, respectively. Furthermore, if  $Q\subset\mathbb{R}^3$ is a {dyadic cube of generation $n$} such that
\begin{align*}
  \bar{Q} = \left[\frac{a}{2^n},\frac{a+1}{2^{n}}\right]\times\left[\frac{b}{2^n},\frac{b+1}{2^{n}}\right]\times
  \left[\frac{c}{2^n},\frac{c+1}{2^{n}}\right],
\end{align*}
for some $a,b,c\in\mathbb{Z}$ we define the similarity transformation $T_Q$ of $\mathbb{R}^3$ as follows
\begin{align*}
  T_Q(\textbf{x}) = \frac{\textbf{x}}{2^n} +\frac{(a,b,c)}{2^n},
\end{align*}
for $\textbf{x}\in\mathbb{R}^3$. Note, then that $Q=T_Q([0,1]^3)$ and $T_Q(0,0,0)=(a,b,c)$.


Next, we let $p:=(\frac{1}{2},\frac{1}{2},\frac{1}{2})$ and define two closed subset of $[0,1]^3$ as follows:
\begin{align*}
\begin{split}
  E' &= \left\{(x,y,z)\in[0,1]^3 : \left|z-\frac{1}{2}\right|\leq \frac{1}{2^{2}}\right\} \cap \pi_2^{-1}(p),\\
   E''&= \left\{(x,y,z)\in[0,1]^3 : |y-\frac{1}{2}|\leq \frac{1}{2^{2}} \mbox{ or } |x-\frac{1}{2}|\leq\frac{1}{2^{2}}\right\} \cap \pi_1^{-1}(p),
\end{split}
\end{align*}
\begin{figure}[htb]
\centerline{
\psfig{figure=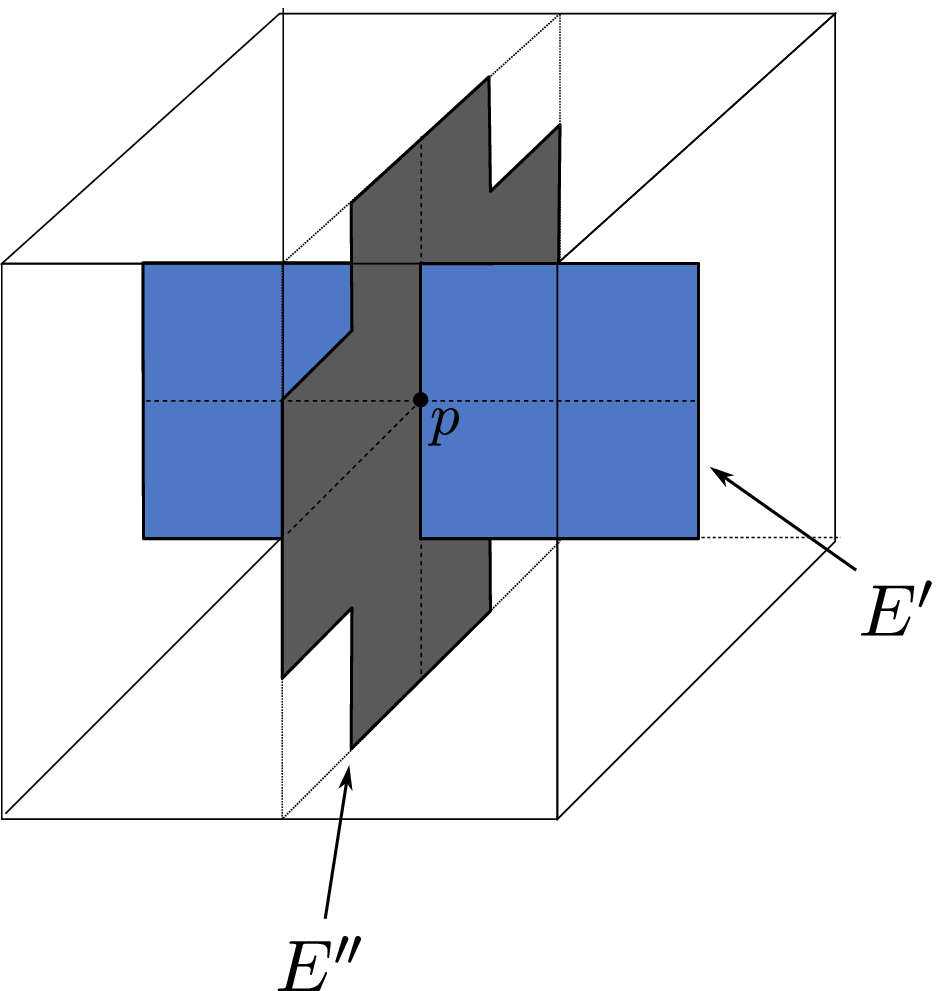,height=6cm}
\quad
\psfig{figure=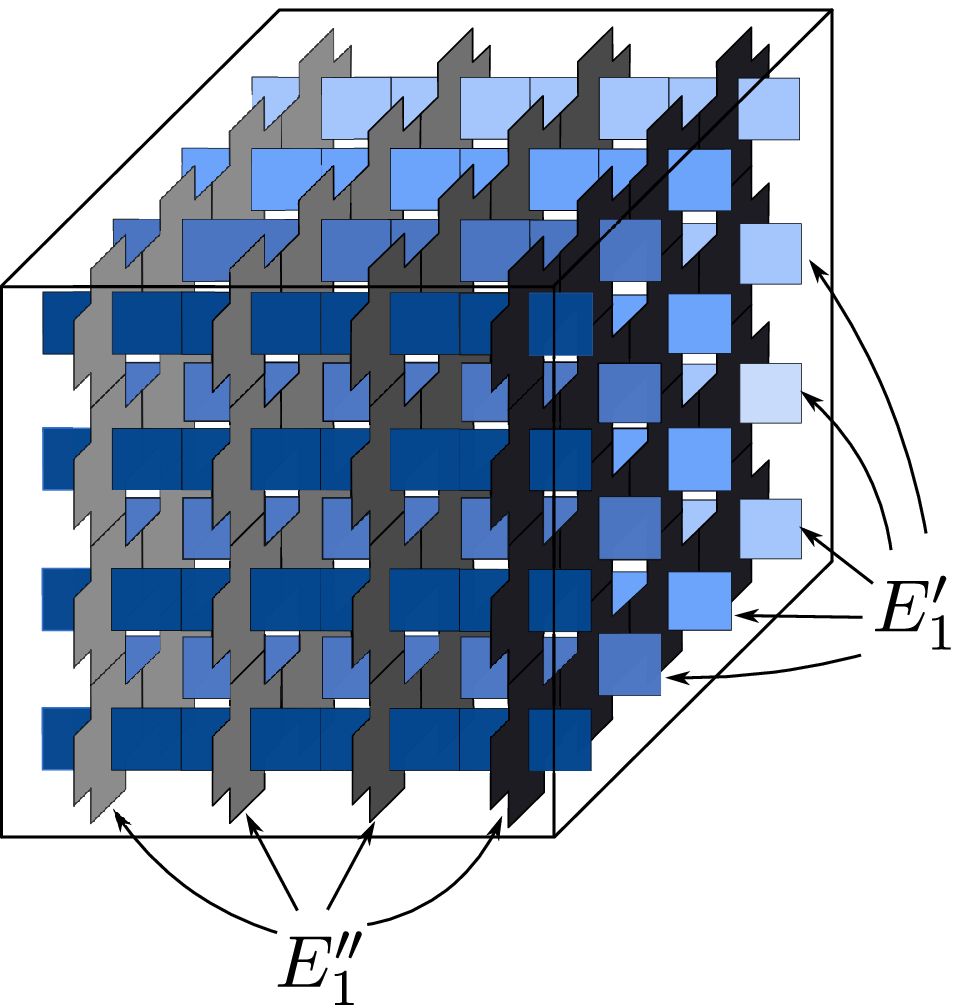,height=6cm}
}
\caption{Construction of the sets $E'$, $E''$, $E_1'$ and $E_1''$. The domains $W_1=[0,1]^3\setminus(E'\cup E'')$, and $W_2=W_1\setminus(E_1'\cup E_1'')$ are the first two steps in the construction of the slit Menger curve $\mathfrak{M}$.}
\end{figure}
Thus, $E'$ is a ``flat tube" containing $p$ that is parallel to the $XZ$-plane, while $E''$ is a ``flat cross" containing $p$ which is contained in the plane $\pi_1^{-1}(p)$ perpendicular to the $X$-axis. Finally, let
\begin{align*}
E_0 &:= E =  E'\cup E'' \subset \overline{W}_0=[0,1]^3.
\end{align*}

We define the domains $W_i$ by induction. Let $W_0=(0,1)^3$ and  $W_1 = W_0\setminus E_0.$
Thus $W_1$ is an open connected subset of the unit cube $[0,1]^3$. Given a dyadic cube $Q\in\D$ we let $E_Q$ be the rescaled copy of $E_0$ in $Q$, i.e.
\begin{align*}
E_Q = T_Q(E_0).
\end{align*}

To define the open set $W_2\subset W_1$ we will remove a union of certain smaller copies of $E$. However, in order to avoid ``non-transversal intersections" it is convenient to ``skip" one generation of dyadic cubes. Thus, for $i\geq 1$ we let
\begin{align*}
  E_i = \bigcup_{Q\in \D_{2i}} T_Q(E_0), \quad E'_i = \bigcup_{Q\in \D_{2i}} T_Q(E'), \quad E''_i = \bigcup_{Q\in \D_{2i}} T_Q(E'')
\end{align*}
and
\begin{align*}
  W_{i+1}&= W_{i} \setminus E_i = W_0 \setminus \bigcup_{i=0}^n E_i
  = (0,1)^3 \setminus \left[\bigcup_{i=0}^i \left(\bigcup_{Q\in\D_{2i}} T_Q(E_0) \right) \right].
\end{align*}

Note that for every $i\geq0$ the set $E'_i$ intersects the $yz$ plane $\{x=0\}\subset\mathbb{R}^3$ and in fact
\begin{align*}
  E'_i\cap\{x=0\} = \bigcup_{Q\in\Delta_{2i}^{yz}} s(Q),
\end{align*}
where the union is over all the diadic squares of generation $2i$ contained in the unit square in the $zy$ plane, denoted by $\Delta_{2i}^{zy}$, and $s(Q)\subset Q$ denotes the vertical slit of length $l(Q)/2=\frac{4^{-n}}{2}$ with the midpoint at the center of $Q$.

Similarly, we have also
\begin{align*}
  E''_i\cap\{y=0\} = \bigcup_{Q\in\Delta_{2i}^{xz}} s(Q),\quad
  E''_i\cap\{z=0\} = \bigcup_{Q\in\Delta_{2i}^{xy}} s(Q),
\end{align*}
where the families $\Delta_{2i}^{xz}$, $\Delta_{2i}^{xy}$ and slits  $s(Q)$ are defined as above, but in the $xz$ and $xy$ planes, respectively.

Let $\overline{W}_i$ be the completion of $W_i$ in the path metric $d_{W_i}$. Note that if $x\in E_i$ then it can ``split" into two or four points in $\overline{W}_{i+1}$, where the latter happens only if $x\in E_i'\cap E_i''$. For each $0\leq i< j$ let
$$\omega_{i,j}:\overline{W}_j\to \overline{W}_i$$ be the map which identifies the points in $W_j$ which correspond to the same point in $W_i$. Similarly to the case of the Sierpi\'nski spaces we obtain an inverse system $(W_i,\omega_{i,j})$ and define the topological space
\begin{align*}
  \mathfrak{M}=\lim_{\leftarrow}(\overline{W}_i,\omega_{i,j}).
\end{align*}
Note also, that for every $p=(p_0,p_1,\ldots)\in\mathfrak{M}$ and every $i\geq 0$ there is a natural projection of $p$ to $\overline{W}_i$ defined as follow:
\begin{align*}
\begin{split}
 \omega_i:\mathfrak{M} &\to \overline{W}_i\\
  p &\mapsto p_i.
\end{split}
\end{align*}
We will denote $\omega_0$ simply by $\omega$. Thus, since $W_0=(0,1)^3$,  we have the projection
\begin{align*}
\begin{split}
  \omega:\mathfrak{M} &\to [0,1]^3,\\
  p &\mapsto p_0.
\end{split}
\end{align*}

The subsets of $\mathfrak{M}$ corresponding the top and bottom faces of the boundary of the unit cube in $\mathbb{R}^3$, i.e.
\begin{align}\label{top&bottom}
\begin{split}
  \mathcal{T} &= \omega^{-1}(\{(x,y,1) \,|\, 0\leq x,y \leq 1\}),\\
 \mathcal{B} &= \omega^{-1}(\{(x,y,0) \,|\, 0\leq x,y \leq 1\}),
\end{split}
\end{align}
will be called the \emph{top} and \emph{base} (or bottom) of $\mathfrak{M}$, respectively.

Note that both $\mathcal{T}$ and $\mathcal{B}$ are homeomorphic to the Sierpi\'nski carpet. In fact, as we show next, equipped with the metric induced from  $d_{\mathfrak{M}}$, $\mathcal{T}$ and $\mathcal{B}$ are isometric to the slit carpet $M(\calS)$ corresponding to a very particular sequence of slits $\mathcal{S} = \{s(Q)\}_{Q\in\Delta^{xy}_{2i}} $ in the unit square $[0,1]\times[0,1]\times\{0\}$ in the $xy$ plane.

\begin{figure}[t]
\centerline{
\psfig{figure=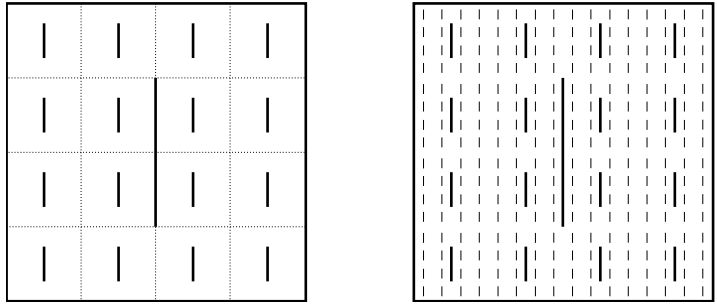,height=4cm}
}
\caption{The ``top" and ``bottom" of the slit Menger curve $\mathfrak{M}$, denoted by $\mathcal{T}$ and $\mathcal{B}$, respectively, are isometric to the slit carpet corresponding to the sequence of slits $\mathcal{S}=\{s_{kl}^i\}$ defined by (\ref{def:slits}). Every slit $s\in\mathcal{S}$ is centered at the center of a diadic square $Q$ of side-length $4^{-i}$ and the length of $s$ is equal to half the side length of $Q$.}
\label{fig:base-carpet}
\end{figure}

More concretely we can write $\mathcal{S}=\{s^i_{kl}\}$, where
$s^i_{kl}$ is the vertical slit of length $\frac{4^{-i}}{2}$ in the diadic square
$$Q^i_{kl}:=\left[\frac{k}{4^i},\frac{k+1}{4^i}\right)\times\left[\frac{l}{4^i},\frac{l+1}{4^i}\right) \in\Delta_{2i}.$$

Equivalently,
\begin{align}\label{def:slits}
\begin{split}
  s_{kl}^{i}&:=\left\{\left(\frac{2k+1}{2\cdot 4^{i}},y \right) \,: \, \left|y-\frac{2l+1}{2\cdot4^{i}}\right|\leq \frac{1}{4^{i+1}}\right\}, \mbox{ where } i\geq0, \mbox{ and } k,l\in\{0,\ldots,4^i-1\}.
\end{split}
\end{align}
To show that $\mathcal{B}$ is isometric to the slit carpet pick two points $p,q\in\mathcal{B}$. Then for every curve $\g\subset\mathfrak{M}$ connecting $p$ and $q$ there is a curve $\g'\subset\mathcal{B}$ which is no longer than $\g$. Indeed, without loss of generality we may assume that $\g$ does not intersect $\omega^{-1}(\bigcup_i E_i)$, since any other curve can be approximated by such curves. Then, denoting by $\pi_{xy}$ the orthogonal projection of $\mathbb{R}^3$ to the $xy$ plane, we can take $\g'=\omega^{-1}\circ\pi_{xy}\circ\omega\circ\g$, i.e. the ``projection of $\g$ to $\mathcal{B}$". Therefore, for $p,q\in\mathcal{B}$ we have $d_{\mathfrak{M}}(p,q) \geq \inf\{l(\g'): \g'\subset \mathcal{B}\}$ and in particular these two quantities are equal. Thus, the restriction of the path metric $d_{\mathfrak{M}}$ to $\mathcal{B}$ gives the path metric defined for the slit carpets.
We will call the slits in $\mathcal{B}$ corresponding  to $s^{i}_{kj}$the \emph{slits of generation} $i\geq 0$.

Define the \textit{double of the slit Menger curve $\mathfrak{M}$ along} $\mathcal{T}\cup\mathcal{B}$, denoted by $D\mathfrak{M}$, by identifying the two copies of $\mathcal{T}\cup\mathcal{B}$ (by the identity map) for two copies of $\mathfrak{M}$. More concretely, if $\mathfrak{M}_1$ and $\mathfrak{M}_2$ are the two copies of the slit Menger curve, $id:\mathfrak{M}_1\to\mathfrak{M}_2$ is the identity map, and $\mathcal{T}_i$ and $\mathcal{B}_i$ are the top and bottom slit carpets in $\mathfrak{M}_i$, $i=1,2$, then we can define the double of $\mathfrak{M}$ as follows:
\begin{align*}
  D\mathfrak{M} = (\mathfrak{M}_1 \sqcup \mathfrak{M}_2) / \sim,
\end{align*}
where, if $x\in\mathfrak{M}_1$ and $y\in\mathfrak{M}_2$ then
\begin{align*}
  x\sim y \quad \Longleftrightarrow \quad y=id(x) \mbox{ and } x\in \mathcal{T}_1\cup\mathcal{B}_1.
\end{align*}

Next we equip $\mathfrak{M}$ and $D\mathfrak{M}$ with metrics. Given $p=(p_0,p_1,\ldots),q=(q_0,q_1,\ldots)\in \mathfrak{M}$ we define
\begin{align*}
  d_{\mathfrak{M}}(p,q)=\lim_{i\to\infty} d_{\overline{W}_i}(p_i,q_i).
\end{align*}
The limit above exists and is finite since the sequence $\{ d_{\overline{W}_i}(p_i,q_i)\}$ is non-decreasing and bounded.

The metric space $(\mathfrak{M},d_{\mathfrak{M}})$ will be called the \textit{slit Menger curve}. The metric $d_{\mathfrak{M}}$ will be called the \textit{path metric on the slit Menger curve $\mathfrak{M}$}.

The path metric on $\mathfrak{M}$ induces a path metric on $D\mathfrak{M}$, which we will denote by $d_{D\mathfrak{M}}$. Indeed, every curve $\g$ in $D\mathfrak{M}$ can be written as a disjoint union of two (not-necessarily connected) curves, one of which the image in one of the copies of $\mathfrak{M}$ (we can denote it $\mathfrak{M}_1$) and the other in $D\mathfrak{M}\setminus\mathfrak{M}_1$. Clearly $d_{D\mathfrak{M}}$ restricts to $\d_{\mathfrak{M}}$ on each of the copies of $\mathfrak{M}$ in the double.

The space $D\mathfrak{M}$ will be called the \textit{double of the slit Menger curve}. We will show below that $D\mathfrak{M}$ is homeomorphic to the Menger curve (cf. Theorem \ref{cor:slit-menger-homeo-menger}). Therefore, Theorem \ref{thm:menger-co-Hopf} follows from the following result.

\begin{theorem}\label{thm:slit-menger-co-Hopf}
  Every quasisymmetric mapping $f:D\mathfrak{M}\to D\mathfrak{M} $ is surjective.
\end{theorem}
Theorem \ref{thm:slit-menger-co-Hopf} is proved in Section \ref{Section:menger-coHopf-proof} by combining the results of Sections \ref{Sec:menger-fibered-over-slit-carpet} and \ref{Section:QS-maps-fiber-preserving}. The rest of this section is devoted to the proof of  the following result.

\begin{theorem}\label{cor:slit-menger-homeo-menger}
The metric spaces $\mathfrak{M}$ and $D\mathfrak{M}$ are homeomorphic to the Menger curve $\M$.
\end{theorem}

\begin{proof}
Theorem \ref{cor:slit-menger-homeo-menger} follows from Lemmas \ref{lemma:slit-menger-top-properties}, \ref{lemma:slit-menger-dimension1}, \ref{lemma:slit-menger-nonplanar} below and Anderson's Theorem \ref{thm:Anderson}.
\end{proof}

\begin{lemma}\label{lemma:slit-menger-top-properties}
The spaces $\mathfrak{M}$ and $D\mathfrak{M}$ are path connected, locally connected topological spaces with no local cut points.
\end{lemma}
\begin{proof}
To show that $\mathfrak{M}$ is path connected, note that for every two points $p$ and $q$ in $\mathfrak{M}$ one can connect them by ``vertical" paths $\g_p$ and $\g_q$ to the top $\mathcal{T}$, which is a path connected slit carpet, cf. \cite{Mer:coHopf}. By a ``vertical path" $\g_p$ here we mean the connected component containing $p$ of the set $\omega^{-1}(v_{\omega(p)})$, where $v_{\omega(p)}$ is the vertical (i.e. parallel to $z$-axis) interval through $\omega(p)$ in $[0,1]^3$.  In fact, just like the slit carpets one may show $\mathfrak{M}$ is a geodesic space, cf. \cite{Mer:coHopf}.

To see that $\mathfrak{M}$ is locally connected, note that for every point $p\in\mathfrak{M}$ and every $\eps>0$ there is a homeomorphic copy of $\mathfrak{M}$ (which is connected) of diameter less than $\eps$ containing $p$, indeed, there is such a subset of $\mathfrak{M}$ isometric to $(\mathfrak{M},4^{-n}d_{\mathfrak{M}})$ (and hence homeomorpchic to $\mathfrak{M}$) for any $n\geq1$.

Just like in the proof of local connectivity, absence of local cut points follows from the fact that $\mathfrak{M}$ is self-similar. Indeed, since removing a point does not disconnect $\mathfrak{M}$, the same holds locally around every point.
\end{proof}

\begin{lemma}\label{lemma:slit-menger-dimension1}
The spaces  $\mathfrak{M}$ and $D\mathfrak{M}$ have topological dimension equal to $1$.
\end{lemma}


\begin{proof}
We will prove the lemma for $\mathfrak{M}$. The case of $D\mathfrak{M}$ can be done the same way.

Since $\mathfrak{M}$ contains subsets homeomorphic to the interval $[0,1]$ we have that $\dim_{top}\mathfrak{M}\geq 1$. To see that $\dim_{top}\mathfrak{M}\leq 1$, by Lemma \ref{lemma:top-dimension-estimate} it is enough to show that for every $\eps>0$ there is an open cover of $\mathfrak{M}$ with sets of diameter less than $\eps$, and such that every point $p\in\mathfrak{M}$ belongs to at most two elements of that cover.

The coverings we are going to construct will consist of (preimages in $\mathfrak{M}$ of) neighborhoods of (intersections of) diadic cubes in $[0,1]^3$.

Given $\eps>0$ and a set $E\subset\overline{W_n}=(W_n,d_{W_n})$ we will denote by $E^{\eps}$ the $\eps$-neighborhood of $E$, i.e.
$$E^{\eps}=\{x\in \overline{W_n} : \dist_{W_n}(x,E)<\eps\},$$

For every $\mathbf{x}=(x_1,x_2,x_3)\in\mathbb{R}$ and $r>0$ we denote
$$Q(\mathbf{x},r) = \{\mathbf{y} \in\mathbb{R}^3: \max_{i\in\{1,2,3\}}|x_i-y_i|<r\},$$
i.e. the ball centered at $\mathbf{x}$ of radius $r>0$ in the $L_{\infty}$ norm on $\mathbb{R}^3$. Note that $Q(\mathbf{x},r)$ is a cube centered at $x$ of sidelength $2r$.

For $n\geq 0$ let
$$(4^{-n}\mathbb{Z})^3 : = \left\{\left(\frac{a}{4^n},\frac{b}{4^n},\frac{c}{4^n}\right) : a,b,c\in\mathbb{Z} \right\}.$$

%
%
%
%

Consider the following family $\mathcal{Q}_n$ of pairwise disjoint cubes in $[0,1]^3$:
\begin{align}
  \mathcal{Q}_n = \left\{Q\left(\mathbf{x},\frac{4^{-n}}{2}\right) \cap[0,1]^3 : \mathbf{x}\in (4^{-n}\mathbb{Z})^3 \right\}.
\end{align}

Note, that even though $\mathcal{Q}_n$ is not an open cover of $\overline{W_n}$, since the boundaries of these cubes are not covered, the family consisting of $\eps$-neighborhoods of the elements of $\mathcal{Q}_n$ is a cover for every $\eps>0$. The families $\mathcal{Q}_n$ have the following property.

\begin{lemma}\label{lemma:cubes-adjacent-faces}
Suppose $n\geq 1$. If  $Q,Q'\in\mathcal{Q}_n$ then
\begin{align}
  \mbox{either } \,& \mbox{$Q$ and $Q'$ share a face},\\
  \mbox{or } \,& \dist_{W_n}(Q,Q')\geq 1/4^{n+1}.
\end{align}
\end{lemma}
Lemma \ref{lemma:cubes-adjacent-faces} will be proved momentarily. Before that we use it to show that $\dim_t\mathfrak{M}\leq1$.

Let $\eps_n=1/4^{n+1}$ and $0<\eps<\eps_n/2$. Consider the family 
$$\tilde{\mathcal{Q}}_n^{\eps} : = \{\omega^{-1}(Q^{\eps}) : Q \in\mathcal{Q}_n \},$$ which is a covering  of $\mathfrak{M}$, whenever $\eps>0$.
Moreover, if $p\in \omega^{-1}(Q^{\eps})\cap \omega^{-1}((Q')^{\eps})$
then
\begin{align*}
  \dist_{W_n}(Q,Q')\leq \dist_{W_n}(Q,\{\omega(p)\})+\dist_{W_n}(\{\omega(p)\},Q') \leq 2\eps <\eps_n.
\end{align*}
Therefore by Lemma \ref{lemma:cubes-adjacent-faces} $Q$ and $Q'$ are adjacent, i.e. share a $2$ dimensional face. Since no three cubes in $\mathbb{R}^3$ can share the same face it follows that $p$ belongs to at most two elements of $\tilde{\mathcal{Q}}_n^{\eps}$ simultaneously.
\end{proof}

\begin{proof}[Proof of Lemma \ref{lemma:cubes-adjacent-faces}]
Suppose $Q=Q(\mathbf{x},4^{-n}/2)\cap(0,1)^3$ and $Q'=Q(\mathbf{x'},4^{-n}/2)\cap(0,1)^3$ are distinct cubes from $\mathcal{Q}_n$. Note, that
$\dist_{\infty}(\mathbf{x},\mathbf{x'})\in\{{4^{-n}},{2}\cdot{4^{-n}},\ldots\}.$

If $\dist_{\infty}(\mathbf{x},\mathbf{x'})\geq 2\cdot 4^{-n}$, then
\begin{align*}
  \dist_{W_{n+1}}(Q,Q')\geq\dist_{\infty}(Q,Q') \geq \dist_{\infty}(\mathbf{x},\mathbf{x'}) - {4^{-n}}\geq 4^{-n}.
\end{align*}

If $\dist_{\infty}(\mathbf{x},\mathbf{x'})< 2\cdot 4^{-n}$ then $\dist_{\infty}(\mathbf{x},\mathbf{x'})= 4^{-n}$  and therefore
\begin{align}
  x_i'\in\{x_i-4^{-n},x_i,x_i+4^{-n}\}, \quad i\in\{1,2,3\}.
\end{align}

In this case, we would like to estimate $\dist_{W_{n+1}}(Q,Q')$ from below. Suppose $Q$ and $Q'$ are not adjacent then the segment $[\mathbf{x},\mathbf{x'}]$ is not parallel to any of the coordinate axes and therefore $x_i\neq x_i'$ for at least two indices $i\in\{1,2,3\}$.

Suppose
\begin{align}\label{non-adjacent-cubes}
  x_1\neq x'_1 \mbox{ and } x_2\neq x_2'.
\end{align}
Let $\g\subset W_{n+1}$ be a curve connecting $Q$ and $Q'$. Let $\pi_{xy}$ be the orthogonal projection of $\mathbb{R}^3$ to the $xy$-plane and $\g'=\pi_{xy}(\g)\subset\{z=0\}$. Then $\g'$ connects the squares $S=\pi_{xy}(Q)$ and $S'=\pi_{xy}(Q')$ in the slit domain $(0,1)\times(0,1)\times\{0\} \setminus E_{n}''$. Moreover, since $\pi_{xy}$ is $1$-Lipschitz we have $l(\g)\geq l(\g')$.

\begin{figure}[t]
\centerline{
\psfig{figure=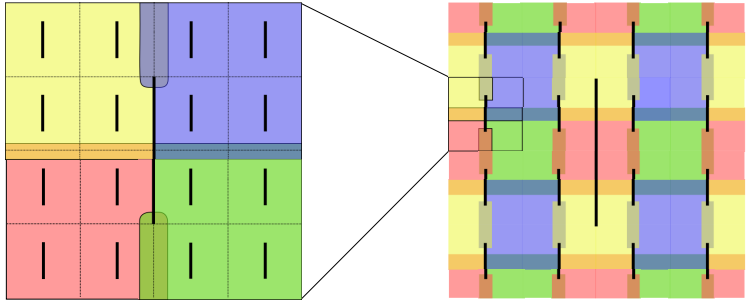,height=4cm}
}
\caption{Coverings of the slit carpet $\mathcal{B}$ of order $1$. The figure on the left gives an initial covering. Scaling the initial covering and using reflections one may obtain finer coverings of $\mathcal{B}$. The coverings pictured above are the two-dimensional analogues of the coverings constructed in the proof of Lemma \ref{lemma:slit-menger-dimension1}.}
\label{fig:non-planarity}
\end{figure}

Since $x_1,x_2,x_1',x_2' \in 4^{-n}\mathbb{Z}$, from (\ref{non-adjacent-cubes}) it follows that there is a square $S\subset(0,1)^2\times\{0\}$ which has the interval from $(x_1,x_2)$ to $(x_1',x_2')$ as a diagonal. Thus, $S$ is a diadic square of sidelength $4^{-n}$, i.e. $S\in \Delta^{xy}_{2n}$. Denoting by $s$ the vertical slit of length $4^{-n}/2$ through the center of $S$ we have that $s\subset \overline{E''}\cup\{z=0\}$. Since $s\times(0,1)\subset E_2''$ it follows that  $\g\subset(0,1)^3\setminus(s\times(0,1))$ and therefore
$$\g'\in(0,1)^2\times\{0\}\setminus s \subset \mathbb{R}^2\setminus s.$$
Finally, there is an isometry $T$ of $\mathbb{R}^2$ mapping $s$ to a vertical interval centered at the origin and such that $Q\subset\{(x,y): x<0,y<0\}$. Then $T(Q')\subset\{(x,y): x>0,y>0\}$ and $T(\g')$ is a curve of the same length as $\g'$, connecting the quadrants $\{(x,y): x<0,y<0\}$ and $\{(x,y): x>0, y>0\}$ in $\mathbb{R}^2\setminus s$. Clearly the length of $T(\g')$ is at least half the length of the slit $s$, since $T(\g')$  has to ``cross" one of the strips $\{-l(s)/2 < y < 0\}$ or $\{0 < y < l(s)/2 \}$. Therefore
$$l(\g)\geq l(\g') = l(T(\g')) \geq \frac{l(s)}{2}= \frac{4^{-n}}{4}.$$
Note that in the remaining cases, when $x_2\neq x_2'$ and $x_3\neq x_3'$ or if $x_2\neq x_2'$ and $x_3\neq x_3'$ then the same proof as above works, if one uses the projections to the $yz$ or $xz$ planes, respectively.
\end{proof}

\begin{remark}
An alternative proof of the fact that $\mathfrak{M}$ has topological dimension $1$ can be given using a different definition of dimension. Namely, a metric space $X$ has topological dimension $1$ if for every $x\in X$ and $\eps>0$ there is an open subset $U\subset B(x,\eps)$ containing $x$ such that $\partial{U}$ is $0$ dimensional (e.g. homeomorphic to a Cantor set). It may seem counter intuitive that there are such open sets in $\mathfrak{M}$, however one can construct them by using graphs of certain piecewise constant functions in $[0,1]^3$. The first step in that direction would be to note that for almost every $t\in(0,1)$ the set $\omega^{-1}(\{z=t\}\cap[0,1]^3)$ is a Cantor set (of Hausdorff dimension $2$) in $\mathfrak{M}$. ``Cutting" and "pasting" such Cantor sets one can construct small neighbourhoods with Cantor set boundaries in $\mathfrak{M}$. We do not provide details, since the proof above seems simpler.
\end{remark}

\begin{lemma}\label{lemma:slit-menger-nonplanar}
The space $\mathfrak{M}$ has no nonplanar open subsets.
\end{lemma}
\begin{proof} We will show that every open subset of $\mathfrak{M}$ contains a homeomorphic copy of $K_5$, the complete graph on $5$ vertices. Since $K_5$ is non-planar, this will imply the theorem. We will first show that there is a copy $K$ of $K_5$ in $\mathfrak{M}$ with vertices at $p_1,\ldots,p_5\in\mathfrak{M}$ which are mapped by $\omega_0$ to the following vertices of $[0,1]^3$: $a=(0,0,0)$, $b=(1,0,0)$, $c=(0,1,0)$, $d=(0,0,1)$ and $e=(1,1,1)$. Consider the following curves connecting the points $a,\ldots,e$:
\begin{align*}
  \g_{a,b} &: = \{(x,0,0): x\in[0,1] \},\quad \g_{b,e} : = \{(1,y,0): y\in[0,1]\} \cup \{(1,1,z): z\in[0,1]\},\\
  \g_{a,c} &: = \{(0,y,0): y\in[0,1] \}, \quad \g_{c,e} : = \{(0,1,z): z\in[0,1]\} \cup \{(x,1,1): x\in[0,1]\},\\
  \g_{a,d} &: = \{(0,0,z): z\in[0,1] \}, \quad \g_{d,e} : = \{(x,0,1): x\in[0,1]\} \cup \{(1,y,1): y\in[0,1]\}.
\end{align*}

\begin{figure}[h]
\centerline{
\psfig{figure=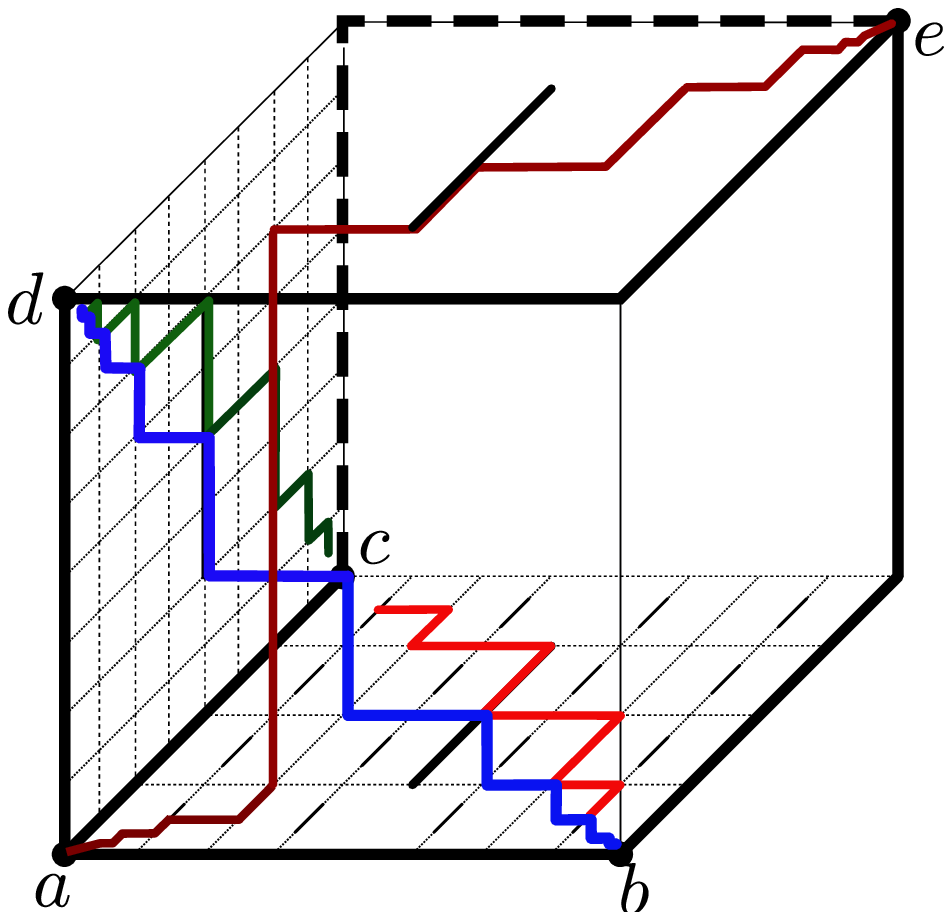,height=5cm} \qquad \qquad \qquad
\psfig{figure=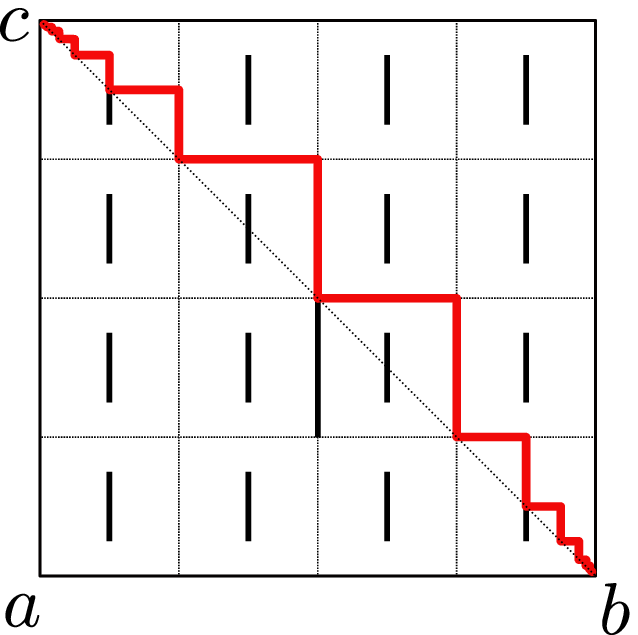,height=5cm}
}
\caption{\small{On the left are the $\binom{5}{2} = 10$ non-intersecting curves connecting the $5$ points $a,\ldots,e\subset[0,1]^3$. The preimage of the union of these curves in $\mathfrak{M}$ is homeomorphic to $K_5$. Six of these curves are just the edges or concatenations of the edges of $[0,1]^3$. Three more, $\g_{bc},\g_{cd},\g_{db}$, are contained in the slit carpets contained in the preimages of the faces of the cube (like the curve connecting $b$ to $c$ on the right). Finally, $\g_{ae}$  connects $a$ and $e$ to two endpoints of a vertical interval connecting the top and bottom faces of the cube.}}
\label{fig:non-planarity}
\end{figure}

To construct a curve $\g_{b,c}$ note that the ``face" of $\mathfrak{M}$ corresponding to the bottom face of $Q_0$, i.e. the set
$$\omega_{0}^{-1}(\{(x,y,0) : 0\leq x,y \leq 1 \})\subset \mathfrak{M},$$ is a metric carpet (in fact it is a slit carpet). From Whyburn's theorem one may easily conclude that given two distinct points $p,q$ on a peripheral circle $\g$ of a carpet there is a curve $\g_{p,q}$ connecting $p$ and $q$ such that $\g\cap\g_{p,q}=\{p\}\cup \{q\}$. In our case $\g_{b,c}$ can be constructed more concretely by connecting the opposite vertices of the square by a piecewise linear curve each piece of which is parallel to one of the coordinated axes, for instance as shown in Figure \ref{fig:non-planarity}. Similarly, we can construct curves $\g_{d,c},\g_{d,b},\g_{b,e}$ which are disjoint from all the previously defined curves except for the endpoints. Finally, let $\g_{a,e}=\g_1\cup\g_2\cup \g_3$ where $\g_1\subset[0,1/4]\times[0,1/4]\times\{0\}$ connects $a$ to $(1/4,1/4,0)$, $\g_{2}=\{(1/4,1/4,z) | 0\leq z\leq 1\}$ and $\g_3\subset[1/4,3/4]\times[1/4,3/4]\times\{1\}$ connects $(1/4,1/4,1)$ to $e$. Thus, each of the points $a,\ldots,e$ is connected to every other point and the connecting curves are pairwise disjoint.

Next, note that for the curves $\g_{a,b},\ldots,\g_{a,e}$ there are (not necessarily unique) lifts $\tilde{\g}_{a,b}, \ldots,\tilde{\g}_{a,e}$ in $\mathfrak{M}$ which connect the points $\omega_0^{-1}(a), \ldots, \omega_0^{-1}(e)$ and which are pairwise disjoint except for these endpoints.  Therefore the set $K=\tilde{\g}_{a,b}\cup\ldots\cup\tilde{\g}_{a,e}\subset \mathfrak{M}$ is homeomorphic to $K_5$. Finally, since for every point $p\in\mathfrak{M}$ and every $\eps>0$ there is an open subset $U\subset B(p,\eps)$ homeomorphic to $\mathfrak{M}$, it follows that $\mathfrak{M}$ has no non-planar open subsets.
%
%

%
%
%
\end{proof}

Just like the slit Sierpi\'nski spaces the spaces defined  in this section are also Ahlfors regular. We again omit the proof of this result since it is very similar to that of the corresponding results in \cite{Mer:coHopf}, cf. Lemmas $2.2$, $2.3$ and Proposition $2.4$ in \cite{Mer:coHopf}.

\begin{lemma}\label{lemma:menger-properties}
  The metric spaces $(\mathfrak{M},d_{\mathfrak{M}})$ and $(D\mathfrak{M},d_{D\mathfrak{M}})$ are Ahlfors $3$-regular metric measure spaces when equipped with the Hausdorff $3$-measure.
\end{lemma}

%

\begin{remark}
We could also define the double of $\mathfrak{M}$ along all of its \textit{outer boundary} $\partial_o\mathfrak{M}$ of $\mathfrak{M}$, where
\begin{align}
  \partial_o\mathfrak{M}:=\left\{ p\in\mathfrak{M} : \omega(p)\in\partial ([0,1]^3)\right\}.
\end{align}
We denote the resulting space
\begin{align}
  D_o\mathfrak{M} = (\mathfrak{M} \sqcup \mathfrak{M}) / \partial_o\mathfrak{M}.
\end{align}
From the proof of Theorem \ref{thm:menger-co-Hopf} given below it will become clear that the theorem holds for $D_{o}\mathfrak{M}$ as well. The reason we formulated the theorem for the double along $\mathcal{T}\cup\mathcal{B}$ is to emphasize that the ``extra" identifications are not needed to conclude QS co-Hopficity in these cases.
\end{remark}

\section{$\mathfrak{M}$ and $D\mathfrak{M}$ as fibred spaces over a slit carpet}\label{Sec:menger-fibered-over-slit-carpet}

Recall from (\ref{top&bottom}) that the base $\mathcal{B}$ of the slit Menger curve $\mathfrak{M}$ was defined as follows,
$$\mathcal{B} = \omega^{-1}{(B)},\mbox{ where } B=\{(x,y,0)\, | \, 0\leq x,y\leq 1 \}.$$ As noted above, cf. the discussion after (\ref{def:slits}), with the metric induced from  $d_{D\mathfrak{M}}$, the base $\mathcal{B}$ is isometric to the slit carpet $M(\calS)$ corresponding to the sequence of slits $\calS=\{s^i_{kl}\}$ as in (\ref{def:slits}), see Figure \ref{fig:base-carpet}.

Let $\pi_{xy}:\mathbb{R}^3\to\mathbb{R}^2$ denote the orthogonal projection of $\mathbb{R}^3$ onto the $xy$-plane, i.e.
$\pi_{xy}(x,y,z)=(x,y).$ Let
\begin{align}\label{def:projection-to-base}
  \Theta:=\pi_{xy}\circ \omega:\mathfrak{M}\to [0,1]^2\subset\mathbb{R}^2.
\end{align}
In an analogous way we may define the ``projection" map, denoted again by $\Theta$, on $D\mathfrak{M}$ as well.

We say that a subset $E$ of $\mathfrak{M}$ or $D\mathfrak{M}$ is \emph{vertical or $z$-parallel} if $\omega(E)$ is contained in a line parallel to the $z$-axis, or equivalently if $\Theta(E)$ is a point in $\mathbb{R}^2$.

For a point $p$ in $\mathfrak{M}$ we define the \textit{fiber through $p$} as the largest connected $z$-parallel subset containing $p$ in $\mathfrak{M}$ and denoted it by $\g_p$.
Equivalently,
\begin{align*}
  \g_p&:= \mbox{the connected component of $\Theta^{-1}(\Theta(p))\subset \mathfrak{M}$ containing $p$}.
\end{align*}

From the definition it follows that given $p$ and $q$ in $\mathfrak{M}$ we have that the subsets $\g_p$ and $\g_q$ either coincide or are disjoint. It is easy to see that there are points $p,q\in \mathfrak{M}$ with disjoint fibers such that $\omega(p)=\omega(q)$, and therefore $\Theta(p)=\Theta(q)$, e.g. points corresponding to (two ways of approaching) the same point on a slit in $[0,1]^2$. Thus, $\Theta$ does not induce a one-to-one correspondence between the fibers $\g_p$ and the square $[0,1]^2$.

On the other hand, there is a natural one-to-one correspondence between  the fibers $\g_p$ and the points of the base slit carpet $\mathcal{B}$. Indeed, for every $p\in \mathfrak{M}$ the fiber $\g_p$ intersects the set $\mathcal{B}$ at a single point. Therefore there is a well defined mapping (``projection" of $\mathfrak{M}$ onto $\mathcal{B}$)
  \begin{align*}
  \begin{split}
  \Pi :\mathfrak{M} &\to\mathcal{B}\\
   p \,\, &\mapsto \g_p \cap \mathcal{B}.
   \end{split}
  \end{align*}

Using the map $\Pi$ the fiber $\g_p\subset \mathfrak{M}$ can then be written as
\begin{align*}
  \g_{p}
  &= \{ q \in \mathfrak{M} \, | \, \Pi(q)=\Pi(p) \} = \Pi^{-1}(\Pi(p)).
\end{align*}
Note also, that $\Theta = \omega|_{\mathcal{B}} \circ \Pi.$
Thus $\mathfrak{M}$ can be thought of as a metric Menger curve which is ``fibred over the slit carpet $\mathcal{B}\subset\mathfrak{M}$", i.e.
\begin{align*}
  \mathfrak{M}=\bigcup_{p\in\mathcal{B}} \g_{p},
\end{align*}
where $\g_{p}\cap\g_{q}=\emptyset$ if $p\neq q$, with $p,q\in\mathcal{B}$.

For this reason we will call the fiber $\g_p$ the \textit{fiber over $\Pi(p)$}. If $p\in\mathcal{B}\subset\mathfrak{M}$ then $\Pi(p)=p$ and we say that $\g_{p}$ is \textit{the fiber over $p$ in $\mathfrak{M}$}.

Similarly to the discussion above  we may define the projection map $\Theta:D\mathfrak{M}\to[0,1]^2$ and for any point $p\in D\mathfrak{M}$ the fiber through $p$, as the largest connected $z$-parallel subset of $D\mathfrak{M}$ containing $p$, which we will denote by $\hat{\g}_p$. Moreover, there is a one-to-one correspondence between the fibers $\hat{\g_p}$ and the points of the slit carpet $\mathcal{B}$ and we can write the double of the slit Menger curve as follows:
\begin{align*}
  D\mathfrak{M}=\bigcup_{p\in\mathcal{B}} \hat{\g}_{p}.
\end{align*}

 From the construction of the double $D\mathfrak{M}$ it follows that if fibers $\g\subset\mathfrak{M}$ and $\hat{\g}\subset D\mathfrak{M}$ correspond to the same point $p\in\mathcal{B}$, and therefore can be written as $\g_p$ and $\hat{\g}_p$, respectively, then $\hat{\g}_p$ is homeomorphic to the double of $\g_p$ along the two end points. If we identify the endpoints with, say $\{0\}$ and $\{1\}$ then we can write that for every $p\in\mathcal{B}$ we have
\begin{align*}
\hat{\g}_p \cong (\g_p\sqcup\g_p)/\{0,1\}.
\end{align*}

\begin{remark}
Note that if the point $p\in{\mathcal{B}}$ is such that $\omega(p)$ does not belong to a slit then $\hat{\g}_{p}$ consists of two intervals of length $1$, corresponding to the two copies of $\mathfrak{M}$, joined at the endpoints. Therefore in this case $\hat{\g}_{p}$ is homeomorphic to $\mathbb{S}^1$.
\end{remark}


Next, we study the topological type of the fibres $\g_p$ where $p\in\mathcal{B}$ belongs to a slit of $\mathcal{B}$. For this we first define a sequence of metric graphs as follows.

Let $L$ be the graph with six edges, each isometric to $[0,1/4]$, such that four of them are attached cyclically by identifying pairs of endpoints and the remaining two edges are attached to the cycle at two non-consecutive vertices, see Figure \ref{fig:fibers}.

Next, let $L_n, n\geq 1,$ be the graph obtained by dividing the interval $[0,1]$ into $4^{n-1}$ equal length subintervals and replacing each interval by a copy of $L$ scaled by $1/4^{n-1}$, so that every edge of the scaled graph if of length $4^{-n}$. Note that $L_1=L$. For $n>1$, two copies of $L$ in $L_n$ are either disjoint or have a common vertex. Note that there is a natural projection of $L_n$ onto $[0,1]$ which maps two of the vertices to $0$ and $1$. To simplify the notations we will denote these vertices of $L_n$ by $0$ and $1$ as well.

\begin{figure}[htb]
\centering
\psfig{figure=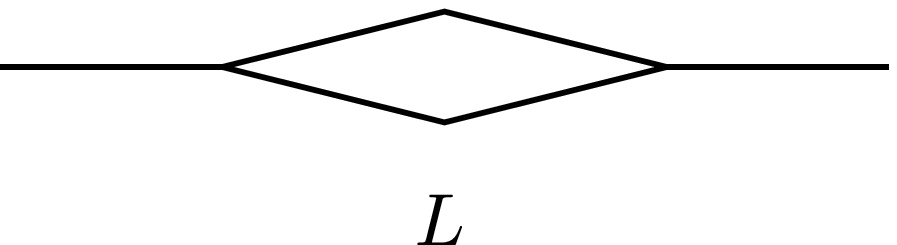,width=4cm} \qquad \qquad
\psfig{figure=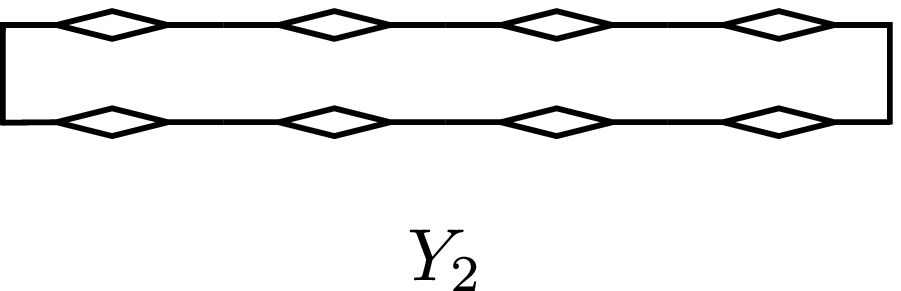,width=4cm}
\caption{\label{fig:fibers} Topology of $\g_p$.}
\end{figure}

Finally, let $Y_n$ be the metric graph obtain by taking two copies of $L_n$ and gluing them along $\{0\}\cup \{1\}$. Thus, $Y_n$ is a ``closed chain" of $2\cdot 4^{n-1}$ scaled copies of $L$.

Note that $Y_n$ and $Y_m$ are homeomorphic if and only if $m=n$.

Suppose $p=(p_0,p_1,\ldots)\in\mathcal{B}$ and $\omega(p)=p_0=(x,y,0)\in[0,1]^2\times\{0\}.$
Note that for every $p\in\mathcal{B}$ the fiber $\g_p\subset\mathfrak{M}$ is the inverse limit of the sequence of fibers over $p_i\in\overline{W}_i, i\geq0$:
\begin{align*}
  (\g_p)_i
  :&=\omega_i(\g_p)\subset\overline{W}_i.
\end{align*}
We will show that the sequence $(\g_p)_i$ can be constructed as a sequence of graphs inductively. Note, that for every $p\in\mathcal{B}$ the fibre
$$\omega(\g_p)=(\g_p)_0=\{(x,y,z):0\leq z\leq 1\}\subset\overline{W}_0$$
is isometric to the interval $[0,1]$, i.e. as a graph has two vertices and one edge connecting them. The topology of $(\g_p)_1$ depends on $p\in\mathcal{B}$ and and can be described  as follows
\begin{lemma}\label{lemma-fibers-0}
Suppose $p=(p_0,p_1,\ldots)\in\mathcal{B}$ belongs to the central slit of $\mathcal{B}$, or equivalently
$$\omega(p)\in s^0_{00} = \left\{ \left(\frac{1}{2},y,0\right): \frac{1}{4}\leq y \leq \frac{3}{4} \right\}.$$ If $\omega(p)=p_0=(1/2,y,0)$ then
\begin{align}\label{gen-1-fibres}
  (\g_p)_1 \cong
  \begin{cases}
    [0,1], &\mbox{ if } y\notin\{1/4,1/2,3/4\},\\
    L_1,  &\mbox{ if } y\in\{1/4,1/2,3/4\}.
  \end{cases}
\end{align}
Moreover, the subsets in $W_1$ corresponding to $\omega(p)$ can be described as follows
\begin{align}
  \omega^{-1}_{01}((\g_p)_0) \cong
  \begin{cases}
    [0,1]\sqcup[0,1], &\mbox{ if } y\in(1/4,1/2)\cup(1/2,3,4),\\
    L_1, &\mbox{ if } y\in\{1/4,3/4\},\\
    L_1\sqcup L_1, &\mbox{ if } y=1/2.
  \end{cases}
\end{align}
\end{lemma}
\begin{proof} 

There are three cases to consider: $(i)$ $\omega(p)$ is the midpoint of $s$; $(ii)$  $\omega(p)$ is one of the two endpoints of $s$; $(iii)$  $\omega(p)$ is neither a midpoint nor an endpoint of $s$.

\textit{\underline{Case {(i)}.}} Suppose $\omega(p)=\left(\frac{1}{2},\frac{1}{2},0\right)$. Then the vertical ($z$-parallel) line through $\omega(p)$ intersects $E'$ as well as $E''$ in $W_0$. For every $0\leq z\leq 1$ the point  $(1/2,1/2,z)\in(\g_p)_0$ has two preimages in $\overline{W}_1$, corresponding to the two sides, namely $x>1/2$ and $x<1/2$, of the ``flat cross" $E''$. In fact the distance between these preimages in $\overline{W}_1$ is at least $1/2$, the length of the slit. Hence, $\omega_{01}^{-1}(\{(1/2,1/2,z): 0\leq z\leq 1\})$ is the union of two connected subsets of $\overline{W_1}$ corresponding to the two sides of the ``flat cross" $E''$. One of these components is $(\g_p)_1$ (say corresponding to $x>1/2$).

If $z\in[0,1/4]\cup[3/4,1]$, then every point $(1/2,1/2,z)\in(\g_p)_0$ corresponds to a single point in $(\g_p)_1$, since it is not in $E'$ or is on the ``lower" or ``upper" edge of $E'$. On the other hand, for every $z\in(1/4,3,4)$ the point $(1/2,1/2,z)\in(\g_p)_0$ has two preimages in $(\g_p)_1$, corresponding to the sides of the ``flat tube" $E'$ (i.e. the half-spaces $y>1/2$ and $y<1/2$), or to the two ways of converging to $(1/2,1/2,y)$ in $W_0\setminus (E'\cup E'')$ while still staying in the region $x>1/2$. Therefore, $\omega_{01}{\big|}_{(\g_p)_1}:(\g_p)_1\to(\g_p)_0$,  is $1$-to-$1$ on the points corresponding to
$$(\g_p)_1\cap\omega_{01}^{-1}(\{(1/2,1/2,z):  z\in[0,1/4]\cup[3/4,1]\})$$
and is $2$-to-$1$ and onto on
$$(\g_p)_1\cap\omega_{01}^{-1}(\{(1/2,1/2,z):  z\in(1/4,3/4)\}).$$ Therefore $(\g_p)_1$ is homeomorphic to $L_1$, while $\omega^{-1}_{01}((\g_p)_0)$ is homeomorphic to $L_1\sqcup L_1$.\\

\textit{\underline{Case (ii)}}. If $\omega(p)=\left(\frac{1}{2},\frac{1}{4},0\right)$ (the same proof will work for $\omega(p)=\left(\frac{1}{2},\frac{3}{4},0\right)$), then the vertical line through $\omega(p)$ does not intersect $E'$, while the intersection with $E''$ is the segment
$$(\g_p)_0\cap E''=\left\{\left(\frac{1}{2},\frac{1}{4},z\right): \frac{1}{4}\leq z \leq \frac{3}{4}\right\}.$$
Just like in Case $(i)$, $\omega_{01}{\big|}_{(\g_p)_1}:\overline{W_1}\to\overline{W}_0$ is $1$-to-$1$ on  $$(\g_p)_1\cap\omega_{01}^{-1}(\{(1/2,1/4,z):  z\in[0,1/4]\cup[3/4,1]\})$$ and is $2$-to-$1$ and onto on
$$(\g_p)_1\cap\omega_{01}^{-1}(\{(1/2,1/4,z):  z\in(1/4,3/4)\}).$$ Therefore $(\g_p)_1$ is homeomorphic to $L_1$. Note that in this case $\omega^{-1}_{01}((\g_p)_0)=(\g_p)_1=L_1$.\\

\underline{\textit{Case (iii)}}. Suppose $\omega(p)=\left(\frac{1}{2},y,0\right)$, with $y\notin\{1/4,1/2,3/4\}$. Equivalently, $\omega(p)$ is not the midpoint, the top or the bottom endpoint of the corresponding slit in $[0,1]^2\times\{0\}$, then the  vertical line through $\omega(p)$ intersects $E''$ but not $E'$ in $W_0$. For every point $(1/2,y,z)$, with $0\leq z\leq 1$, there are two distinct points in $\overline{W}_1$, corresponding to the two sides (i.e. $x>1/2$ and $x<1/2$) of the ``flat cross" $E''$. Moreover, $\omega_{01}$ is continuous, and $1$-to-$1$ near these preimages and therefore  $\omega_{01}^{-1}(\{(1/2,y,z):  z\in[0,1]\})$ is a union of two copies of $[0,1]$, with one of the components being $(\g_p)_1$. Thus, $(\g_p)_1$ is homeomorphic to $[0,1]$.
\end{proof}

From the proof above we obtain an algorithm for inductively constructing the sequence $(\g_p)_i, i=0,1,\ldots$,  provided $\omega(p)$ belongs to a slit of the base slit carpet $\mathcal{B}$. In fact, the sets $(\g_p)_{i}$ are graphs which are constructed inductively starting from $(\g_p)_0=[0,1]$ as follows. For $i\geq0$, $(\g_p)_{i+1}$ is obtained from $(\g_{p})_i$ by dividing every edge of $(\g_p)_i$ into $4$ equal parts (equivalently, by replacing it with a linear graph with $5$ vertices and $4$ edges of length $1/4^{i+1}$), and by replacing each edge either by a (scaled) copy of $L_1$ (with edges of length $4^{-(i+1)}$) or by a scaled copy of $[0,1]$ (i.e. not changing the topology of $(\g_p)_i$). This leads to a topological characterization of the fibers $\g_p$. In fact, one may obtain a topological model of  $\g_p$ for every $p\in\mathcal{B}$, but we will only need the points $p$ which belong to a slit of $\mathcal{B}$.

\begin{lemma}[Topology of fibres in $\mathfrak{M}$]\label{lemma:menger-fibres}
Suppose $p\in\mathcal{B}$ belongs to a slit $s\subset\mathcal{B}$ of generation $i\geq1$. Then
\begin{align}\label{fibers-over-slits}
  (\g_p) \cong
  \begin{cases}
    L_i, &\mbox{ if } \omega(p)\in\{s^{+},s^{-}\},\\
    L_j, &\mbox{ if } \exists j\geq i \mbox{ s.t. } \omega(\g_p)\cap E_j'\neq \emptyset \\
    [0,1],  &\mbox{otherwise.}
  \end{cases}
\end{align}
Where $s^{-}$ and $s^{+}$ are the bottom and top endpoints of the slit $s=\{x\}\times[c,d]$ in $[0,1]^2\subset\mathbb{R}^2$, i.e. $s^{-} = (x,c)$ and $ s^{+} = (x,d)$.
\end{lemma}

\begin{figure}[htb]
\centering
\psfig{figure=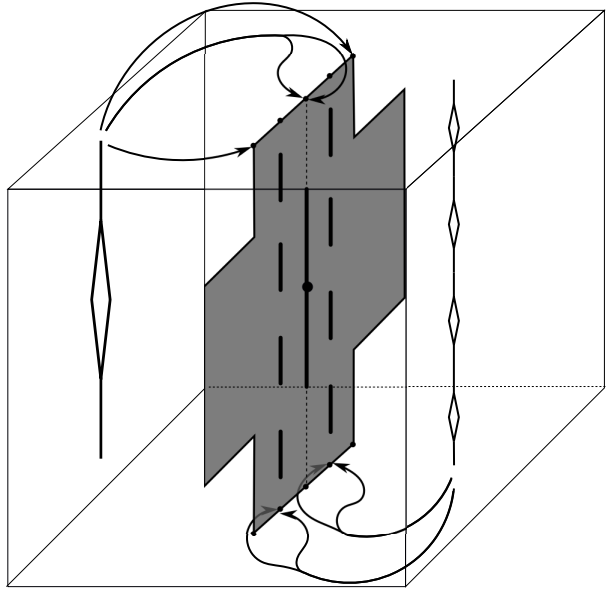,width=8cm}
\caption{\label{fig:fibers-in-Menger}\small{ As  $p$ runs along the largest slit $s$ in $\mathcal{B}$ corresponding to the interval $\{1/2\} \times[1/4,3/4]$ the fiber $\g_p$ takes on infinitely many topologically distinct types. To understand the topology of these fibers one needs to consider only those components of $E_i'$ which intersect $E_1''$ above the slit. Thus, the fibers over the endpoints $(1/2,1/4)$ and $(1/2,3/4)$ are connected and homeomorphic to the graph $L_1$. If $y\in[1/4,3/4]$ is not a diadic rational number then the fibers over $(1/2,y)$ in $\mathfrak{M}$ are homeomorphic to the interval $[0,1]$. For all the other points of the form $(1/2,y)$ in the slit $s$ there are two corresponding points in the base carpet $\mathcal{B}$ and thus two fibers $\g_{y}^{-}$ and $\g_{y}^{+}$ which are homeomorphic to each other. If $y$ is a diadic point of generation $k$, i.e. $y=\frac{l}{2^k}$ in the reduced form, then the two fibers over $y$ are homeomorphic to $L_k$. Thus, the two fibers over $(1/2,1/2)$ are homeomorphic to $L_1$, the four fibers over $(1/2,7/16)$ and $(1/2,9/16)$ are homeomorphic to $L_2$, etc.}}
\end{figure}

Note that if $p$ belongs to a generation $i$ slit then $\omega(\g_p)\cap E_j' = \emptyset$ for $j<i$. Thus, in Lemma \ref{lemma:menger-fibres} we may assume $j\geq i$. Moreover, if $p$ is the center of the slit $s$, denoted by $s^0$, then $\omega(\g_p)\cap E_i \neq \emptyset$ and therefore $\g_p\cong L_i$. Since there are exactly two point in $\mathcal{B}$ projecting to the center of $s$ we obtain the following corollary of Lemma \ref{lemma:menger-fibres}.

\begin{corollary}[Topology of fibres in $D\mathfrak{M}$.]\label{lemma:double-menger-fibres}
Suppose the point $p\in \mathcal{B} \subset D\mathfrak{M}$ belongs to a slit of $\mathcal{B}$ of generation $i\geq1$ (i.e. of diameter $2^{-1}\cdot 4^{-i}$). Then the fiber over $p$ in $D\mathfrak{M}$, denoted by $\hat{\g_p}$, is either a topological circle or is homeomorphic to $Y_j$ for some $j\geq i$.
Moreover, for every slit $s$ of $\mathcal{B}$ of generation $i$ there are exactly four points $p_l, l=1,\ldots,4$, corresponding to $s^{0},s^{-},s^{+}$, such that $\hat{\g}_{p_l}$ is homeomorphic to $Y_i$.
\end{corollary}

\begin{proof}
Since $\hat{\g}_p\cong(\g_p\sqcup \g_p)/\{0,1\}$, from Lemma \ref{lemma:menger-fibres} we have that either
\begin{align*}
\hat{\g}_p&\cong ([0,1]\sqcup [0,1])/\{0,1\} \cong \mathbb{S}^1, \mbox{ or }\\
\hat{\g}_p&\cong (L_j\sqcup L_j)/\{0,1\} = Y_j, \mbox{ for some } j\geq i.
\end{align*}

Moreover, by the discussion above, we have $\g_p$ is homeomorphic to $L_i$ if and only if $\omega(p)\in\{s^0,s^{-},s^{+}\}$. The endpoints $s^-$ and $s^{+}$ correspond to unique points in $\mathcal{B}$, while there are two distinct points in $\mathcal{B}$ corresponding to $s^{0}$. Thus, there are four points in $\mathcal{B}$ such that $\hat{\g}_p\cong Y_i$.
\end{proof}

\begin{proof}[Proof of Lemma \ref{lemma:menger-fibres}]
Let $p\in\mathcal{B}$ and $(x,y,0)=\omega(p)\in\overline{W}_0$. Since $p$ belongs to a slit of $\mathcal{B}$ we have that $\omega(p)\in s^i_{kl}$ for some $i\geq0$ and $k,l\in\{0,\ldots,4^i-1\}$, see (\ref{def:slits}). In particular,
$$x=\frac{2k+1}{2\cdot 4^i} \mbox{ for some } k\in\{0,\ldots,4^i-1\}.$$
For simplicity we let $s:=s^i_{kl}$. We also let $s^0$, $s^+$ and $s^-$ denote the midpoint, top and bottom end points of $s$, respectively, thinking of $s$ as a subset of $\mathbb{R}^2$.
Recall, that $\omega(\g_p)$ is the vertical interval through $(x,y,0)$. Since $s$ is a slit of generation $i$, from the construction for the sets $E_i'$ and $E_i''$, it follows that $ \omega(\g_p)\cap E''_j=\emptyset$ for $j\neq i$, and moreover
\begin{align}\label{intersecting-flat-tubes}
\omega(\g_p)\cap E_j'=\emptyset \mbox{ for }  j<i.
\end{align}
Therefore,
\begin{align}\label{intersecting-flat-tubes-and-crosses}
  \omega(\g_p)\cap \left[\bigcup_{j<i} E_j' \cup E_j''\right] = \emptyset.
\end{align}
It follows that $\omega_{0,j}$ is $1$-to-$1$ on the preimage of $\g_0$ in $\overline{W}_j$, and $(\g_p)_j$ is homeomorphic to the interval $[0,1]$ for every $j<i$.

For every diadic cube $Q\in\D_{2i}$ which intersects $\omega(\g_p)$ we have $(\g_p)_{i-1}\cap Q \simeq [0,1]$. Therefore we may apply the proof of (\ref{gen-1-fibres}) from Lemma \ref{lemma-fibers-0} and conclude that for a cube $Q\in\Delta_{2i}$ either $(\g_p)_i\cap Q =\emptyset$ or
\begin{align*}
  (\g_p)_i\cap Q
  \simeq
  \begin{cases}
    [0,1] &\mbox{ if } p\notin\{s^0,s^+,s^-\} \\
    L_1 &\mbox{ if } p\in \{s^0,s^+,s^-\}.
  \end{cases}
\end{align*}
Since $(\g_p)_i$ is obtained by consecutively attaching $4^i$ copies (one for every $Q\in\D_{2i}$ intersecting $(\g_p)_{i-1}$) of either $[0,1]$ or $L_1$ one after the other, it follows that
\begin{align}\label{fiber-classification}
  (\g_p)_i
  \simeq
  \begin{cases}
    [0,1] &\mbox{ if } p\notin\{s^0,s^+,s^-\} \\
    L_i &\mbox{ if } p\in \{s^0,s^+,s^-\}.
  \end{cases}
\end{align}

To prove (\ref{fibers-over-slits}) we consider three cases.\\

\textit{\underline{Case (i)}}. Suppose $\omega(p)\in\{s^+,s^-\}$. Then $\omega(\g_p)\cap E_j'=\emptyset$ for $j>i$ and therefore by (\ref{fiber-classification}) we have
$$(\g_p)_j\simeq(\g_p)_i\simeq L_i, \mbox{ for every } j>i.$$
Since $\g_p$ is the inverse limit of the sequence $(\g_p)_i$ it follows that
$\g_p\simeq L_i \mbox{ if } p\in\{s^0,s^+,s^-\}.$\\

\textit{\underline{Case (ii)}}.
If $\omega(\g_p)\cap  E'_{m}\neq\emptyset$ then $(\g_p)_j\simeq [0,1]$ and $\omega(\g_p)\cap  E_i =\emptyset$ for every $n<j<m$. Therefore,  $(\g_p)_{m-1}$ is homeomorphic to $[0,1]$. In particular for $Q\in\Delta_{2m}$ the set $(\g_p)_{m-1}\cap Q$ is either empty or is homeomorphic to $[0,1]$.

By the proof of Lemma \ref{lemma-fibers-0} again we have that the fiber $(\g_p)_{m}$ is obtained from $(\g_p)_{m-1}$ by replacing every nonempty intersection $(\g_p)_{m-1}\cap Q$, with a copy of $L_1$. Therefore $(\g_p)_{m}$ is homeomorphic to $L_{m}$. Since for $j>m$ we have that $\omega(\g_p)\cap E_j'=\emptyset$ it follows that $(\g_p)_j\simeq(\g_{m})$ and therefore
\begin{align}
  \g_p \simeq L_{m}.
\end{align}

\textit{\underline{Case (iii)}}. Suppose $\omega(p)$ is not an endpoint of $s$ and $\omega(\g_p)\cap  E_j' =\emptyset, \,\, \forall j\geq0.$ By (\ref{fiber-classification}) we have $(\g_p)_n\cong [0,1]$. Moreover, since $\omega(\g_p)\cap  E_j' =\emptyset, \,\, \forall j\geq0$, it follows that $\omega_{jm}:\overline{W}_j\to\overline{W}_i$ is injective on $(\g_p)_j$ and therefore $(\g_p)_j\simeq(\g_p)_m \simeq [0,1]$ for all $j>n$. Thus $\g_p\simeq[0,1]$.
\end{proof}

\section{QS maps of $D\mathfrak{M}$ are fiber-preserving}\label{Section:QS-maps-fiber-preserving}

In this section we show that a QS mapping of $D\mathfrak{M}$ into itself maps fibers onto fibers, where fibers are understood like in Section  \ref{Sec:menger-fibered-over-slit-carpet}. This will imply that a QS mapping of $D\mathfrak{M}$ induces a mapping $f_{\mathcal{B}}$ of the base slit carpet $\mathcal{B}$ into itself, thus reducing the question of surjectivity of $f$ to the same question for $f_{\mathcal{B}}$.

Recall that a subset $\g\subset\mathfrak{M}$ is $z$-parallel (vertical) if $\omega(\g)\subset[0,1]^3$ is a subset of a line which is parallel
to the $z$-axes.

\begin{lemma}\label{lemma:menger-non-z-parallel}
  Let $\G_{nz}$ be the collection of non $z$-parallel curves in $(\mathfrak{M},d_{\mathfrak{M}},\calH^3)$ or $(D\mathfrak{M},d_{D\mathfrak{M}},\calH^3)$. Then
  $\m_p \G_{nz} = 0$
  for every $p\geq 1$.
\end{lemma}
\begin{proof}

Let
\begin{align*}
\begin{split}
  F' &= \left\{(x,y,z)\in[0,1]^3 : \max\left(|z-\frac{1}{2}|,|y-\frac{1}{2}|\right)\leq \frac{1}{2^{2}} \right\} \cap \pi_2^{-1}(p) \subset E',\\
  F'' &= \left\{(x,y,z)\in[0,1]^3 : \max\left(|x-\frac{1}{2}|,|y-\frac{1}{2}|\right)\leq \frac{1}{2^{2}}\right\} \cap \pi_1^{-1}(p) \subset E''.
\end{split}
\end{align*}
Thus $F'$ and $F''$ are the largest squares in $E'$ and $E''$, respectively, centered at $p$.
Next we consider the families of slits in $[0,1]^3$,
\begin{align*}
  \mathcal{F}' &= \{T_{Q}(F')\}_{Q\in\D_{2i}, i\geq 0},\\
  \mathcal{F}'' &= \{T_{Q}(F'')\}_{Q\in\D_{2i}, i\geq 0}.
\end{align*}
Clearly the families of slits $\mathcal{F}'$ and $\mathcal{F''}$ are uniformly relatively separated and occur in all locations and scales. By Lemma \ref{lemma:curves_in_porous_carpets} we have that
$\m_p \G_{\mathcal{F}''}=0.$ Considering slits in planes perpendicular to the $y$-axis, (an analogue of) Lemma \ref{lemma:curves_in_porous_carpets} would also show that $\m_p \G_{\mathcal{F}'}=0$. Moreover, the proof of the modulus estimate in Section \ref{Section:main-estimate-proof} shows that if 
for $i\in\{1,2\}$ we define 
\begin{align*}
\G_i:= \{ \g \in \G : |\pi_i(\omega(\g)| >0 \}, 
\end{align*}
then $\m_p \omega(\G_1)=\m_p \omega(\G_2)=0$. Since $\omega$ is $1$-Lipschitz 
by (the proof of) Lemma \ref{lemma:modulus-under-projections} we have that 
$\m_p \G_1=\m_p \G_2=0$. From the definitions we have that $\G_{nz}\subset \G_{1}\cup \G_{2}$ and therefore $\m_p \G_{nz}=0$.
\end{proof}

%
%
%
%
%
%

The following result follows from the lemma above just like in the case of Lemma \ref{lemma:vert-to-vert}.

\begin{lemma}\label{lemma:menger-parallel-to-parallel}
Let $X$ be $\mathfrak{M}$, or $D\mathfrak{M}$. If $f$ is a quasisymmetric embedding of $X$ into itself, then it maps every  $z$-parallel simple closed curve in $X$ to a $z$-parallel simple closed curve.
\end{lemma}

\begin{proof}
Let $B'$ be the subset of $[0,1]^2\subset\mathbb{R}^2$ such that the $z$-parallel lines through the points in $B'$ do not intersect the set $\cup_{i=0}^{\infty} E_n \subset [0,1]^3$. Note that $B'$ is the collection of points $(x,y)\in[0,1]^2$ such that at least one of the coordinates $x$ or $y$ is not a dyadic rational. In particular $B'$ is of full Lebesgue measure in $[0,1]^2$.

Let $\G^3_{z\to nz}$ be the collection of $z$-parallel simple closed curves $\g$ in $D\mathfrak{M}$ such that the segment $\omega_0(\g)$ connects the top and bottom faces of the cube $[0,1]^3$, passes through a point in $B'$ and $f(\g)\notin\G_{nz}$. Since $\G^{3}_{z\to nz}\subset \G_{z\to nz}$, by Lemma \ref{lemma:menger-non-z-parallel} and by Tyson's theorem we have $\m_3 (\G^3_{z\to nz})=0$. On the other hand since $\G^{3}_{z\to nz}$ is the preimage of a product family under $\omega$, which is $1$-Lipschitz, we have that $\m_3(\G^3_{z\to nz}) \gtrsim \calH^{2}(\Theta(\G^{3}_{z\to nz}))$. Therefore, $\calH^{2}(\Theta(\G^{3}_{z\to nz}))=0$, which means that for $\calH^2$ almost every point $\alpha:=(x,y)\in[0,1]^2$ the curve $\Theta^{-1}(\alpha)$ is mapped by $f$ to a vertical curve in $\mathfrak{M}$. By Lemma \ref{lemma:menger-properties} it follows that such curves form a full measure set and therefore a dense subset of $\mathfrak{M}$.

Now, if $p$ and $q$ are two distinct points in $\mathfrak{M}$ belonging to the same fiber, say $\g$, then there is a sequence of fibers $\g_i$ and points $p_i,q_i\in\g_i$ in $\mathfrak{M}$ such that $p_i\to p$, $q_i\to q$ and such that the $f(\g_i)$ is a vertical set. Since $f$ and $\Theta$ are continuous it follows that $\Theta(f(p))=\Theta(f(q))$. Since $p$ and $q$ were arbitrarily two points in $\g$ it follows that $f(\g_p)$  is a $z$-parallel set in $\mathfrak{M}$ for every fiber $\g_p$.
\end{proof}

%
%

%

By Lemma \ref{lemma:menger-parallel-to-parallel} we have that a quasisymmetric mapping $f$ of $D\mathfrak{M}$ into itself maps fibers into fibers. The next result shows that $f$ is in fact surjective on each or these fibres.

\begin{lemma}[Surjectivity on fibres]\label{lemma:surject-on-fibers}
If $f$ is a quasisymmetric mapping of $D\mathfrak{M}$ into itself then for every $p\in D\mathfrak{M}$ the restriction of
$$f|_{\hat{\g}_p} : \hat{\g}_p \to \hat{\g}_{f(p)}$$ is surjective, i.e.
\begin{align}\label{equality:circles-to-circles}
  f(\hat{\g}_p) = \hat{\g}_{f(p)}.
\end{align}
\end{lemma}
\begin{proof}
  Let $B'\subset[0,1]^2$ be as in the proof of Lemma \ref{lemma:menger-parallel-to-parallel}. Note that if $\Theta(p)\in B'$ then $\hat{\g}_p\subset D\mathfrak{M}$ is a topological circle. Let
\begin{align*}
  \G'&=\{\hat{\g}_p\subset D\mathfrak{M} \, | \, \Theta(p)\in B' \},\\
  \G''&= \{\hat{\g}_p \in \G' \, | \, \hat{\g}_{f(p)} \mbox { is not a topological circle } \}.
\end{align*}

If $\hat{\g}_p\in\G''$ then
$$\Theta(\hat{\g}_{f(p)})\subset [0,1]^2\setminus B'.$$
Since $\calH^2([0,1]^2\setminus B')=0$ it follows that $\omega(f(\G''))$ is contained in a zero measure set $([0,1]^2\setminus B')\times[0,1]\subset \mathbb{R}^3$ and therefore the family $f(\G'')$ is a subset of zero $\mathcal{H}^3$-measure in $D\mathfrak{M}$. Moreover, there is a lower bound on the length of the curves $f(\g)$ for $\g\in\G''$, since $f$ is uniformly continuous. Therefore  $\m_3 f(\G'') =0$. By Tyson's theorem then also $\m_3 \G''=0$. In particular,  for $\m_3$-almost every $\hat{\g}_p\in\G'$ we have that $\hat{\g}_p\in\G'\setminus\G''$ or that $f(\hat{\g}_p)$ is a subset of a topological circle. Since $f|_{\hat{\g}_p}$ is continuous and one-to-one mapping of $\mathbb{S}^1$ into itself, $f$ is in fact onto, and we have that (\ref{equality:circles-to-circles}) holds for $\m_3$ almost every curve $\hat{\g}_p\in\G'$.

Thus, for almost every $p\in D\mathfrak{M}$ we have that the fiber $\hat{\g}_p$ is a topological circle and $f$ maps it onto the fiber $\hat{\g}_{f(p)}$ which is also a topological circle.

Now, if $p\in D\mathfrak{M}$ is arbitrary then by Lemma \ref{lemma:menger-parallel-to-parallel} we have that $f(\hat{\g}_p)\subset\hat{\g}_{f(p)}$. Pick any point $q'\in\hat{\g}_{f(p)}$. We want to show that there is a point $q\in \hat{\g}_p$ such that $f(q)=q'$. From the discussion above it follows that there is a sequence of topological circles $\hat{\g}_i'\subset f(D\mathfrak{M})$ such that $\Theta(\hat{\g}_i')\to\Theta(\hat{\g}_{f(p)})$ and $\hat{\g}_i:=f^{-1}(\hat{\g}_i')$ is a vertical topological circle. Next if $p_i',q_i'\in\g_i'$ are such that $q_i'\to q'$ and $p_i' \to f(p)$ then since $f^{-1}$ is continuous and $D\mathfrak{M}$ is compact the sequence $f^{-1}(q_i')$ converges to some point $q\in D\mathfrak{M}$ and $f(q)=\lim{f(f^{-1}(q_i'))} = q'$. Moreover, since $\hat{\g}_i$'s are vertical it follows that $q=f^{-1}(q')\in \hat{\g}_p$. Therefore $q'=f(q)\in f(\hat{\g}_{p})$. Since $q'\in\hat{\g}_{f(p)}$ was arbitrary, it follows that $\hat{\g}_{f(p)}\subset f(\hat{\g}_p)$.
\end{proof}

\section{$D\mathfrak{M}$ is QS co-Hopfian}\label{Section:menger-coHopf-proof}

From Lemma \ref{lemma:surject-on-fibers} is follows that every quasisymmetric mapping $f$ of $D\mathfrak{M}$ into itself induces a well defined mapping $f_{\mathcal{B}}$ of the base slit carpet $\mathcal{B}$ into itself as follows
\begin{align}
\begin{split}
  f_{\mathcal{B}}:=& \, \mathcal{B}\longrightarrow \,\, \mathcal{B} \\
  &\, p \, \,\, \mapsto  \,\, \Pi (f(p)).
  \end{split}
\end{align}
Where $\Pi:D\mathfrak{M}\to\mathcal{B}$ is defined as in (\ref{def:projection-to-base}).

\begin{lemma}\label{lemma:induced-map-surjective}
If $f$ is a quasisymmetric mapping of $D\mathfrak{M}$ into itself then the mapping $f_{\mathcal{B}}$ is a homeomorphism of $\mathcal{B}$ onto itself.
\end{lemma}
\begin{proof}
Since $f$ and $\Pi$ are continuous, it follows that $f_{\mathcal{B}} = \Pi\circ f|_{\mathcal{B}}$ is a continuous map as well. Moreover, if $p,q\in\mathcal{B}$ are distinct then $\g_p\cap\g_q =\emptyset$. Since $f$ is injective we have that $f(\g_p)=\g_{f(p)}$ and $f(\g_q)=\g_{f(q)}$ are disjoint subsets of $D\mathfrak{M}$. Therefore
$$f_{\mathcal{B}}(p) = \g_{f(p)} \cap \mathcal{B} \neq \g_{f(q)} \cap \mathcal{B}=f_{\mathcal{B}}(q),$$
or equivalently $f_{\mathcal{B}}$ is injective. Since $\mathcal{B}$ is compact it follows that $f$ is a homeomorphism onto its image $f_{\mathcal{B}}(\mathcal{B})$.

By Whyburn's theorem $f_{\mathcal{B}}$ maps peripheral circles or slits of $\mathcal{B}$ to peripheral circles. The slits of $\mathcal{B}$ are in one to one correspondence with the dyadic subsquares of $B=[0,1]^2$ of generations $2n$ with $n\geq 0$. In fact every dyadic subsquare $Q$ of generation $2n$ contains a slit of length $2^{-(2n+1)}=2^{-1}l(Q)$, where $l(Q)$ is the sidelength of $Q$. We will show that $f_{\mathcal{B}}$ permutes the slits of $\mathcal{B}$ of the same diameter and therefore is onto since the slits are dense in $\mathcal{B}$.



Suppose $s$ is a slit of $\mathcal{B}$ of diameter $4^{-n}/2$. By Lemma \ref{lemma:double-menger-fibres} there is a point $p\in s$ such that $\g_{p}$ is homeomorphic to $Y_n$. Then $\hat{\g}_{f(p)}=\hat{\g}_{f_{\mathcal{B}}(p)}$ is also homeomorphic to $Y_n$ by (\ref{equality:circles-to-circles}). Moreover, since $f_{\mathcal{B}}$ is a homeomorphism onto its image, it follows from Whyburn's theorem that $f_{\mathcal{B}}(p)$ belongs to a slit of $\mathcal{B}$ as well. However, if $f_{\mathcal{B}}(p)$ belongs to a slit of generation $m$, i.e. of diameter $4^{-m}/2$, then since $\g_p$ is not a topological circle, $\g_{f_{\mathcal{B}}}(p)$ is homeomorphic to $Y_j$ for some $j\geq m$. Therefore $m\leq n$, or equivalently $f_{\mathcal{B}}(p)$ belongs to a slit of diameter at least $2^{-1}\cdot 4^{-n}$. Thus every slit $s$ of $\mathcal{B}$ is mapped by $f_{\mathcal{B}}$ to a slit of diameter no less than $\diam(s)$. In particular, the largest slit is mapped to itself. Similarly every one of the $16$ slits of diameter $2^{-1}\cdot 4^{-1}$ is mapped to a slit of the same diameter and therefore the slits of generation $2$ are permuted. Continuing by induction we see that $f_{\mathcal{B}}$ permutes the slits of the same diameter. Since slits are dense in $\mathcal{B}$ it follows that $f_{\mathcal{B}}$ is surjective.
\end{proof}


\begin{proof}[Proof of Theorem \ref{thm:slit-menger-co-Hopf}]
Suppose $f:D\mathfrak{M}\to D\mathfrak{M}$ is a quasisymmetric mapping and $p'\in D\mathfrak{M}$. Let $q'=\Pi(p')\in\mathcal{B}$. Then by Lemma \ref{lemma:induced-map-surjective} there is a point $q\in\mathcal{B}$ such that $f_{\mathcal{B}}(q)=\Pi(f(q))=q'$. By Lemma \ref{lemma:surject-on-fibers} we have that $f(\hat{\g}_q)=\hat{\g}_{f(q)}=\hat{\g}_{q'}\ni p'$. Thus there exists $p\in\hat{\g}_q \subset D\mathfrak{M}$ such that $f(p)=p'$.
\end{proof}

\section{Quasisymmetric equivalence classes of Menger curves.}\label{Section:uncountable-menger}

We will show that for every infinite subset $A\subset \mathbb{N}=\{0,1,2,\ldots\}$ we can construct a corresponding slit Menger curve $\mathfrak{M}(A)$ and its double $D\mathfrak{M}(A)$ so that if $A \neq B\subset \mathbb{N}$ then $D\mathfrak{M}(A)$ and $D\mathfrak{M}(B)$ are not quasisymmetrically equivalent. Since there are uncountably many infinite subsets of $\mathbb{N}$ this would imply the theorem.

Suppose $A\subset\mathbb{N}$ is an infinite subset of natural numbers. Consider the decreasing sequence of domains
$$W_{k}(A)  = (0,1)^3 \setminus \bigcup_{\substack{0\leq j\leq k \\ j\in A}} E_{j}.$$
Then, like in Section \ref{section:slit-menger-definition}, we can define the metric measure space $\mathfrak{M}(A)$ as the inverse limit of the sequence of metric measure spaces $(\overline{W_k(A)},d_{W_k(A)})$, where the closures are taken with respect to the corresponding inner metrics. The space $D\mathfrak{M}(A)$ is then defined as the double of $\mathfrak{M}(A)$ along its top $\mathcal{T}(A)$ and bottom $\mathcal{B}(A)$ carpets. It is easy to see that if $A$ is an infinite set then $\mathfrak{M}(A)$ is homeomorphic to the Menger curve. Indeed, all the properties in Anderson's theorem, except for being of topological dimension $1$, hold true for any choice of $A$ the same way as for $\mathfrak{M}$. To show that $\mathfrak{M}(A)$ is $1$-dimensional when $A$ is infinite note that if $i\in A$ then the coverings $\tilde{\mathcal{Q}}_i$ considered in Lemma \ref{lemma:slit-menger-dimension1} are also covering for $\mathfrak{M}(A)$ of order $1$ and the diameter of each element in $\tilde{\mathcal{Q}}_i$ is comparable to $4^{-i}$. Since $A$ is infinite, we can take $i$ to be arbitrarily large and therefore the diameter of the elements in the coverings arbitrarily small. 

Below we will need the following result.
\begin{lemma}\label{lemma:moduli-menger-uncountable}
Suppose $A\subset\mathbb{N}$ is an infinite subset of integers and $\G_{nz}(A)$ is the collection of non $z$-parallel curves in $\mathfrak{M}(A)$. Then $\m_3 \G_{nz}(A)=0$.
\end{lemma}
\begin{proof}The proof is similar to Lemma \ref{lemma:non-vertical0}
Consider the families of slits in $[0,1]^3$,
\begin{align*}
  \mathcal{F}'(A) &= \{T_{Q}(F')\}_{Q\in\D_{2i}, i\in A},\\
  \mathcal{F}''(A) &= \{T_{Q}(F'')\}_{Q\in\D_{2i}, i\in A}.
\end{align*}
Note, that these are examples of standard non-self-similar collections of slits in $\mathbb{R}^3$ considered in Lemma \ref{lemma:curves-in-diadic-carpets}. From the definition it follows that $\G_{\mathcal{F}'(A)}=\G_{\mathbf{r}}$, where $\G_{\mathbf{r}}$ is defined as in Section \ref{Section:non-self-similar}, for a sequence $\mathbf{r}=\{r_j\}_{j=0}^{\infty}$, where
\begin{align*}
  r_j = 
  \begin{cases}
  1/2 & \mbox{ if } j\in 2A\\
  0 &\mbox{ otherwise.}  
  \end{cases}
\end{align*}
Therefore $\sum_{j=0}^{\infty}{r_j}^3=\infty$ if $A$ is infinite. By Lemma \ref{lemma:curves-in-diadic-carpets} we have that $\m_3\G_{\mathcal{F}'(A)}=0$. Similarly, $\m_3\G_{\mathcal{F}''(A)}=0$. By the same argument as in Lemma \ref{lemma:menger-non-z-parallel} we obtain that $\m_3{\G}_{nz}=0.$
\end{proof}

Finally, Theorem \ref{thm:uncountable-menger} follows from the following.

\begin{lemma}
Suppose $A,B$ are infinite subsets of the natural numbers. If $\mathfrak{M}(A)$ is quasisymmetrically equivalent to $\mathfrak{M}(B)$ then $A=B$.
\end{lemma}

\begin{proof} Just like $\mathfrak{M}$ and $D\mathfrak{M}$ the spaces $\mathfrak{M}(A)$ and $D\mathfrak{M}(A)$ are fibered spaces over a base slit Sierpi\'nski carpet, which we may denote by $\mathcal{B}(A)$. In fact the fibers over the points of the slit carpet $\mathcal{B}(A)$ can be characterized as in Corollary \ref{lemma:double-menger-fibres}. Namely, if a point $p$ belongs to a slit of $\mathcal{B}(A)$ of diameter $2\cdot 4^{-i}$ for some $i\geq 0$ then $\hat{\g}_p$ is either a topological circle or homeomorphic to $Y_j$ for some $j\geq i$.

Moreover, there are exactly $4$ fibers homeomorphic to $Y_i$. The difference of $D\mathfrak{M}(A)$ from the case considered in Corollary \ref{lemma:double-menger-fibres} is that not all natural numbers $i\geq 0$ occur.

Now suppose $f$ is a quasisymmetric mapping of $\mathfrak{M}(A)$ to $\mathfrak{M}(B)$, where $A$ and $B$ are subsets of $\mathbb{N}$. Using Tyson's theorem, Ahlfors regularty of $f(D\mathfrak{M}(A))$ and Lemma \ref{lemma:moduli-menger-uncountable}  just like in Lemmas \ref{lemma:menger-parallel-to-parallel} and \ref{lemma:surject-on-fibers} it follows that for every $p\in\mathcal{B}(A)$ the mapping $f$ takes the fiber $\hat{\g}_p\subset\mathfrak{M}(A)$ onto a fiber $\hat{\delta}_q \subset\mathfrak{M}(B)$ over a point $q\in\mathcal{B}(B)$. Therefore, following the proof of Lemma \ref{lemma:induced-map-surjective}, there is an induced continuous and one-to-one mapping $f_{\mathcal{B}}$ of the slit carpet $\mathcal{B}(A)$ into $\mathcal{B}(B)$, which in particular maps peripheral circles to peripheral circles, i.e. slits to slits.

Now, suppose $A=\{a_0,a_1,\ldots\}, B=\{b_0,b_1,\ldots\},$
where $\{a_i\}$ and $\{b_i\}$ are increasing sequences. We say below that a slit in $\mathcal{B}(A)$ is of generation $i\geq0$ if it is of  diameter $4^{-a_i}/2$.


As explained above, for every slit $s\subset\mathcal{B}(A)$ of generation $0$ there is a point  $p\in s$ such that $\hat{\g}_p \cong Y_{a_0}$. Since $f$ maps fibers to fibers, there is a point $q\in\mathcal{B}(B)$ such that the fiber over $q$ is homeomorphic to $Y_{a_0}$. Which means that $\min B=b_0\leq a_0$, since otherwise all the non-circle fibers in $\mathfrak{M}(B)$ would have been homeomorphic to $Y_i$ for some $i>a_0$.  By considering $f^{-1}$ we conclude that also $a_0\leq b_0$. Thus $a_0=b_0$. In particular, $f_{\mathcal{B}}$ gives a one-to-one correspondence between the the slits of generation $0$ of $\mathcal{B}(A)$ and $\mathcal{B}(B)$.

By induction, assume that $a_j=b_j$ for $0\leq j \leq k$ and that $f_{\mathcal{B}}$ produces a one-to-one correspondence between the slits of these carpets of generations $j\in\{0,\ldots,k\}$. Thus, a slit $s\subset \mathcal{B}(A)$ of generation $(k+1)$  cannot be mapped to a slit of smaller generation in $\mathcal{B}(B)$. On the other hand, $s$ cannot be mapped to a slit of diameter $4^{-i}/2$ with $i>a_{k+1}$ either, since then the points in the image slit $f_{\mathcal{B}}(s)$ would not have fibers homeomorphic to $Y_{a_{k+1}}$. This means that $b_{k+1} \leq a_{k+1}$. Considering $f^{-1}$ we can conclude again that $a_{k+1}=b_{k+1}$. Moreover there is a one-to-one correspondence between the slits of generations $\leq (k+1)$ of $\mathcal{B}(A)$ and $\mathcal{B}(B)$.

Thus, we obtain that $a_k=b_k$ for every $k\geq 0$, or that $A=B$.
\end{proof}

\section{Remarks and Open Problems}\label{section:remarks-problems}

\subsection{Quasisymmetric embeddings of slit spaces} In \cite{MerWildrick} Merenkov and Wildrick showed that the slit carpet considered in \cite{Mer:coHopf} cannot be embedded quasisymmetrically into $\mathbb{R}^2$. The idea behind the proof was that if such an embedding existed then the image of the slit carpet would have to be a porous subset of $\mathbb{R}^2$, since porosity is a quasisymmetrically invariant property for carpets, and therefore would have to have Hausdorff dimension strictly less than $2$. However this would contradict the fact that conformal dimension of the slit carpet is $2$, i.e. every quasisymmetric image of the carpet has Hausdorff dimension at least $2$, which also follows from Lemma \ref{lemma:lower-Ahlfors-regular}.

The argument above would not work if the slit carpet is not porous. This suggests the following question. 

\begin{question}
Is there a non-porous slit carpet which does not embed quasisymmetrically in $\mathbb{R}^2$? More generally, is it possible to characterize slit carpets which can be quasisymmetrically embedded in $\mathbb{R}^2$?
\end{question}

\subsection{Slit Sierpi\'nski spaces admitting Poincare inequalities}
The summability conditions in  Theorem \ref{thm:main1} is analogous to one appearing in the work of Mackay, Tyson and Wildrick \cite{MTW} on non-self-similar square carpets supporting Poincar\'e inequalities.

Recall that a metric measure space $(X,d,\mu)$ is said to support a \emph{$p$
- Poincar\'e inequality}, $p\geq 1$, (or is a PI space) if there are constants $C,\la\geq 1$
such that for every continuous function $u:X\to\mathbb{R}$ and every ball $B\subset X$ the following inequality holds
\begin{equation}
  \dashint_B\left|u -\dashint u\, d\mu\right| d\mu \leq C \diam(B) \left(\dashint_{\la B} \rho^p \, d\mu\right)^{1/p},
\end{equation}
whenever $\rho:X\to[0,\infty]$ is an upper gradient of $u$, i.e. if
$$|u(x)-u(y)| \leq \int_{\g} \rho ds,$$
for every rectifiable curve $\g$ connecting $x$ and $y$. Here we denoted by
$\dashint_B u \, d\mu = \frac{1}{\mu(B)} \int_B u \, d\mu$. See
\cite{Hein:book,HKActa} for further details on PI spaces.

The carpets considered in \cite{MTW} are constructed from the unit square by
dividing it into smaller congruent subsquares, removing the middle one and
repeating the procedure with the remaining ones. Thus, similarly to the diadic slit
carpets, every \textit{non-self-similar square carpet} is associated to a
sequence $\mathbf{a}=(a_1,a_2,\ldots)$, where each $a_i$ is equal to the
ratio of the side-length of the removed square to that of the square from
which it is being removed at step $i\geq 1$. We will denote such a carpet by
$S_{\mathbf{a}}$. The restriction of the Euclidean metric to $S_{\mathbf{a}}$
denoted by $d$, while $\mu$ is the (multiple of) Hausdorff measure in the
dimension of $S_{\mathbf{a}}$. It is not hard to see that $S_{\mathbf{a}}$
has positive area if and only if $\mathbf{a}\in\ell^2$. The following result
is from \cite{MTW}.

\begin{theorem}[Mackay,Tyson,Wildrick]\label{thm:MTW}
The following conditions are equivalent
\begin{itemize}
  \item [$(a)$] $\mathbf{a}\in\ell^2$
  \item [$(b)$] $(S_{\mathbf{a}},d,\mu)$ supports a $p$-Poincar\'e
      inequality for all $p>1$.
  \item [$(c)$] $(S_{\mathbf{a}},d,\mu)$ supports a $p$-Poincar\'e
      inequality for some $p>1$.
\end{itemize}
\end{theorem}

For slit carpets we prove one of the analogous implications. Indeed, Lemma \ref{lemma:curves_in_porous_carpets} implies that if $\textbf{r}\notin \ell^2$ then the family of all ``non-vertical" curves in the slit carpet $S(\textbf{r})$ has zero $2$ modulus. This gives the following result.

\begin{theorem}\label{thm:no-PI}
  If $\mathbf{r}\notin\ell^2$ then the slit carpet $(S({\mathbf{r}}),d_{\calS_{\mathbf{r}}},\calH^2)$ does not support a $p$-Poincar\'e inequality for any $p\geq 1$.
\end{theorem}

It would be interesting to know if the full analogue of Theorem \ref{thm:MTW} holds for non-self-similar slit carpets.


\section{Aknowledgements} The author would like to thank Mario Bonk and Sergei Merenkov for drawing his attention to the problem of quasisymmetric co-Hopficity and for several helpful conversations at the early stages of this work.


\bibliographystyle{alpha}

\end{document}